\documentclass[11pt]{article}
\usepackage[utf8]{inputenc}
\usepackage{amsthm}
\usepackage{amsmath}
\usepackage{amssymb}
\usepackage{mathtools}
\usepackage{enumerate}
\usepackage{cancel}
\usepackage{enumerate}
\usepackage{mathrsfs}
\usepackage{tikz}
\usepackage{graphicx}
\usepackage{gensymb}
\usetikzlibrary[shapes.geometric]
\usepackage{tabularx}
\usepackage{hyperref}
\usepackage{setspace}
\usepackage[top=1in, left=1in, right=1in, bottom=1in, footskip=0.5in]{geometry}
\usepackage{graphicx}
\usepackage{scrextend}
\usepackage{lipsum}
\usepackage{caption}
\usepackage{comment}
\usepackage{fancyhdr}
\usepackage{float}
\usepackage{optidef}
\fancypagestyle{plain}{
\fancyhf{}
\fancyfoot[C]{\large\thepage}

}
\usepackage{xcolor}
\usepackage{pgffor}
\usepackage{graphicx}
\usepackage{graphbox}
\usepackage{soul}
\usepackage{color}

\DeclareMathOperator{\inte}{int}

\newtheorem{theorem}{Theorem}[section]
\newtheorem{lemma}[theorem]{Lemma}
\newtheorem{corollary}[theorem]{Corollary}

\newtheorem{case}{Case}

\newtheorem*{claim*}{Claim}

\theoremstyle{definition}
\newtheorem{definition}[theorem]{Definition}

\numberwithin{table}{section}   
\numberwithin{equation}{section}

\newcommand{\abs}[1]{\left\vert#1\right\vert}

\newcommand{\bit}{\begin{itemize}}
\newcommand{\eit}{\end{itemize}}
\newcommand{\ben}{\begin{enumerate}}
\newcommand{\een}{\end{enumerate}}
\newcommand{\beq}{\begin{equation}}
\newcommand{\eeq}{\end{equation}}
\newcommand{\bea}{\begin{eqnarray*}}
\newcommand{\eea}{\end{eqnarray*}}
\newcommand{\bean}{\begin{eqnarray}}
\newcommand{\eean}{\end{eqnarray}}
\newcommand{\bpf}{\begin{proof}}
\newcommand{\epf}{\end{proof}\ms}
\newcommand{\bmt}{\begin{bmatrix}}
\newcommand{\emt}{\end{bmatrix}}
\newcommand{\ms}{\medskip}

\newcommand{\beqa}{\begin{array}}
\newcommand{\eeqa}{\end{array}}

\title{Landscapes of the Tetrahedron and Cube: An Exploration of Shortest Paths on Polyhedra}
\author{McKenzie Fontenot\thanks{Department of Mathematics, University of North Texas, Denton, TX 76205 (kenzie.fontenot@unt.edu).}\and Erin Raign\thanks{Department of Mathematics, University of North Texas, Denton, TX 76205
(erinraign@my.unt.edu).}\and August Sangalli \thanks{Department of Mathematics, University of Denver, Denver, CO 80210 (gus.sangalli@du.edu).} \and Emiko Saso  \thanks{Department of Mathematics, Trinity College, Hartford, CT 06106 (emiko.saso@trincoll.edu).} \and Houston Schuerger \thanks{Department of Mathematics, Trinity College, Hartford, CT 06106 (houston.schuerger@trincoll.edu).} \and Xin Shi \thanks{Department of Mathematics, Trinity College, Hartford, CT 06106 (xin.shi@trincoll.edu).} \and Ethan Striff-Cave \thanks{Department of Mathematics, Williams College, Williamstown, MA 01267 (ecs11@williams.edu)}}

\begin{document}

\maketitle

\begin{abstract}
We consider the problem of determining the length of the shortest paths between points on the surfaces of tetrahedra and cubes.  Our approach parallels the concept of Alexandrov's star unfolding but focuses on specific polyhedra and uses their symmetries to develop coordinate based formulae.  We do so by defining a coordinate system on the surfaces of these polyhedra. Subsequently, we identify relevant regions within each polyhedron's nets and develop formulae which take as inputs the coordinates of the points and produce as an output the distance between the two points on the polyhedron being discussed.
\end{abstract}

\section {Introduction}
There is a rich history of geodesics on polyhedra stemming from the well-known fact that the shortest path between two points on the surface of a convex polyhedron restricted to the polyhedron is contained as a straight line segment in one of the polyhedron's nets. Building on this fact, we also have Alexandrov's star unfolding which provides a process for determining the shortest distance between any two points on the surface of a convex polyhedron  \cite{Alexandrov}. It does so by fixing an initial point and constructing a figure which identifies shortest paths between said source point and the remaining points of the polyhedron. However, this method of fixing an initial point and constructing the star unfolding is relatively laborious for the end user, since the process must be repeated for each new choice of the source point. On the other hand, this paper provides coordinate system based formulae, thus placing the weight of the calculation on determining these formulae for the different polyhedra rather than in the final computation of the distance performed by the end user.

Dijkstra's Algorithm provides another method of calculating the shortest paths on polyhedral surfaces, thus doing so for a larger class of polyhedra \cite{Cook}. This process involves approximating polyhedral surfaces using meshes, and thus turns a problem of geometry into one of graph theory and algorithms. While calculations through this method can be quick with a computer, calculating by hand is very difficult and can often lead to a large margin of error. 

Problems involving the optimization of pathways are critical for a number of applications across different fields. The general problem stemming from optimizing collision-free paths for robots in the realm of numerical analysis is studied in \cite{Agarwal}. In graph theory, Dijkstra's algorithm is used to look at the shortest paths on polyhedral surfaces, noting applications in robotics, geographic information systems, and route finding in \cite{Kanai}. The same type of problem relating to robotics and motion planning is explored by using sequence trees in \cite{Chen}. An approach based in finding the shortest path on a polyhedral surface is used to model a network design problem with applications in telecommunications and transportation in \cite{Magnanti}. In this paper, rather than find an algorithm to be applied, we provide our polyhedra with a coordinate system so that exact formulae can be derived. As a result, once these formulae are obtained, numerical calculations can be easily performed by anyone wishing to apply these findings.

In order to develop these formulae, we will utilize nets to calculate the shortest distance between points along the surface of convex regular polyhedra. Specifically, we will restrict our view to the paths lying along the surfaces of cubes and tetrahedra. With this goal in mind, we will define new net substructures, and using them identify where these shortest paths can lie as well as the lengths of these paths (for any two given points).  

To ensure a clear discussion,  we now provide some basic definitions and common conventions.  A {\em net} is a 2-dimensional polygonal shape that can be folded along prescribed line segments interior to the polygon to form the surface of a 3-dimensional polyhedron. As a result of this folding, the collection of edges of the net (both those in the interior of the polygon and the line segments forming its exterior edges) become edges of the polyhedron, the vertices of the net (the endpoints of these line segments) become the vertices of the polyhedron, and the regions bounded by these edges and vertices (which are referred to as the faces of the net) become the faces of the polyhedron.  For a net $N$ of a polyhedron $\mathcal{P}$, by definition there exists a function $f:N \rightarrow \mathcal P$ induced by this folding. Since such a function identifies each point along an exterior edge of the net with at least one other distinct point, it follows that $f$ will not be one-to-one.  For these points $x \in \mathcal P$ such that $\abs{f^{-1}(x)}>1$, we can think of each point $a \in N$ with $a \in f^{-1}(x)$ as one of multiple copies of $x$. While it seems that these multiple copies of $x$ could cause confusion, the specific copy of $x$ given by the point $a$ currently being discussed will be implied by other facts, such as to which face in the net the point $a$ belongs.  Furthermore, through a slight abuse of notation, we will not always differentiate between a point $a$ in a net $N$ and its image $f(a)$ on the surface of the polyhedron $\mathcal P$, but through context it will be clear whether we are referring to $a$ or $f(a)$. 

It is also worth noting that in development of the figures and concepts for this paper we will utilize both synthetic geometry (geometry without coordinates, equations, and formulae) and analytic geometry (geometry with coordinates, equations, and formulae).  Since the faces of polyhedra and the points on their surface will be of the utmost importance, over the course of the paper these faces and the points contained within will respectively occur as subsets and elements of polyhedra, and when being considered in the context of net substructures as subsets and elements of both planes in the synthetic setting  and planes in the analytic setting.  However, through context it will be clear in exactly which way we are referring to said faces and points.

As mentioned above, several new net substructures will be introduced in this paper.  To help illustrate these concepts as they are introduced through sections 2 and 3 we will develop them for the specific case of the tetrahedron.  Due to this, sections 2 and 3 will include multiple results and concepts which apply to a wide array of polyhedra, but are primarily dedicated to solving the problem of shortest paths along the surface of the tetrahedron.  On the other hand, sections 4 and 5 include only one general result, but address the more complex problem of determining the collection of shortest paths along the surface of the cube.

\section {A Coordinate System on a Tetrahedron}

In this section we define a coordinate system for the set of points on the surface of a tetrahedron.  Before we define said coordinate system it is worth mentioning that several of the definitions, results, and, in this particular instance, the coordinate system in this paper are valid on convex polyhedra provided every edge of a given polyhedron is of the same length.  Without loss of generality, any such polyhedron can be scaled so that every edge is of length 1, so to simplify the wording in several places we will refer to such polyhedra as convex unit polyhedra.  Note that since tetrahedra are regular polyhedra, without loss of generality we can assume the tetrahedron we are working with is a unit tetrahedron.  The tetrahedron's faces will be labeled with elements of the set $\{F_1, F_2, F_3, F_4\}$ and the vertices with elements of the set $\left\{\{1,2,3\},\{1,2,4\},\{1,3,4\},\{2,3,4\}\right\}$ such that $F_n$ is incident to $\{n_1, n_2, n_3\}$ if and only if $n \in \{n_1, n_2, n_3\}$(see Figure \ref{tetlabel}). While it might seem that this notation for vertices of the tetrahedron is a bit cumbersome, this notation will be quite useful in the case of the cube and to maintain a sense of uniformity we will thus adapt it for the tetrahedron as well.   

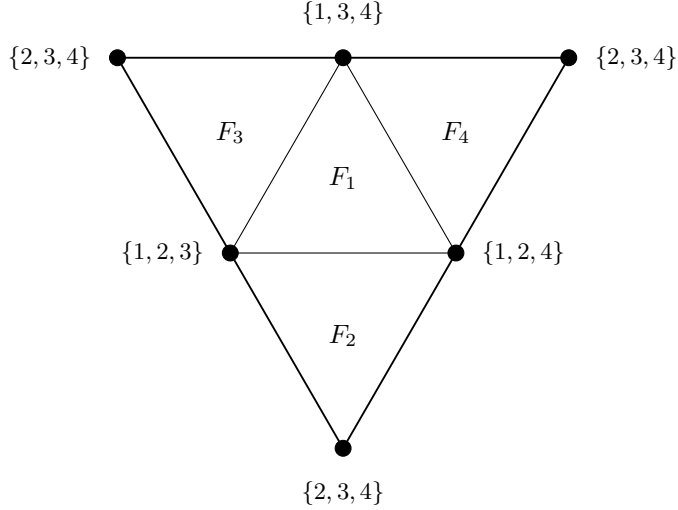
\begin{figure}[h]
  \centering
\begin{tikzpicture}[scale=3]
    \filldraw[fill=black,draw=black] (1.5,1.866) circle (1pt);
    \filldraw[fill=black,draw=black] (2.5,1.866) circle (1pt);
    \filldraw[fill=black,draw=black] (2,2.732) circle (1pt);
    \filldraw[fill=black,draw=black] (2,1) circle (1pt);
    \filldraw[fill=black,draw=black] (1,2.732) circle (1pt);
    \filldraw[fill=black,draw=black] (3,2.732) circle (1pt);
    \draw[line width = 0.25mm] (1,2.732) -- (2,1);
    \draw[line width = 0.25mm] (3,2.732) -- (2,1);
    \draw[line width = 0.25mm] (3,2.732) -- (1,2.732);
    \draw (1.5,1.866) -- (2,2.732) -- (2.5,1.866) -- (1.5,1.866) ;
    \node[font = {\small}] at (2,.8) {$\{2,3,4\}$};
    \node[font = {\small}] at (1.2,1.866) {$\{1,2,3\}$};
    \node[font = {\small}] at (2.8,1.866) {$\{1,2,4\}$};
    \node[font = {\small}] at (.7,2.732) {$\{2,3,4\}$};
    \node[font = {\small}] at (3.3,2.732) {$\{2,3,4\}$};
    \node[font = {\small}] at (2,2.932) {$\{1,3,4\}$};
    \node at (2,2.2) {$F_{1}$};
    \node at (2,1.5) {$F_{2}$};
    \node at (1.5,2.4) {$F_{3}$};
    \node at (2.5,2.4) {$F_{4}$};
\end{tikzpicture}
\caption{Labeling of the Tetrahedron}
\label{tetlabel}
\end{figure}

We will view the labeling of the net in Figure \ref{tetlabel} as fixed, thus yielding a fixed labeling of our tetrahedron which we will refer to as $\mathcal P_4$ for the duration of the paper.  We will follow the convention that any net considered will preserve the direction of rotation from one of a face's vertices to another and be a result of cutting the relevant polyhedron along a subset of its edges and unfolding the polyhedron into a subset of $\mathbb R^2$.  We now develop a coordinate system for use on the surface of any convex unit polyhedron (and in particular unit tetrahedra and unit cubes). 

\begin{definition} 
Given a face $F_n$ of a convex unit polyhedron $\mathcal P$ and a pair of vertices $u$ and $v$ incident to $F_n$ such that $v$ occurs directly after $u$ under counter-clockwise rotation about the interior of $F_n$, let $g$ be the function $g:F_n \rightarrow \mathbb R^2$ such that $g$ maps $F_n$ into the upper half-plane, preserves the distance between points in $F_n$, and maps $u$ and $v$ to $(0,0)$ and $(1,0)$ respectively, and thus $\overline{uv}$ to $\overline{(0,0),(1,0)}$.  Given a point $p$ on $\mathcal P$, if $g(p)=(x_0,y_0)$, $p$ is an element of $F_n$, and $F_m$ is the other face of $\mathcal P$ incident to $u$ and $v$, then $p$ will be said to have the ordered quadruple $(F_n, F_m, x_0, y_0)$ as a {\em representation}. Given a representation $(F_n,F_m,x_0,y_0)$ for a point $p$, we will refer to $(x_0,y_0)$ as the {\em standard position} of $p$ with respect to $(F_n,F_m,x_0,y_0)$. This representation of $p$ has {\em home-face} $F_n$, {\em shared-face} $F_m$, $x$-{\em coordinate} $x_0$, and $y$-{\em coordinate} $y_0$.
\end{definition}

\begin{figure}[h]
  \centering
\begin{tikzpicture}[scale=1.5]
    \filldraw[fill=black,draw=black] (2,4) circle (1pt);
    \filldraw[fill=black,draw=black] (2.15,1.6) circle (1pt);
    \filldraw[fill=black,draw=black] (-0.25,1.4) circle (1pt);
    \filldraw[fill=black,draw=black] (3.5,.8) circle (1pt);
    \filldraw[fill=black,draw=black] (1.85,2.1) circle (1pt);
    \draw[line width = 0.25mm] (2,4) -- (2.15,1.6);
    \draw[line width = 0.25mm] (2.15,1.6) -- (-0.25,1.4);
    \draw[line width = 0.25mm] (-0.25,1.4) -- (3.5,.8);
    \draw[line width = 0.25mm] (-0.25,1.4) -- (2,4);
    \draw[line width = 0.25mm] (3.5,.8) -- (2.15,1.6);
    \draw[line width = 0.25mm] (3.5,.8) -- (2,4);
    \node[font = {\small}] at (2.15,1.4) {$v$};
    \node[font = {\small}] at (-0.25,1.25) {$u$};
    \node[font = {\small}] at (1.85,1.95) {$p$};
    \node at (1.85,1.35) {$F_{m}$};
    \node at (1.3,2.2) {$F_{n}$};
    \draw[line width = 0.025mm] (5,3.9) -- (8.5,3.9);
    \draw[line width = 0.025mm] (5,3.9) -- (5,.9);
    \draw[line width = 0.025mm] (5,.9) -- (8.5,.9);
    \draw[line width = 0.025mm] (8.5,.9) -- (8.5,3.9);
    \draw[line width = 0.025mm] (6.75,3.4) -- (5.8,1.55);
    \draw[line width = 0.025mm] (5.8,1.55) -- (7.7,1.55);
    \draw[line width = 0.025mm] (5.8,1.55) -- (6.15,.9);
    \draw[line width = 0.025mm] (7.7,1.55) -- (6.75,3.4);
    \draw[line width = 0.025mm] (7.7,1.55) -- (7.35,.9);
    \draw[line width = 0.025mm] (5,3.4) -- (8.5,3.4);
    \draw[line width = 0.025mm] (5.8,1.55) -- (5,3.05);
    \draw[line width = 0.025mm] (7.7,1.55) -- (8.5,3.05);
    \filldraw[fill=black,draw=black] (6.75,3.4) circle (1pt);
    \filldraw[fill=black,draw=black] (5.8,1.55) circle (1pt);
    \filldraw[fill=black,draw=black] (7.7,1.55) circle (1pt);
    \filldraw[fill=black,draw=black] (7.35,1.9) circle (1pt);
    \node[font = {\small}] at (6.75,3.6) {$(\frac{1}{2},\frac{\sqrt{3}}{2})$};
    \node[font = {\small}] at (7.925,1.35) {$(1,0)$};
    \node at (6.75,2.3) {$F_{n}$};
    \node[font = {\small}] at (7.15,1.75) {$(x_{0},y_{0})$};
    \node at (6.75,1.1) {$F_{m}$};
    \node[font = {\small}] at (3.9,2.65) {$g$};
    \draw[->][ultra thick](3.5,2.5)--(4.3,2.5);
    \draw[stealth-stealth] (5.8,.9) -- (5.8,3.9)node[below right]{$y$};
    \draw[stealth-stealth] (5,1.55) -- (8.5,1.55)node[above left]{$x$};
\end{tikzpicture}
\caption{The Image of the Point $p=(F_n, F_m, x_0, y_0)$}
\end{figure}
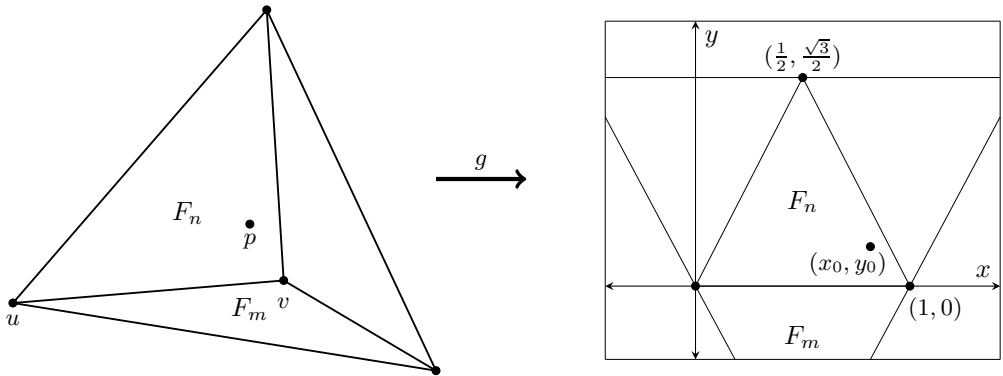

It is worth noting that since any edge of a given face of $\mathcal P_4$ can be identified with the line segment $\overline{(0, 0)(1, 0)}$, if $\{n_1, n_2, n_3, n_4\}=\{1, 2, 3, 4\}$ (an assumption we will keep for the duration of our discussion of $\mathcal P_4$) given a point $p$ with home-face $F_{n_1}$, $p$ has three representations, one with each of the faces $F_{n_2}$, $F_{n_3}$, and $F_{n_4}$ as the shared-face.  At first glance this might seem problematic, but while a given point can be represented by multiple ordered quadruples, a given ordered quadruple defines a unique point.  In this way, a point on a polyhedron can be thought of as an equivalence class consisting of its representations.  However, through a slight abuse of notation given a point $p$ and one of its representations $r$, we will simply write ``$p=r$.''   Furthermore, each pair of home-faces and shared-faces defines a different copy of $\mathbb R^2$ which $p$ can be viewed as an element of, and in this copy of $\mathbb R^2$ while the home-face is inherently a subset of the first quadrant the shared-face has not been given a distinct location, only the edge it shares with the home-face has.  Rather the exact location of the image of the shared-face will be determined by which net is being used, a detail which will be addressed in greater detail in section 3.    

Since under certain circumstances it will be useful to be able to switch from one representation of a point to another we now establish the relationships between said representations for points of $\mathcal P_4$.  First note, if a point is a vertex of $\mathcal P_4$, say $\{n_1, n_2, n_3\}$, then it can be represented as $\left(F_{n_1}, F_{n_4}, \frac{1}{2},\frac {\sqrt {3}}{2}\right)$, $\left(F_{n_2}, F_{n_4},  \frac {1}{2}  ,  \frac {\sqrt {3}}{2}  \right)$, and $\left(F_{n_3}, F_{n_4},  \frac {1}{2}  ,  \frac {\sqrt {3}}{2}  \right)$.  Next note, if a point lies on an edge of $\mathcal P_4$ say, $\overline {\{n_1, n_2, n_4\} \{n_1, n_2, n_3\}}$, and is represented as $(F_{n_1}, F_{n_2}, x, 0)$ then it can also be represented as $(F_{n_2}, F_{n_1}, 1-x, 0)$.  In this way, for any point incident to more than one face we can switch between representations using said faces as the point's home-face.  However, it takes a little more care to switch between representations with the same home-face and different shared-faces.  Due to this we introduce the following lemma.

\begin{lemma}\label{tet-rep}
Suppose $\{n_1, n_2, n_3, n_4\}=\{1, 2, 3, 4\}$ such that in nets of our copy of $\mathcal P_4$ the faces $F_{n_2}$, $F_{n_3}$, and $F_{n_4}$ occur in the same counter-clockwise order around the face $F_{n_1}$ as the faces $F_2$, $F_3$, and $F_4$ occur around the face $F_1$.  Then if a point $p\in\mathcal P_4$ can be represented as $(F_{n_1}, F_{n_2}, x, y)$, then it can also be represented as 
\[\left(F_{n_1}, F_{n_4}, \frac {1-x+\sqrt{3}y}{2}, \frac{\sqrt{3}-\sqrt{3}x-y}{2}\right).\] 
\end{lemma}

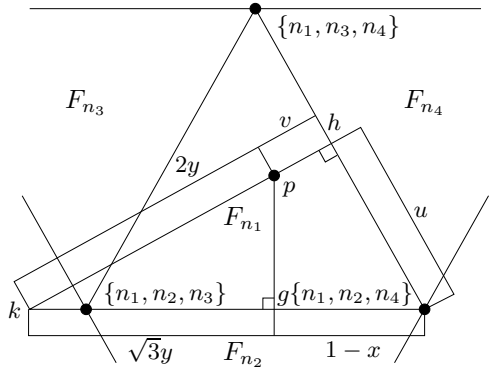
\begin{figure}[h]
  \centering
\begin{tikzpicture}[scale=1]
    \filldraw[fill=black,draw=black] (3,4) circle (2pt);
    \filldraw[fill=black,draw=black] (.75,0) circle (2pt);
    \filldraw[fill=black,draw=black] (5.25,0) circle (2pt);
    \filldraw[fill=black,draw=black] (3.25,1.78) circle (2pt);
    \draw[line width = 0.025mm] (0,4) -- (6,4);
    \draw[line width = 0.025mm] (3,4) -- (5.25,0) -- (.75,0) -- (3,4);
    \draw[line width = 0.025mm] (.75,0) -- (0,0) -- (4,2.2);
    \draw[line width = 0.025mm] (3.25,1.78) -- (3.25,-.35);
    \draw[line width = 0.025mm] (-.015,0) -- (-.015,-.35) -- (5.25,-.35) -- (5.25,0);
    \draw[line width = 0.025mm] (3.1,0) -- (3.1,.15) -- (3.25,.15);
    \node[font = {\small}] at (4.1,3.75) {$\{n_{1},n_{3},n_{4}\}$};
    \node[font = {\normalsize}] at (.75,2.75) {$F_{n_3}$};
    \node[font = {\normalsize}] at (5.25,2.75) {$F_{n_4}$};
    \node[font = {\small}] at (4.05,2.5) {$h$};
    \node[font = {\small}] at (3.45,1.6) {$p$};
    \node[font = {\small}] at (-.2,0) {$k$};
    \node[font = {\small}] at (4.3,-.55) {$1-x$};
    \node[font = {\small}] at (1.6,-.55) {$\sqrt{3}y$};
    \node[font = {\normalsize}] at (2.85,-.6) {$F_{n_2}$};
    \node[font = {\small}] at (3.375,.15) {$g$};
    \node[font = {\small}] at (4.275,.2) {$\{n_{1},n_{2},n_{4}\}$};
    \node[font = {\small}] at (1.8,.2) {$\{n_{1},n_{2},n_{3}\}$};
    \node[font = {\normalsize}] at (2.85,1.2) {$F_{n_1}$};
\usetikzlibrary{calc}
\newcommand\rightAngle[4]{
  \pgfmathanglebetweenpoints{\pgfpointanchor{#2}{center}}{\pgfpointanchor{#1}{center}}
  \coordinate (tmpRA) at ($(#2)+(\pgfmathresult+45:#4)$);
  \draw[black,thin] ($(#2)!(tmpRA)!(#1)$) -- (tmpRA) -- ($(#2)!(tmpRA)!(#3)$);
}
  \coordinate (O) at (4,2.2);
  \coordinate (X) at (0,0);
  \coordinate (Y) at (5.25,0);
  \rightAngle{X}{O}{Y}{0.25}
  \draw[line width = 0.05mm] (4,2.2) -- (4.4,2.42) -- (5.640453489,.2) -- (5.25,0);
\node[font = {\small}] at (5.2,1.3) {$u$};
  \draw[line width = 0.05mm] (3.8,2.5777) -- (-.2068669528,.368223176) -- (0,0);
\node[font = {\small}] at (3.4,2.5) {$v$};
\node[font = {\small}] at (2.1,1.9) {$2y$};
  \draw[line width = 0.05mm] (3.25,1.78) -- (3.039914163,2.15395279);
  \draw[line width = 0.05mm] (1.15,-.7114) -- (-.1,1.511);
  \draw[line width = 0.05mm] (4.85,-.7114) -- (6.1,1.51111);
\end{tikzpicture}
\caption{Representations of a Point $p$ with Different Shared-Faces}
\end{figure}

\begin{proof}
As mentioned previously, $p$ has a representation for which $F_{n_4}$ is used as the shared face.  Suppose $(u,v) \in \mathbb R^2$ is such that $p=(F_{n_1}, F_{n_4}, u, v)$.  Let $g \in \overline{\{n_1, n_2, n_3\}\{n_1, n_2, n_4\}}$ and $h \in \overline{\{n_1, n_2, n_4\}\{n_1, n_3, n_4\}}$ such that $\overline {pg} \perp \overline {\{n_1, n_2, n_3\}\{n_1, n_2, n_4\}}$ and $\overline{ph} \perp \overline {\{n_1, n_2, n_4\}\{n_1, n_3, n_4\}}$.  Now, let $k=\overleftrightarrow{ph} \cap \overleftrightarrow{g\{n_1, n_2, n_3\}}$.  Since $F_{n_1}$ is equilateral, the triangle $\triangle kpg$ is a 30-60-90 triangle, and so the lengths of line segments $\overline{gk}$ and $\overline{pk}$ are $gk=\sqrt{3}y$ and $pk=2y$.  Since $g\{n_1, n_2, n_4\}=1-x$, we have $k\{n_1, n_2, n_4\}=1-x+\sqrt{3}y$.  Similarly, $\triangle k\{n_1, n_2, n_4\}h$ is a 30-60-90 triangle, and so $u=h\{n_1, n_2, n_4\}=\frac{1-x+\sqrt{3}y}{2}$ and 
$v=ph=\frac{\sqrt{3}-\sqrt{3}x-y}{2}$. 
\end{proof}

Note: due to the symmetries of $\mathcal P_4$, the transformation used in the previous lemma to switch between viewing a point $p$ as having shared-face $F_{n_2}$ and having shared-face $F_{n_4}$ can also be used to switch between viewing a point $p$ as having shared-face $F_{n_4}$ and having shared-face $F_{n_3}$ or to switch between viewing a point $p$ as having shared-face $F_{n_3}$ and having shared-face $F_{n_2}$.  So for any point $p$ on face $F_{n_1}$ of $\mathcal P_4$, one can switch between the representations of $p$ viewing $p$ as having any of the remaining faces as its shared-face (by applying the transformation at most twice).   We have thus established how to switch between a given representation of a point $p$ and any other representation of $p$.

\section {Landscapes of a Tetrahedron}

Since the shortest path on the surface of a convex polyhedron is a line segment contained in one of the polyhedron's nets, to calculate the length of said paths or determine the set of points along the paths it is sufficient to do so for each subset of the nets in which these paths could be contained.  We now provide for reference the definition of a dual graph, and introduce the new notion of a landscape which will be a useful tool in determining exactly which subsets of a polyhedron's nets need to be considered for such calculations.  Since the shortest path between two points on the same face of a polyhedron is simply the line segment in $\mathbb R^3$ connecting the two points, landscapes need only be defined for a net's subsets of at least two faces.

\begin{definition}
Given a plane graph $G$, the \emph{dual graph} of $G$ is the graph $H$ for which the set of vertices of $H$ is equal to the set of faces of $G$, provided for all pairs of vertices $ a, b$ in $H$, $a$ and $b$ are connected by an edge in $H$ if and only if $a$ and $b$ are adjacent faces in $G$.
\end{definition}

\begin{definition}
Let $\mathcal P$ be a convex polyhedron and $N$ some net of $\mathcal P$.  Considering the vertices and edges of $N$ as a plane graph with faces given by the faces of $\mathcal P$, let $G$ be the dual graph of $N$.  If $H$ is a path subgraph of $G$ of at least two vertices, then the union of the collection of faces of $N$ given by the vertices of $H$ form a \emph{landscape} $L$.  In particular, if $F_{n_1}$ and $F_{n_2}$ are the vertices of $G$ with degree 1, then we say $L$ is a landscape from \emph{origin face} $F_{n_1}$ to \emph{destination face} $F_{n_2}$ and denote this by $L(F_{n_1} \rightarrow F_{n_2})$.
\end{definition}

\begin{definition}
Let $L$ be a landscape of a convex unit polyhedron $\mathcal P$ and $G$ be the dual graph of $L$.  If $H$ is a path subgraph of $G$ of at least two vertices, then the union of the collection of faces of $\mathcal P$ given by the vertices of $H$ form a \emph{sublandscape} $L'$ of $L$. In particular, if $H$ is a proper subgraph of $G$, then $L'$ is a \emph{proper sublandscape} of $L$.
\end{definition}

\begin{figure}[h]
  \centering
\begin{tikzpicture}[scale=1.75]
    \filldraw[fill=black,draw=black] (1,2.732) circle (1pt);
    \filldraw[fill=black,draw=black] (2,2.732) circle (1pt);
    \filldraw[fill=black,draw=black] (3,2.732) circle (1pt);
    \filldraw[fill=black,draw=black] (1.5,1.866) circle (1pt);
    \filldraw[fill=black,draw=black] (2.5,1.866) circle (1pt);
    \filldraw[fill=black,draw=black] (2,1) circle (1pt);
    \filldraw[fill=black,draw=black] (1,-.5) circle (1pt);
    \filldraw[fill=black,draw=black] (2,-.5) circle (1pt);
    \filldraw[fill=black,draw=black] (3,-.5) circle (1pt);
    \filldraw[fill=black,draw=black] (1.5,.366) circle (1pt);
    \filldraw[fill=black,draw=black] (2.5,.366) circle (1pt);
    \filldraw[fill=black,draw=black] (3.5,.366) circle (1pt);
    \draw[line width = 0.25mm] (1,-.5) -- (2,-.5) -- (3,-.5) -- (3.5,.366) -- (2.5,.366) -- (1.5,.366) -- (2,-.5) -- (2.5,.366) -- (3,-.5);
    \draw[line width = 0.25mm] (1,-.5) -- (1.5,.366);
    \draw[line width = 0.25mm] (1,2.732) -- (2,2.732) -- (3,2.732) -- (2.5,1.866) -- (2.5,1.866) -- (2,1) -- (1.5,1.866) -- (2,2.732) -- (2.5,1.866) -- (1.5,1.866) -- (1,2.732);
    \node at (1.5,2.45) {$F_{n_3}$};
    \node at (2.5,2.45) {$F_{n_4}$};
    \node at (2,2.2) {$F_{n_1}$};
    \node at (2,1.5) {$F_{n_2}$};
    \node at (2,.1) {$F_{n_3}$};
    \node at (3,.1) {$F_{n_4}$};
    \node at (1.5,-.2) {$F_{n_2}$};
    \node at (2.5,-.2) {$F_{n_1}$};
    \node[font = {\small}] at (1,2.9) {$\{n_{2},n_{3},n_{4}\}$};
    \node[font = {\small}] at (2,2.9) {$\{n_{1},n_{3},n_{4}\}$};
    \node[font = {\small}] at (3,2.9) {$\{n_{2},n_{3},n_{4}\}$};
    \node[font = {\small}] at (.9,1.866) {$\{n_{1},n_{2},n_{3}\}$};
    \node[font = {\small}] at (3.1,1.866) {$\{n_{1},n_{2},n_{4}\}$};
    \node[font = {\small}] at (2,.9) {$\{n_{2},n_{3},n_{4}\}$};
    \node[font = {\small}] at (1,-.65) {$\{n_{1},n_{2},n_{4}\}$};
    \node[font = {\small}] at (2,-.65) {$\{n_{1},n_{2},n_{3}\}$};
    \node[font = {\small}] at (3,-.65) {$\{n_{1},n_{2},n_{4}\}$};
    \node[font = {\small}] at (1.5,.5) {$\{n_{2},n_{3},n_{4}\}$};
    \node[font = {\small}] at (2.5,.5) {$\{n_{1},n_{3},n_{4}\}$};
    \node[font = {\small}] at (3.5,.5) {$\{n_{2},n_{3},n_{4}\}$};
    \draw[->](3.9,2)--(4.1,2);
    \draw[->](3.9,-.1)--(4.1,-.1);
    \filldraw[fill=black,draw=black] (4.3,2.45) circle (1pt);
    \filldraw[fill=black,draw=black] (5.3,2.45) circle (1pt);
    \filldraw[fill=black,draw=black] (4.8,2.1) circle (1pt);
    \filldraw[fill=black,draw=black] (4.8,1.5) circle (1pt);
    \draw[line width = 0.25mm] (4.3,2.45) -- (4.8,2.1) -- (5.3,2.45);
    \draw[line width = 0.25mm] (4.8,2.1) -- (4.8,1.5);
    \node at (4.3,2.6) {$F_{n_3}$};
    \node at (5.3,2.6) {$F_{n_4}$};
    \node at (4.8,2.3) {$F_{n_1}$};
    \node at (4.8,1.35) {$F_{n_2}$};
    \filldraw[fill=black,draw=black] (4.3,-.25) circle (1pt);
    \filldraw[fill=black,draw=black] (5.3,-.25) circle (1pt);
    \filldraw[fill=black,draw=black] (4.75,.1) circle (1pt);
    \filldraw[fill=black,draw=black] (5.75,.1) circle (1pt);
    \draw[line width = 0.25mm] (4.3,-.25) -- (4.75,.1) -- (5.3,-.25) -- (5.75,.1);
    \node at (4.3,-.4) {$F_{n_2}$};
    \node at (5.3,-.4) {$F_{n_1}$};
    \node at (4.75,.25) {$F_{n_3}$};
    \node at (5.75,.25) {$F_{n_4}$};
\end{tikzpicture}
\caption{Dual Graphs of Two Nets of a Tetrahedron}
\end{figure}
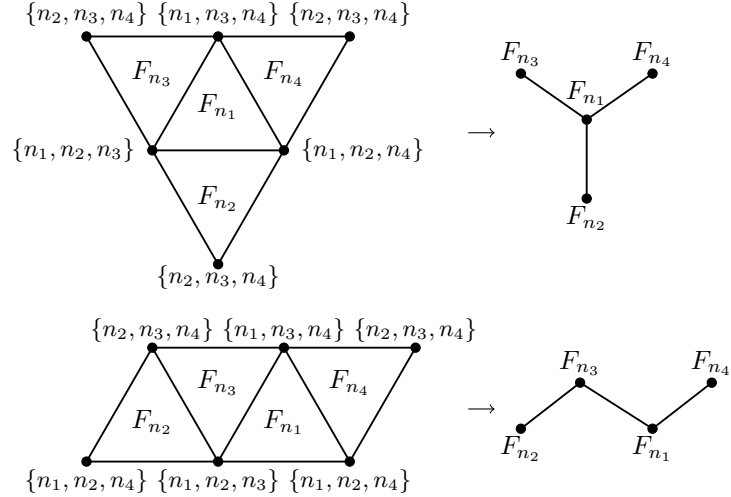

As mentioned earlier, the ordered quadruple representing a point assigns a location in $\mathbb R^2$ to points in the home-face, but does not assign a location in $\mathbb R^2$ to any other points of any other faces of a given convex unit polyhedron.  However, a choice of landscape of the polyhedron and a home-face along with the assignment of the location of just one of the home-face's edges would serve this purpose for every point within said landscape.  Furthermore, when performing the calculations inherently necessary for identifying and measuring shortest paths on the surface of a tetrahedron it will be convenient to consider a point as an element of the copy of $\mathbb R^2$ determined by one such choice and moments later as an element of the copy of $\mathbb R^2$ determined by a different choice of said type.  We address this need more rigorously with the following definition.

\begin{definition}
Let $L$ be a landscape of a convex unit polyhedron $\mathcal P$, $F_n$ be contained in $L$, $F_m$ be a face of $\mathcal P$ adjacent to $F_n$, and $u$ and $v$ be the vertices incident to $F_n$ and $F_m$ with $v$ occurring directly after $u$ under counter-clockwise rotation about the interior of $F_n$.  Let $O=(L, F_n, F_m)$ reference the subset of $\mathbb R^2$ contained in the image of $L$ when $F_n$ is mapped into the upper half-plane by a function $g:L \rightarrow \mathbb R^2$ which preserves the distance between points in $L$ and $u$ and $v$ are mapped to $(0,0)$ and $(1,0)$ respectively.  Refer to $O$ as an \emph{orientation} of $L$.  When we wish to view a point $p$ as an element of the orientation $(L, F_n, F_m)$ we will denote it by $p(L, F_n, F_m)$ and when we wish to denote its $x$-coordinate or $y$-coordinate we will denote them by $p(L, F_n, F_m)_x$ and $p(L, F_n, F_m)_y$ respectively. 
\end{definition}

For matters of discussion, it may be handy to view an orientation $O=(L, F_n, F_m)$ in certain contexts as an orientation of the polyhedron $\mathcal P$.  Since $O$ is an orientation of $L$, which in turn is a landscape of $\mathcal P$, this convention is admissible when convenient. 

Since two distinct landscapes of a convex polyhedron $\mathcal P$ are constructed from a different finite sequence of faces of $\mathcal P$, if the points $x$ and $y$ are each elements of two distinct landscapes $L$ and $K$ it is likely that the line segment $\overline{xy}$ in $L$ is not comprised of the same points nor of the same length as the line segment $\overline {xy}$ in $K$.  To help distinguish between the line segments joining two given points in each of multiple distinct landscapes, or orientations, we introduce the notion of the trail.  

\begin{definition}
Given $O$ an orientation of a polyhedron $\mathcal P$, if $p_1, p_2 \in O$ and the line segment $\overline{p_1p_2} \subset O$, then we call said line segment the \emph{trail} from $p_1$ to $p_2$ in $O$ and denote it $T_O(p_1,p_2)$.  The portion of $T_O(p_1,p_2)$ contained in face $F_n$ will be denoted by ${}_nT_O(p_1,p_2)$.
\end{definition}

Since trails are defined as subsets of copies of $\mathbb R^2$, as sets of points in $\mathbb R^2$ they are dependent upon the orientation being used, but since length is translation and rotation invariant the lengths of said paths are only dependent upon the landscape in question.  With this in mind we introduce the following notation.

\begin{definition}
Given $O=(L,F_n,F_m)$ an orientation of a polyhedron $\mathcal P$, if $\overline{p_1p_2} \subset O$, then we denote the length of $T_O(p_1,p_2)$ by $\abs{T_L(p_1,p_2)}$.  On the other hand, if $p_1,p_2 \in O$, but $\overline{p_1p_2} \not \subset O$ then we will set $\abs{T_L(p_1,p_2)}=\infty$.
\end{definition}

It might seem problematic that paths on the surface of a convex unit polyhedron are considered as subsets of $\mathbb R^2$.  However, it is important to remember that, due to the similarities between the way the coordinate system on convex unit polyhedra and the way their orientations have been defined, after application of the folding function the exact coordinates of the images of the points on these paths will be easy to determine.  For the curious reader an example of this will be shown at the end of this section.  

The following theorem establishes that evaluating the lengths of the trails between two points $p_1$ and $p_2$ on a convex unit polyhedron $\mathcal P$ is sufficient for determining the length of the shortest path restricted to the surface of $\mathcal P$ between $p_1$ and $p_2$.  Furthermore, since said path will actually be one of $\mathcal P$'s trails, the previous definition will allow us to determine the points along said path.

\begin{theorem}\label{fullgenerality}
Let $\mathcal P$ be a convex unit polyhedron, $F_{n}$ be a face of $\mathcal{P}$, $F_{m}$ and $F_{r}$ be faces of $\mathcal{P}$ distinct from $F_{n}$, and $p_1 \in F_n\setminus F_m$ and $p_2 \in F_m \setminus F_n$ be points on the surface of $\mathcal P$.  Then there exists an orientation $O=(L, F_n, F_r)$ of $\mathcal P$ such that the shortest path between $p_1$ and $p_2$ is $T_O(p_1, p_2)$.  
\end{theorem}

\begin{proof}
Since $\mathcal P$ is a convex polyhedron, it is known that the shortest path restricted to the surface of $\mathcal P$ between $p_1$ and $p_2$ will be a straight line segment in one of $\mathcal P$'s nets.  Let $N$ be said net, $G$ be $N$'s dual graph, and $S$ be said path.  Since line segments are connected, the sequence of vertices in $G$ corresponding to the faces traversed by $S$ will form a walk in $G$.   Also since $\mathcal P$ is convex each of $\mathcal P$'s faces will be convex.  Due to this convexity, the fact that $S$ is a line segment gives us that the portion of $S$ within any given face will be either empty or a single connected subset of $S$.  So, in fact, the sequence of vertices in $G$ corresponding to the faces traversed by $S$ will be a path in $G$.  Due to this, the sequence of faces traversed by $S$ will be a landscape of $\mathcal P$.  Call this landscape $L$, and note that $S$ will be contained in any of the orientations associated with $L$, in particular $S$ will be contained in $(L, F_n, F_r)$.  
\end{proof}

From a computational standpoint, the fact that the shortest path can be considered as a line segment in the plane tells us that the set of points making up this shortest path $T_O(p_1,p_2)$ is a subset of the line $y-y_1=m(x-x_1)$, where $p_1=(F_n,F_r,x_1,y_1)$ and $\displaystyle m=\frac{y_1-p_2(L, F_n, F_r)_y}{x_1-p_2(L, F_n, F_r)_x}$ or in the special case that $x_1=p_2(L, F_n, F_r)_x$, the line $x=x_1$.  To address that the line segment forming this trail is finite, one must simply require that for $(x,y)$ to be on the line segment 
\[\min\left\{x_1,p_2(L, F_n, F_r)_x\right\} \leq x \leq \max\left\{x_1,p_2(L, F_n, F_r)_x\right\} \text{ and}\]
\[\min\left\{y_1,p_2(L, F_n, F_r)_y\right\} \leq y \leq \max\left\{y_1,p_2(L, F_n, F_r)_y\right\}\]

\noindent Finally, by a simple application of the Pythagorean Theorem, we have that the length of this trail is given by 
\[\abs{T_L(p_1, p_2)}=\sqrt{\left(x_1-p_2(L, F_n, F_r)_x\right)^2+\left(y_1-p_2(L, F_n, F_r)_y\right)^2}.\]

Since a given convex unit polyhedron $\mathcal P$ has only finitely many faces, it will have only finitely many orientations, and thus only finitely many trails between any two fixed points $p_1$ and $p_2$ on its surface. Furthermore, since the length of the shortest path on the surface of a convex unit polyhedron $\mathcal P$ between two of its points is given by the length of one of $\mathcal P$'s trails, we can establish that the length of said shortest path, which we will henceforth refer to as the \emph{surface distance}, can be defined as
\[d_{\mathcal P}(p_1,p_2)=\min\left\{\abs{T_L(p_1,p_2)}: L ~\text{is a landscape of}~\mathcal P\right\}.\]

Having shown that considering the landscapes of a convex unit polyhedron $\mathcal P$ is sufficient for determining surface distances, we will now define the subset of landscapes which are sufficient for determining surface distances.  The definition is useful not simply due to the attractiveness of having a set of landscapes which is both necessary and sufficient for our purposes, but also for speed and ease of computation.  Since a given polyhedron may have many landscapes it will be useful to be able to restrict one's attention to the landscapes in which shortest (and simplest) paths can occur.  To address this need we introduce the notion of a valid landscape.   

\begin{definition}
Given a landscape $L_i$ of a polyhedron $\mathcal P$, $L_i$ is said to be a {\em valid} landscape of $\mathcal P$ if there exist points $p_1, p_2 \in L_i$ such that $\abs{T_{L_i}(p_1,p_2)}=d_{\mathcal P}(p_1,p_2)$ and for all other landscapes $L_j$ of $\mathcal P$ with $p_1, p_2 \in L_j$, if $\abs{T_{L_i}(p_1,p_2)} = \abs{T_{L_j}(p_1,p_2)}$ then $L_i$ does not contain $L_j$ as a sublandscape. 
\end{definition}

Having defined the concept of a valid landscape we now turn our focus to determining just how many valid landscapes a tetrahedron has.  Note that   due to the symmetries of $\mathcal P_4$, the fact that the faces of $\mathcal P_4$ are pairwise adjacent, and the presence of a formula for switching from one shared-face to another given points $p_1 \in F_{n_1}$ and $p_2 \in F_{n_2}$, we can assume without loss of generality that the representations of the points $p_1$ and $p_2$ are such that $p_1=(F_{n_1}, F_{n_2}, x_1, y_1)$ and $p_2=(F_{n_2}, F_{n_1}, x_2, y_2)$.

\begin{theorem}\label{tet-valid}
Let $F_{n_1}$ and $F_{n_2}$ be distinct faces of $\mathcal P_4$.  
\begin{enumerate}
    \item Then there are 5 landscapes $L_i(F_{n_1} \rightarrow F_{n_2})$ of $\mathcal P_4$, and in particular all 5 landscapes are valid.
    \item Given two points $p_1 \in F_{n_1} \setminus F_{n_2}$ and $p_2 \in F_{n_2} \setminus F_{n_1}$, the surface distance between $p_1=(F_{n_1},F_{n_2},x_1,y_1)$ and $p_2=(F_{n_2},F_{n_1},x_2,y_2)$ is given by 
\[d_{\mathcal P_4}(p_1,p_2)=\min\left\{\abs{T_{L_i}(p_1,p_2)}: i \in \mathbb N \text{ with } 1 \leq i \leq 5\right\},\]
where for each $i \in \mathbb N$ with $1 \leq i \leq 5$, the trail length $\abs{T_{L_i}(p_1,p_2)}$ is given in the table below:

\begingroup
\begin{center}
\begin{tabular}{|c|c|}
    \hline
    Landscape & Trail Length Formula\\
    \hline
    \rule{0pt}{12pt}$L_1$ &\rule{-4pt}{12pt} $\sqrt{(x_1+x_2-1)^2+(y_1+y_2)^2}$\\[5pt]
    \hline
    \rule{0pt}{12pt}$L_2$ &\rule{-4pt}{12pt} $\sqrt{(x_1-x_2+1)^2+(y_1-y_2)^2}$\\[5pt]
    \hline
    \rule{0pt}{12pt}$L_3$ &\rule{-4pt}{12pt} $\sqrt{(x_1-x_2-1)^2+(y_1-y_2)^2}$\\[5pt]
    \hline
    \rule{0pt}{16pt}$L_4$ &\rule{-4pt}{16pt} $\sqrt{(x_1+x_2)^2+(y_1+y_2-\sqrt{3})^2}$\\[5pt]
    \hline
    \rule{0pt}{16pt}$L_5$ &\rule{-4pt}{16pt} $\sqrt{(x_1+x_2-2)^2+(y_1+y_2-\sqrt{3})^2}$\\[5pt]
    \hline
\end{tabular}
\end{center}
\endgroup
\end{enumerate}

\end{theorem}

\begin{proof}
Suppose $\{n_1, n_2, n_3, n_4\}=\{1, 2, 3, 4\}$ such that in nets of our copy of $\mathcal P_4$ the faces $F_{n_2}$, $F_{n_3}$, and $F_{n_4}$ occur in the same counter-clockwise order around the face $F_{n_1}$ as the faces $F_2$, $F_3$, and $F_4$ occur around the face $F_1$.  To discuss the trails which occur in the definition of a valid landscape it will be necessary to consider a pair of points $p_1$ and $p_2$ on the surface of $\mathcal P_4$. 
 To construct the landscapes we will break the problem into cases determined by the number of vertices in the dual graphs of each landscape.  Since $\mathcal P_4$ has only four faces, no dual graph of $\mathcal P_4$ contains a path of more than four vertices, and since the dual graph of a landscape must contain at least two vertices; landscapes of $\mathcal P_4$ must contain either two, three, or four faces.  Once we have constructed each landscape $L_i$ we will then determine $p_2(L_i, F_{n_1},F_{n_2})$ so that we can use it to develop formulae for $\abs{T_{L_i}(p_1,p_2)}$.  Having developed a formula for each trail length, our final step will be to produce points witnessing the validity of each landscape. 

\begin{figure}[h!]
  \centering
\begin{tikzpicture}[scale=1]
    \node[diamond,
    draw,
    minimum width =3.6cm,
    minimum height =5.9cm] (d) at (3,4) {};
    \filldraw[fill=black,draw=black] (d.north) circle (2pt);
    \filldraw[fill=black,draw=black] (d.west) circle (2pt);
    \filldraw[fill=black,draw=black] (d.south) circle (2pt);
    \filldraw[fill=black,draw=black] (d.east) circle (2pt);
    \draw[line width = 0.05mm] (d.north) -- (d.west) -- (d.south) -- (d.east) -- cycle;
    \draw[line width=0.05mm] (d.east) -- (d.west);
    \node[above][font = {\small}] at (d.north) {$\{n_{1},n_{3},n_{4}\}$};
    \node[below][font = {\small}] at (d.south) {$\{n_{2},n_{3},n_{4}\}$};
    \node[left][font = {\small}] at (d.west) {$\{n_{1},n_{2},n_{3}\}$};
    \node[right][font = {\small}] at (d.east) {$\{n_{1},n_{2},n_{4}\}$};
    \node[above, inner sep=25] at (d.south) {$F_{n_2}$};
    \node[below, inner sep=25] at (d.north) {$F_{n_1}$};
    \draw [line width=0.01mm] (d.west) rectangle (2.5,4.3);
    \draw[ultra thin,dash pattern={on 1pt}] (2.5,4) -- (2.5,3.2);
    \filldraw[fill=black,draw=black] (2.5,3.2) circle (2pt);
    \node[left][font = {\tiny}] at (2.5,3.2) {$p_{2}$};
    \draw [ultra thin] (2.5,3.2) rectangle (3.5,2.95);
    \node[below, inner sep=.5][font = {\tiny}] at (3,2.95) {$|x_{1}+x_{2}-1|$};
    \draw [ultra thin] (3.5,3.2) rectangle (3.8,5.2);
    \filldraw[fill=black,draw=black] (3.5,5.2) circle (2pt);
    \node [right, inner sep=.5] [font = {\tiny}] at (3.8,4.2) {$y_{1}+y_{2}$};
    \draw [ultra thin] (3.5,5.2) -- (2.5,3.2);
    \draw [ultra thin] (3.3,3.2) rectangle (3.5,3.4);
    \node[above] [font = {\tiny}] at (3.5,5.2) {$p_{1}$};
    \node[font = {\tiny}] at (2.05,4.45) {$1-x_{2}$};
\end{tikzpicture}
    \caption{$L_1(F_{n_1} \rightarrow F_{n_2})$ and $T_O(p_1, p_2)$, where $O=(L_1,F_{n_1},F_{n_2})$}
    \label{trail 1}
\end{figure}
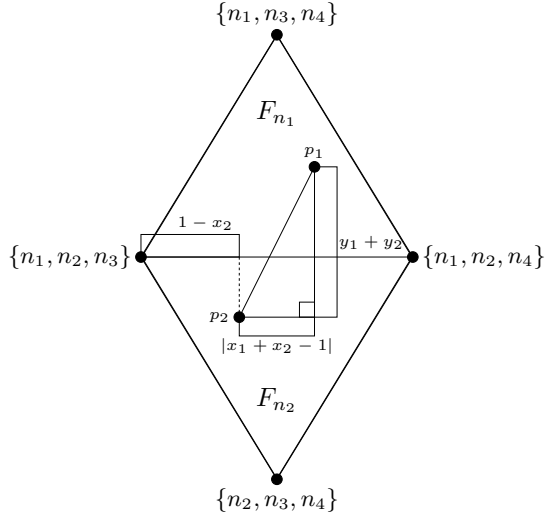

\begin{case} 
The dual graph of the landscape is a path of two vertices.
\end{case}

Since we are constructing landscapes from $F_{n_1}$ to $F_{n_2}$, and $F_{n_1}$ and $F_{n_2}$ have only one edge in common, the only such landscape is that given in Figure \ref{trail 1}, which we will refer to as $L_1$. 

Since $p_2$'s home-face is $F_{n_2}$ and $p_2$'s shared-face is $F_{n_1}$, in the orientation $(L_1, F_{n_1}, F_{n_2})$ the point $p_2$ has been rotated $180$ degrees about the origin and shifted $1$ unit to the right from the standard position with respect to the representation $(F_{n_2},F_{n_1},x_2,y_2)$. Due to this $p_2(L_1, F_{n_1}, F_{n_2})=(1-x_2, -y_2)$.  Thus we get that
\[\abs{T_{L_1}(p_1, p_2)}=\sqrt{(x_1+x_2-1)^2+(y_1+y_2)^2}.\]

\begin{case}
The dual graph of the landscape is a path of three vertices.
\end{case}

Since we are constructing landscapes from $F_{n_1}$ to $F_{n_2}$ and $\mathcal P_4$ only has four faces; the first face must be $F_{n_1}$, the third face must be $F_{n_2}$, and the second face must be either $F_{n_3}$ or $F_{n_4}$.  Since two distinct faces of $\mathcal P_4$ have exactly one edge in common the two landscapes given in Figure \ref{tettrail2} and Figure \ref{tettrail3} are the only such landscapes.

\begin{figure}[h!]
  \centering
\begin{tikzpicture}[scale=1]
    \node[isosceles triangle,
    isosceles triangle apex angle=60,
    rotate=270,
    draw,
    minimum size=3cm] (T)at (0,0) {};
    \node[above][font = {\small}] at (T.right corner) {$\{n_{2},n_{3},n_{4}\}$};
    \node[above][font = {\small}] at (T.left corner) {$\{n_{1},n_{3},n_{4}\}$};
    \node[font = {\small}] at (1,-1.8) {$\{n_{1},n_{2},n_{3}\}$};
    \filldraw[fill=black,draw=black] (T.right corner) circle (2pt);
    \filldraw[fill=black,draw=black] (T.left corner) circle (2pt);
    \filldraw[fill=black,draw=black] (T.apex) circle (2pt);
    \node[above, inner sep=60] at (T.apex) {$F_{n_3}$};
    \node[below, inner sep=20] at (T.left corner) {$F_{n_1}$};
    \node[below, inner sep=20] at (T.right corner) {$F_{n_2}$};
    \filldraw[fill=black,draw=black] (3.5,-2) circle (2pt);
    \filldraw[fill=black,draw=black] (-3.5,-2) circle (2pt);
    \draw[line width=0.05mm] (1.75,1) -- (3.5,-2) -- (0,-2);
    \draw[line width=0.05mm] (-1.75,1) -- (-3.5,-2) -- (0,-2);
    \node[left][font = {\small}] at (-3.5,-2) {$\{n_{1},n_{2},n_{4}\}$};
    \node[right][font = {\small}] at (3.5,-2) {$\{n_{1},n_{2},n_{4}\}$};
    \filldraw[fill=black,draw=black] (2.2,-.7) circle (2pt);
    \node[below][font = {\small}] at (2.2,-.7) {$p_{1}$};
    \filldraw[fill=black,draw=black] (-1.2,-1.4) circle (2pt);
    \node[below][font = {\small}] at (-1.2,-1.4) {$p_{2}$};    
    \draw[ultra thin] (-1.2,-1.4) -- (2.2,-.7);
    \draw [ultra thin] (-1.2,-1.4) rectangle (-1.5,-.7);
    \draw [ultra thin] (2.2,-.7) rectangle (-1.2,-.4);
    \draw [ultra thin] (-1.2,-.7) rectangle (-1,-.9);
    \node[above, inner sep=2][font = {\small}] at (.1,-.4) {$1+x_{1}-x_{2}$};
    \node[left, inner sep=2][font = {\small}] at (-1.5,-.95) {$|y_{1}-y_{2}|$};
    \draw [ultra thin] (-3.5,-2) rectangle (2.2,-2.4);
    \node[below][font = {\small}] at (-.65,-2.4) {$1+x_{1}$};  
\end{tikzpicture}
    \caption{$L_2(F_{n_1} \rightarrow F_{n_2})$ and $T_O(p_1, p_2)$, where $O=(L_2,F_{n_1},F_{n_2})$}
    \label{tettrail2}
\end{figure}
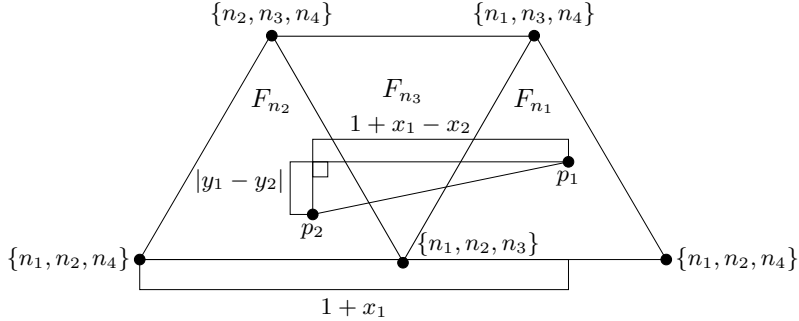

\begin{figure}[h!]
  \centering
\begin{tikzpicture}[scale=1]
    \node[isosceles triangle,
    isosceles triangle apex angle=60,
    rotate=270,
    draw,
    minimum size=3cm] (T)at (0,0) {};
    \node[above][font = {\small}] at (T.right corner) {$\{n_{1},n_{3},n_{4}\}$};
    \node[above][font = {\small}] at (T.left corner) {$\{n_{2},n_{3},n_{4}\}$};
    \node[font = {\small}] at (-1,-1.825) {$\{n_{1},n_{2},n_{4}\}$};
    \filldraw[fill=black,draw=black] (T.right corner) circle (2pt);
    \filldraw[fill=black,draw=black] (T.left corner) circle (2pt);
    \filldraw[fill=black,draw=black] (T.apex) circle (2pt);
    \node[above, inner sep=50] at (T.apex) {$F_{n_4}$};
    \node[below, inner sep=30] at (T.left corner) {$F_{n_2}$};
    \node[below, inner sep=30] at (T.right corner) {$F_{n_1}$};
    \filldraw[fill=black,draw=black] (3.5,-2) circle (2pt);
    \filldraw[fill=black,draw=black] (-3.5,-2) circle (2pt);
    \draw[line width=0.05mm] (1.75,1) -- (3.5,-2) -- (0,-2);
    \draw[line width=0.05mm] (-1.75,1) -- (-3.5,-2) -- (0,-2);
    \node[left][font = {\small}] at (-3.5,-2) {$\{n_{1},n_{2},n_{3}\}$};
    \node[right][font = {\small}] at (3.5,-2) {$\{n_{1},n_{2},n_{3}\}$};
    \filldraw[fill=black,draw=black] (-1.2,-.7) circle (2pt);
    \node[above][font = {\small}] at (-1.2,-.7) {$p_{1}$};
    \filldraw[fill=black,draw=black] (2.2,-1.4) circle (2pt);
    \node[right][font = {\small}] at (2.2,-1.4) {$p_{2}$};    
    \draw[ultra thin] (2.2,-1.4) -- (-1.2,-.7);
    \draw [ultra thin] (-1.2,-1.4) rectangle (-1.5,-.7);
    \draw [ultra thin] (2.2,-1.4) rectangle (-1.2,-1.65);
    \draw [ultra thin] (-1.2,-1.4) rectangle (-1,-1.2);
    \node[below, inner sep=2][font = {\small}] at (1.2,-1.625) {$1+x_{2}-x_{1}$};
    \node[left, inner sep=2][font = {\small}] at (-1.5,-.95) {$|y_{1}-y_{2}|$};
    \draw [ultra thin] (-3.5,-2) rectangle (2.2,-2.4);
    \node[below][font = {\small}] at (-.65,-2.4) {$1+x_{2}$};  
    \end{tikzpicture}
    \caption{$L_3(F_{n_1} \rightarrow F_{n_2})$ and $T_O(p_1, p_2)$, where $O=(L_3,F_{n_1},F_{n_2})$}
    \label{tettrail3}
\end{figure}
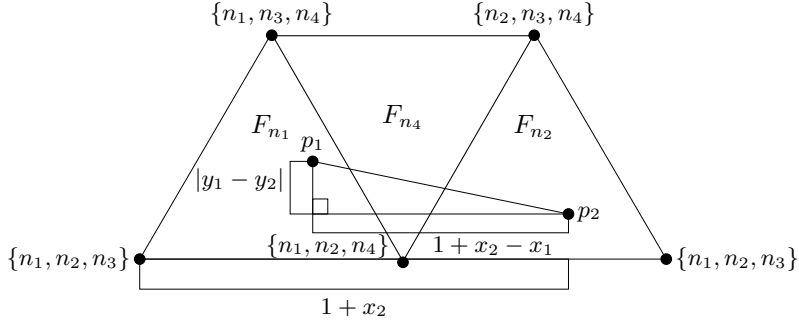

Since $p_2$'s home-face is $F_{n_2}$ and $p_2$'s shared-face is $F_{n_1}$, in the orientation $(L_2, F_{n_1}, F_{n_2})$ the point $p_2$ has been shifted $1$ unit to the left from $p_2$'s standard position with respect to the representation $(F_{n_2},F_{n_1},x_2,y_2)$.  Due to this $p_2(L_2, F_{n_1}, F_{n_2})=(x_2-1, y_2)$.  Thus we get that
\[\abs{T_{L_2}(p_1, p_2)}=\sqrt{(x_1-x_2+1)^2+(y_1-y_2)^2}.\]

Similarly, in the orientation $(L_3, F_{n_1}, F_{n_2})$ the point $p_2$ has been shifted $1$ unit to the right from $p_2$'s standard position with respect to the representation $(F_{n_2},F_{n_1},x_2,y_2)$.  Due to this $p_2(L_3, F_{n_1}, F_{n_2})=(x_2+1, y_2)$.  Thus we get that
\[\abs{T_{L_3}(p_1, p_2)}=\sqrt{(x_1-x_2-1)^2+(y_1-y_2)^2}.\]

\begin{case}
The dual graph of the landscape is a path of four vertices.
\end{case}

Since we are constructing landscapes from $F_{n_1}$ to $F_{n_2}$ and $\mathcal P_4$ only has four faces; the first face must be $F_{n_1}$, the fourth face must be $F_{n_2}$, the second face must be either $F_{n_3}$ or $F_{n_4}$, and the third face must be the remaining face.  Since two distinct faces of $\mathcal P_4$ have exactly one edge in common the two landscapes given in Figure \ref{tettrail4} and Figure \ref{tettrail5} are the only such landscapes.
  
Since $p_2$'s home-face is $F_{n_2}$ and $p_2$'s shared-face is $F_{n_1}$, in the orientation $(L_4, F_{n_1}, F_{n_2})$ the point $p_2$ has been rotated $180$ degrees about the origin and shifted $\sqrt{3}$ units up from $p_2$'s standard position with respect to the representation $(F_{n_2},F_{n_1},x_2,y_2)$.  Due to this $p_2(L_4, F_{n_1}, F_{n_2})=(-x_2, \sqrt{3}-y_2)$.  Thus we get that
\[\abs{T_{L_4}(p_1, p_2)}=\sqrt{(x_1+x_2)^2+(y_1+y_2-\sqrt{3})^2}.\]

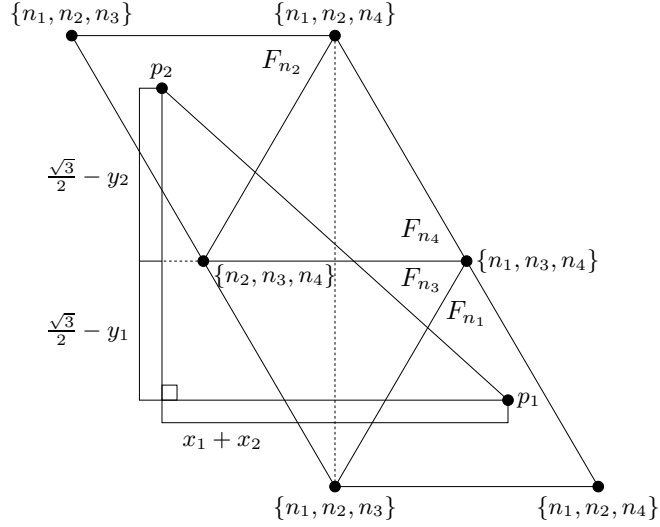
\begin{figure}[h]
  \centering
\begin{tikzpicture}[scale=1]
\filldraw[fill=black,draw=black] (0,6) circle (2pt);
\filldraw[fill=black,draw=black] (0,0) circle (2pt);
\filldraw[fill=black,draw=black] (3.5,0) circle (2pt);
\filldraw[fill=black,draw=black] (-3.5,6) circle (2pt);
\filldraw[fill=black,draw=black] (1.75,3) circle (2pt);
\filldraw[fill=black,draw=black] (-1.75,3) circle (2pt);
\draw[line width = 0.05mm] (0,0) -- (3.5,0) -- (1.75,3) -- (0,0) -- (-1.75,3) -- (1.75,3) -- (0,6) -- (-1.75,3) -- (-3.5,6) -- (0,6);
\node[below][font = {\small}] at (3.5,0) {$\{n_{1},n_{2},n_{4}\}$};
\node[above][font = {\small}] at (0,6) {$\{n_{1},n_{2},n_{4}\}$};
\node[above][font = {\small}] at (-3.5,6) {$\{n_{1},n_{2},n_{3}\}$};
\node[below][font = {\small}] at (0,0) {$\{n_{1},n_{2},n_{3}\}$};
\node[right][font = {\small}] at (1.75,3) {$\{n_{1},n_{3},n_{4}\}$};
\node[font = {\small}] at (-.8,2.8) {$\{n_{2},n_{3},n_{4}\}$};
\node[font = ] at (1.75,2.3) {$F_{n_1}$};
\node at (-.7,5.65) {$F_{n_2}$};
\node at (1.15,2.75) {$F_{n_3}$};
\node at (1.15,3.4) {$F_{n_4}$};
\draw[ultra thin,dash pattern={on 1pt}] (0,6) -- (0,0);
\filldraw[fill=black,draw=black] (2.3,1.15) circle (2pt);
\filldraw[fill=black,draw=black] (-2.3,5.3) circle (2pt);
\node[right][font = {\small}] at (2.3,1.15) {$p_{1}$};
\node[above][font = {\small}] at (-2.3,5.3) {$p_{2}$};
\draw[line width = 0.05mm] (2.3,1.15) -- (-2.3,5.3);
\draw[line width=0.05mm](-2.3,5.3) rectangle (-2.6,1.15);
\draw[line width=0.05mm](2.3,1.15) rectangle (-2.3,.85);
\draw[line width=0.05mm](-2.3,1.15) rectangle (-2.1,1.35);
\draw[ultra thin,dash pattern={on 1pt}] (-1.75,3) -- (-2.3,3);
\draw[ultra thin] (-2.3,3) -- (-2.6,3);
\node[left][font = {\small}] at (-2.6,4.15) {$\frac{\sqrt{3}}{2}-y_{2}$};
\node[left][font = {\small}] at (-2.6,2.1) {$\frac{\sqrt{3}}{2}-y_{1}$};
\node[below][font = {\small}] at (-1.5,.85) {$x_{1}+x_{2}$};
\end{tikzpicture}
    \caption{$L_4(F_{n_1} \rightarrow F_{n_2})$ and $T_O(p_1, p_2)$, where $O=(L_4,F_{n_1},F_{n_2})$}
    \label{tettrail4}
\end{figure}

\begin{figure}[h]
  \centering
\begin{tikzpicture}[scale=1]
\filldraw[fill=black,draw=black] (0,6) circle (2pt);
\filldraw[fill=black,draw=black] (0,0) circle (2pt);
\filldraw[fill=black,draw=black] (-3.5,0) circle (2pt);
\filldraw[fill=black,draw=black] (3.5,6) circle (2pt);
\filldraw[fill=black,draw=black] (1.75,3) circle (2pt);
\filldraw[fill=black,draw=black] (-1.75,3) circle (2pt);
\draw[line width = 0.05mm] (0,0) -- (-3.5,0) -- (-1.75,3) -- (0,0) -- (1.75,3) -- (-1.75,3) -- (0,6) -- (1.75,3) -- (3.5,6) -- (0,6);
\node[below][font = {\small}] at (-3.5,0) {$\{n_{1},n_{2},n_{3}\}$};
\node[above][font = {\small}] at (0,6) {$\{n_{1},n_{2},n_{3}\}$};
\node[above][font = {\small}] at (3.5,6) {$\{n_{1},n_{2},n_{4}\}$};
\node[below][font = {\small}] at (0,0) {$\{n_{1},n_{2},n_{4}\}$};
\node[right][font = {\small}] at (1.75,3) {$\{n_{2},n_{3},n_{4}\}$};
\node[left][font = {\small}] at (-1.75,3) {$\{n_{1},n_{3},n_{4}\}$};
\node at (-1.75,2.2) {$F_{n_1}$};
\node at (2.8,5.65) {$F_{n_2}$};
\node at (-1.1,2.6) {$F_{n_4}$};
\node at (-1.1,3.4) {$F_{n_3}$};
\draw[ultra thin,dash pattern={on 1pt}] (0,6) -- (0,0);
\filldraw[fill=black,draw=black] (1.1,5.4) circle (2pt);
\node[left][font = {\small}] at (1.1,5.4) {$p_{2}$};
\filldraw[fill=black,draw=black] (-1.3,1.25) circle (2pt);
\node[left][font = {\small}] at (-1.3,1.25) {$p_{1}$};
\draw[line width = 0.05mm] (1.1,5.4) -- (-1.3,1.25);
\draw[line width=0.05mm](1.1,5.4) rectangle (1.5,1.25);
\draw[line width=0.05mm](-1.3,1.25) rectangle (1.1,.875);
\draw[line width=0.05mm](1.1,1.25) rectangle (.9,1.45);
\node[right, inner sep=1][font = {\small}] at (1.5,4.7) {$\frac{\sqrt{3}}{2}-y_{2}$};
\node[right, inner sep=1][font = {\small}] at (1.5,2.2) {$\frac{\sqrt{3}}{2}-y_{1}$};
\node[below, inner sep=2][font = {\small}] at (-.9,.875) {$1-x_{1}$};
\node[below, inner sep=2][font = {\small}] at (.9,.875) {$1-x_{2}$};
\draw[ultra thin](0,1.25) -- (0,.875);
\end{tikzpicture}
    \caption{$L_5(F_{n_1} \rightarrow F_{n_2})$ and $T_O(p_1, p_2)$, where $O=(L_5,F_{n_1},F_{n_2})$}
    \label{tettrail5}
\end{figure}
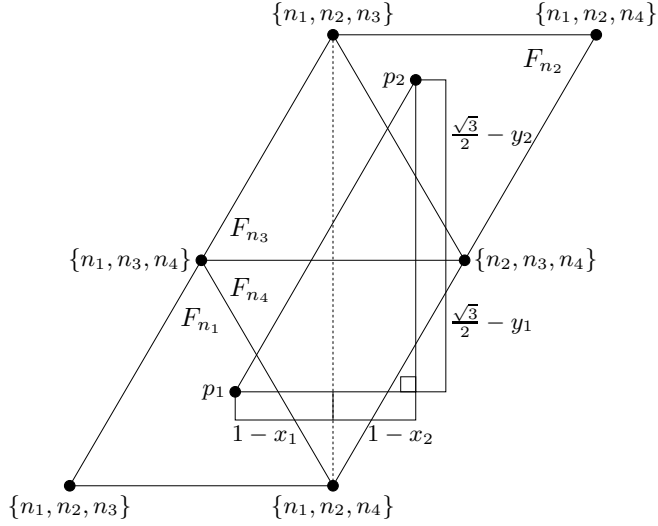

Similarly, in the orientation $(L_5, F_{n_1}, F_{n_2})$ the point $p_2$ has been rotated $180$ degrees about the origin, shifted right $2$ units, and shifted up $\sqrt{3}$ units from $p_2$'s standard position with respect to the representation $(F_{n_2},F_{n_1},x_2,y_2)$.  Due to this $p_2(L_5, F_{n_1}, F_{n_2})=(2-x_2, \sqrt{3}-y_2)$.  Thus we get that
\[\abs{T_{L_5}(p_1, p_2)}=\sqrt{(x_1+x_2-2)^2+(y_1+y_2-\sqrt{3})^2}.\]

One can confirm via a short computation that the following table provides pairs of points which witness the validity of the landscapes $L_1, L_2, L_3, L_4,$ and $L_5$ respectively.  Thus all five landscapes are valid.

\begin{center}
\begin{tabular}[10pt]{|c|c|}
\hline
Landscape & Pair of Points Witnessing Validity \\
\hline
\rule{0pt}{12pt}
$L_1$ & $\left(F_{n_1}, F_{n_2}, \frac{1}{2}, \frac{1}{5}\right), \left(F_{n_2}, F_{n_1}, \frac{1}{2}, \frac{1}{5}\right)$\\[3pt]
\hline
\rule{0pt}{12pt}
$L_2$ & $\left(F_{n_1}, F_{n_2}, \frac{4}{9},  \frac{3}{4}\right), \left(F_{n_2}, F_{n_1}, \frac{5}{9}, \frac{3}{4}\right)$\\[3pt]
\hline
\rule{0pt}{12pt}
$L_3$ & $\left(F_{n_1}, F_{n_2}, \frac{5}{9}, \frac{3}{4}\right), \left(F_{n_2}, F_{n_1},    \frac{4}{9}, \frac{3}{4}\right)$\\[3pt]
\hline
\rule{0pt}{15pt}
$L_4$ & $\left(F_{n_1}, F_{n_2}, \frac{10}{21}, \frac{10\sqrt{3}-1}{20}\right), \left(F_{n_2}, F_{n_1}, \frac{10}{21}, \frac{10\sqrt{3}-1}{20}  \right)$\\[5pt]
\hline
\rule{0pt}{15pt}
$L_5$ & $\left(F_{n_1}, F_{n_2}, \frac{11}{21},  \frac{10\sqrt{3}-1}{20}\right), \left(F_{n_2}, F_{n_1}, \frac{11}{21}, \frac{10\sqrt{3}-1}{20}\right)$\\[5pt]
\hline
\end{tabular}
\end{center}
\end{proof}

Due to the nature of their construction, instead of being viewed through the lens of nets, landscapes can be viewed alternatively as a sequence of faces.  Specifically, given a sequence of distinct faces $\{F_{a_i}\}_{i=1}^k$ of a polyhedron $\mathcal P$, provided $F_{a_{i_1}}$ and $F_{a_{i_2}}$ are adjacent anytime $\abs{i_1-i_2}=1$, then the sequence can be viewed as defining a landscape $L(F_{a_1} \rightarrow F_{a_k})$.  With this in mind, even though for a landscape $L_i \in \{L_i\}_{i=1}^5$ of $\mathcal P_4$ and an identification sending $\{n_1,n_2,n_3,n_4\}$ to $\{1,2,3,4\}$, $L_i$ is defined to have origin face $F_{n_1}$ and destination face $F_{n_2}$, using the following convention we can consider further landscapes with this general form.  If $\{F_{a_i}\}_{i=1}^k$ and $\{F_{b_i}\}_{i=1}^k$ are sequences of faces of $\mathcal P_4$, with $F_{a_1}=F_{n_1}$ and $F_{a_k}=F_{n_2}$, such that in nets of our copy of $\mathcal P_4$ the faces $\{F_{a_i}\}_{i=2}^k$ occur in the same counter-clockwise order around the face $F_{n_1}$ as the faces $\{F_{b_i}\}_{i=2}^k$ occur around the face $F_{b_1}$, $L_i \in \{L_i\}_{i=1}^5$ is the landscape $L_i(F_{n_1} \rightarrow F_{n_2})$ given by the sequence of faces $\{F_{a_i}\}_{i=1}^k$, and $L$ is the landscape given by the sequence of faces $\{F_{b_i}\}_{i=1}^k$, then $L$ can be said to be of the form $L_i(F_{b_1} \rightarrow F_{b_k})$.  As an interesting consequence, the structure given by one landscape can often be considered in terms of a different landscape, for instance $L_2=L_3(F_{n_2} \rightarrow F_{n_1})$.  Furthermore, we will use the same convention in the sections concerning the landscapes of the cube.

Concerning the collection of valid landscapes of the tetrahedron, it might seem counter-intuitive that $L_4$ and $L_5$ are valid.  However, to illustrate why this is the case consider Figure \ref{tetvalid} which helps to identify points which witness the validity of $L_4$.

\begin{figure}[h]
  \centering
\begin{tikzpicture}[scale=1]
\filldraw[fill=black,draw=black] (0,0) circle (2pt);
\filldraw[fill=black,draw=black] (1,1.732) circle (2pt);
\filldraw[fill=black,draw=black] (-1,1.732) circle (2pt);
\filldraw[fill=black,draw=black] (-2,0) circle (2pt);
\filldraw[fill=black,draw=black] (2,0) circle (2pt);
\draw[line width = 0.05mm] (2,0) -- (1,1.732) -- (0,0) -- (-1,1.732) -- (-2,0) -- (0,0) -- cycle;
\draw[line width = 0.05mm] (-1,1.732) -- (1,1.732);
\node[above, inner sep=25] at (0,0) {$F_{n_4}$};
\node[below, inner sep=25] at (1,1.732) {$F_{n_2}$};
\node[below, inner sep=25] at (-1,1.732) {$F_{n_1}$};

\filldraw[fill=black,draw=black] (-4.25,0) circle (2pt);
\filldraw[fill=black,draw=black] (-3.25,1.732) circle (2pt);
\filldraw[fill=black,draw=black] (-5.25,1.732) circle (2pt);
\filldraw[fill=black,draw=black] (-6.25,0) circle (2pt);
\filldraw[fill=black,draw=black] (-2.25,0) circle (2pt);
\draw[line width = 0.05mm] (-2.25,0) -- (-4.25,0) -- (-6.25,0) -- (-5.25,1.732) -- (-4.25,0) -- (-3.25,1.732) -- cycle;
\draw[line width = 0.05mm] (-3.25,1.732) -- (-5.25,1.732);
\node[above, inner sep=25] at (-4.25,0) {$F_{n_3}$};
\node[below, inner sep=25] at (-5.25,1.732) {$F_{n_2}$};
\node[below, inner sep=25] at (-3.25,1.732) {$F_{n_1}$};

\filldraw[fill=black,draw=black] (-6.75,0) circle (2pt);
\filldraw[fill=black,draw=black] (-6.75,3.464) circle (2pt);
\filldraw[fill=black,draw=black] (-5.75,1.732) circle (2pt);
\filldraw[fill=black,draw=black] (-7.75,1.732) circle (2pt);
\draw[line width = 0.05mm] (-6.75,0) -- (-5.75,1.732) -- (-7.75,1.732) -- (-6.75,3.464) -- (-5.75,1.732);
\draw[line width = 0.05mm] (-6.75,0) -- (-7.75,1.732);
\node[above, inner sep=25] at (-6.75,0) {$F_{n_2}$};
\node[below, inner sep=25] at (-6.75,3.464) {$F_{n_1}$};

\filldraw[fill=black,draw=black] (-5.125,-4.464) circle (2pt);
\filldraw[fill=black,draw=black] (-3.125,-4.464) circle (2pt);
\filldraw[fill=black,draw=black] (-5.125,-1) circle (2pt);
\filldraw[fill=black,draw=black] (-7.125,-1) circle (2pt);
\filldraw[fill=black,draw=black] (-6.125,-2.732) circle (2pt);
\filldraw[fill=black,draw=black] (-4.125,-2.732) circle (2pt);
\draw[line width = 0.05mm] (-5.125,-4.464) -- (-3.125,-4.464) -- (-4.125,-2.732) -- (-5.125,-4.464) -- (-6.125,-2.732) -- (-4.125,-2.732) -- (-5.125,-1) -- (-6.125,-2.732) -- (-7.125,-1) -- (-5.125,-1);
\node[above, inner sep=25] at (-5.125,-4.464) {$F_{n_3}$};
\node[below, inner sep=25] at (-4.125,-2.732) {$F_{n_1}$};
\node[below, inner sep=25] at (-5.125,-1) {$F_{n_4}$};
\node[above, inner sep=25] at (-6.125,-2.732) {$F_{n_2}$};

\filldraw[fill=black,draw=black] (-2.625,-4.464) circle (2pt);
\filldraw[fill=black,draw=black] (-0.625,-4.464) circle (2pt);
\filldraw[fill=black,draw=black] (-0.625,-1) circle (2pt);
\filldraw[fill=black,draw=black] (1.375,-1) circle (2pt);
\filldraw[fill=black,draw=black] (0.375,-2.732) circle (2pt);
\filldraw[fill=black,draw=black] (-1.625,-2.732) circle (2pt);
\draw[line width = 0.05mm] (-2.625,-4.464) -- (-0.625,-4.464) -- (-1.625,-2.732) -- (-2.625,-4.464) --  (-1.625,-2.732) -- (-0.625,-1) -- (0.375,-2.732) -- (1.375,-1) -- (-0.625,-1);
\draw[line width = 0.05mm] (-1.625,-2.732) -- (0.375,-2.732) -- (-0.625,-4.464);
\node[above, inner sep=25] at (-0.625,-4.464) {$F_{n_4}$};
\node[below, inner sep=25] at (-1.625,-2.732) {$F_{n_1}$};
\node[below, inner sep=25] at (-0.625,-1) {$F_{n_3}$};
\node[above, inner sep=25] at (0.375,-2.732) {$F_{n_2}$};

\node[below, inner sep=8] at (-6.75,0) {$L_{1}$};
\node[below, inner sep=8] at (-4.25,0) {$L_{2}$};
\node[below, inner sep=8] at (0,0) {$L_{3}$};
\node[below, inner sep=8] at (-4.125,-4.464) {$L_{4}$};
\node[below, inner sep=8] at (-1.625,-4.464) {$L_{5}$};

\fill[semitransparent, black!40!white] (-5.125,-2.732) circle (1cm);
\fill[semitransparent, black!40!white] (1,1.732) arc (360:300:1cm);
\fill[semitransparent, black!40!white] (-3.25,1.732) arc (360:300:1cm);
\fill[semitransparent, black!40!white] (-1,1.732) arc (360:300:1cm);
\fill[semitransparent, black!40!white] (-5.25,1.732) arc (360:300:1cm);
\fill[semitransparent, black!40!white] (-6.25,.866) arc (120:180:1cm);
\fill[semitransparent, black!40!white] (-6.75,3.464) arc (360:300:1cm);
\fill[semitransparent, black!40!white] (-1.625,-2.732) arc (360:300:1cm);
\fill[semitransparent, black!40!white] (.875,-1.866) arc (120:180:1cm);
\draw[ultra thin] (-5.125,-2.732) circle (1cm);
\draw[ultra thin] (1,1.732) arc (360:300:1cm);
\draw[ultra thin] (-3.25,1.732) arc (360:300:1cm);
\draw[ultra thin] (-1,1.732) arc (360:300:1cm);
\draw[ultra thin] (-5.25,1.732) arc (360:300:1cm);
\draw[ultra thin] (-6.25,.866) arc (120:180:1cm);
\draw[ultra thin] (-6.75,3.464) arc (360:300:1cm);
\draw[ultra thin] (-1.625,-2.732) arc (360:300:1cm);
\draw[ultra thin] (.875,-1.866) arc (120:180:1cm);

\end{tikzpicture}
    \caption{Points Witnessing the Validity of $L_4$}
    \label{tetvalid}
\end{figure}
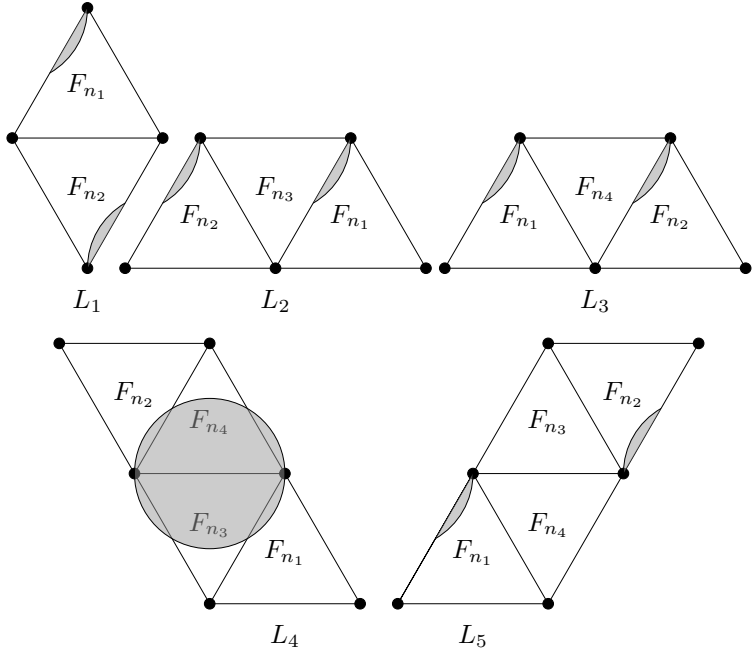

Since $\mathcal P_4$ is a unit tetrahedron, the shaded disc illustrated in $L_4$ has diameter 1, and thus all points inside of it are less than distance 1 from each other.  The regions within this disc are illustrated in each of the other four landscapes.  First consider $L_1$ and $L_5$, it is easy to see that one can identify points $p_1=(F_{n_1}, F_{n_2}, x_1, y_1)$ and $p_2=(F_{n_2}, F_{n_1}, x_2, y_2)$ further than distance one from each other in $L_1$ and $L_5$ by requiring that $p_1$ and $p_2$ are interior to the portions of the disc in $F_{n_1}$ and $F_{n_2}$ respectively.  Now, consider $L_2$ and $L_3$, again it is easy to see that one can ensure $p_1$ and $p_2$ are at least distance one from each other in both $L_2$ and $L_3$ by requiring $x_1=x_2$.  Thus, $L_4$ will be the only landscape for which any $p_1$ and $p_2$ meeting these requirements are less than distance one from each other, and thus $L_4$ is by definition valid.  Due to the similarities between $L_4$ and $L_5$, it is now also clear why $L_5$ is valid.      

Having defined the valid landscapes of $\mathcal P_4$,  one can also quickly determine the explicit formulae for the exact set of points on the surface of $\mathcal P_4$ along any given trail.  For example suppose $p_1$ and $p_2$ were as given in the previous paragraph, and one wished to determine the set of points comprising $f\left(T_O(p_1, p_2)\right)$ where $O=(L_4,F_{n_1},F_{n_2})$, and more specifically $f\left({}_{n_1}T_O(p_1, p_2)\right)$, $f\left({}_{n_2}T_O(p_1, p_2)\right)$, $f\left({}_{n_3}T_O(p_1, p_2)\right)$, and $f\left({}_{n_4}T_O(p_1, p_2)\right)$ where $f$ is the folding function $f: L_4 \rightarrow \mathcal P_4$.

\indent Since $p_2(L_4, F_{n_1}, F_{n_2})=(-x_2, \sqrt{3}-y_2)$, it is clear that the slope of $T_O(p_1, p_2)$ is $m_4=\frac{y_1+y_2-\sqrt{3}}{x_1+x_2}$ and that points along this trail in face $F_{n_1}$ must satisfy $y=m_4(x-x_1)+y_1$, when $m_4$ is defined.  Furthermore, since $p_1$ and $p_2$ have been chosen to be interior to the disc, $x_1, x_2>0$, and thus $m_4$ is defined.  Since the portion of this trail in $F_{n_1}$ is bounded on one end by $p_1$ and on the other by $\overline{(0, 0)( \frac{1}{2}  ,  \frac{\sqrt{3}}{2}  )}$ (a line segment satisfying the equation $y=\sqrt{3}x$), we have that
\[f\left({}_{n_1}T_O(p_1, p_2)\right)=\left\{(F_{n_1}, F_{n_2}, x, y) \in P_4: y=m_4(x-x_1)+y_1 \text{~and~} y_1 \leq y \leq \sqrt{3}x\right\}.\]

Now considering $(L_4, F_{n_2}, F_{n_1})$, to determine $f\left({}_{n_2}T_O(p_1, p_2)\right)$, we identify that 180 degree rotations preserve the slopes of lines, so the slope remains $m_4$.  Furthermore, the portion of this trail in $F_{n_2}$ is bounded on one end by $p_2$ and on the other by  $\overline{(0, 0)( \frac{1}{2}  ,  \frac{\sqrt{3}}{2}  )}$.  So we have that 
\[f\left({}_{n_2}T_O(p_1, p_2)\right)=\left\{(F_{n_2}, F_{n_1}, x, y) \in P_4: y=m_4(x-x_2)+y_2 \text{~and~} y_2 \leq y \leq \sqrt{3}x\right\}.\]

To determine $f\left({}_{n_3}T_O(p_1, p_2)\right)$, note that in the orientation $(L_4, F_{n_3}, F_{n_4})$ the face $F_{n_1}$ has been rotated 180 degrees about the origin, shifted right $\frac{1}{2}$ unit, and shifted up $\frac{\sqrt{3}}{2}$ units.  Thus $p_1(L_4, F_{n_3}, F_{n_4})=( \frac{1}{2}  -x_1,  \frac{\sqrt{3}}{2}  -y_1)$, $m_4$ remains the slope, and the portion of this trail in $F_{n_3}$ is bounded on one end by $\overline{(0, 0)( \frac{1}{2}  ,  \frac{\sqrt{3}}{2}  )}$ and on the other by $\overline{(0,0)(1, 0)}$.  Finally, to determine $f\left({}_{n_4}T_O(p_1, p_2)\right)$, note that it is determined in a way analogous to $f\left({}_{n_3}T_O(p_1, p_2)\right)$ and thus we have
\[f\left({}_{n_3}T_O(p_1, p_2)\right)=\left\{(F_{n_3}, F_{n_4}, x, y) \in P_4: y=m_4(x-  \frac{1}{2}  +x_1)+ \frac{\sqrt{3}}{2}  -y_1 \text{~and~} 0 \leq y \leq \sqrt{3}x \right\} \text{~and}\]
\[f\left({}_{n_4}T_O(p_1, p_2)\right)=\left\{(F_{n_4}, F_{n_3}, x, y) \in P_4: y=m_4(x-  \frac{1}{2}  +x_2)+ \frac{\sqrt{3}}{2}  -y_2 \text{~and~} 0 \leq y \leq \sqrt{3}x \right\}.\]

It is important to note that during the preceding two sections all results and definitions except for Lemma \ref{tet-rep} and Theorem \ref{tet-valid} applied to every member of the family of convex unit polyhedra.  Thus, to obtain analogous results for any other convex unit polyhedron, one simply must develop and prove the analogs of Lemma \ref{tet-rep} and Theorem \ref{tet-valid}.  Having accomplished our goal with the tetrahedron we now turn to obtaining said results for the more complicated case provided by the cube.

\section{A Coordinate System on a Cube}

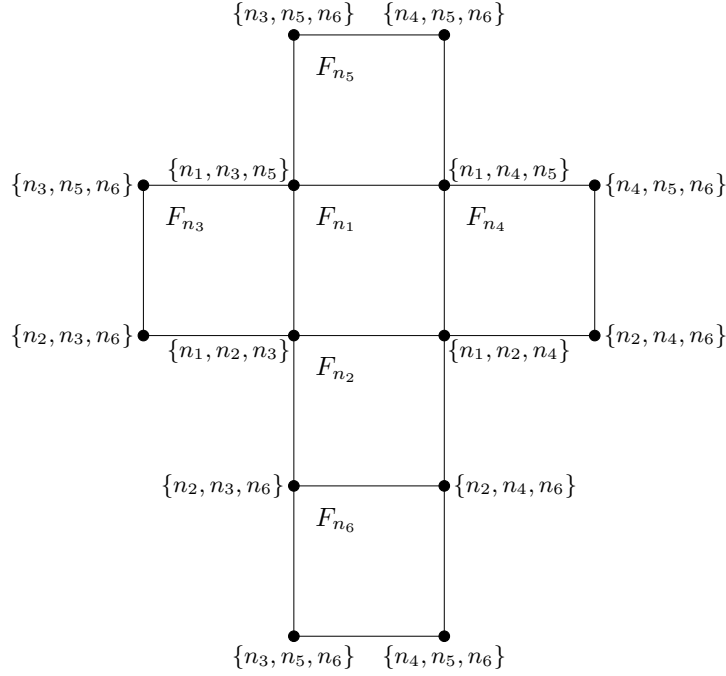
\begin{figure}[h]
  \centering
\begin{tikzpicture}[scale=1]
\draw [line width=0.05mm] (0,0) rectangle (2,8);
\draw[line width = 0.05mm] (-2,6) -- (0,6) -- (2,6) -- (4,6) -- (4,4) -- (2,4) -- (0,4) -- (-2,4) -- cycle;
\draw[line width = 0.05mm] (2,2) -- (0,2);
\filldraw[fill=black,draw=black] (0,0) circle (2pt);
\filldraw[fill=black,draw=black] (2,0) circle (2pt);
\filldraw[fill=black,draw=black] (2,2) circle (2pt);
\filldraw[fill=black,draw=black] (0,2) circle (2pt);
\filldraw[fill=black,draw=black] (0,4) circle (2pt);
\filldraw[fill=black,draw=black] (0,6) circle (2pt);
\filldraw[fill=black,draw=black] (0,8) circle (2pt);
\filldraw[fill=black,draw=black] (2,4) circle (2pt);
\filldraw[fill=black,draw=black] (2,6) circle (2pt);
\filldraw[fill=black,draw=black] (2,8) circle (2pt);
\filldraw[fill=black,draw=black] (-2,4) circle (2pt);
\filldraw[fill=black,draw=black] (-2,6) circle (2pt);
\filldraw[fill=black,draw=black] (4,6) circle (2pt);
\filldraw[fill=black,draw=black] (4,4) circle (2pt);
\node[below right, inner sep=8] at (0,2) {$F_6$};
\node[below right, inner sep=8] at (0,4) {$F_2$};
\node[below right, inner sep=8] at (0,6) {$F_1$};
\node[below right, inner sep=8] at (-2,6) {$F_3$};
\node[below right, inner sep=8] at (2,6) {$F_4$};
\node[below right, inner sep=8] at (0,8) {$F_5$};
\node[left][font = {\small}] at (-2,6) {$\{3,5,6\}$};
\node[above left, inner sep=1][font = {\small}] at (0,6) {$\{1,3,5\}$};
\node[left][font = {\small}] at (-2,4) {$\{2,3,6\}$};
\node[above][font = {\small}] at (0,8) {$\{3,5,6\}$};
\node[above][font = {\small}] at (2,8) {$\{4,5,6\}$};
\node[above right, inner sep=1][font = {\small}] at (2,6) {$\{1,4,5\}$};
\node[right][font = {\small}] at (4,6) {$\{4,5,6\}$};
\node[right][font = {\small}] at (4,4) {$\{2,4,6\}$};
\node[below right, inner sep=1][font = {\small}] at (2,4) {$\{1,2,4\}$};
\node[right][font = {\small}] at (2,2) {$\{2,4,6\}$};
\node[below][font = {\small}] at (2,0) {$\{4,5,6\}$};
\node[below][font = {\small}] at (0,0) {$\{3,5,6\}$};
\node[left][font = {\small}] at (0,2) {$\{2,3,6\}$};
\node[below left, inner sep=1][font = {\small}] at (0,4) {$\{1,2,3\}$};
\end{tikzpicture}
    \caption{Labeling of the Cube}
    \label{cubenet}
\end{figure}

Analogous to our exploration of the tetrahedron, in order to have a discussion concerning paths on the surface of the cube, we must first fix a net for the cube. 
Without loss of generality, we assume the edges of our cube are of length 1. Furthermore, we will label the faces with elements of the set $\{F_1, F_2, F_3, F_4, F_5, F_6\}$ and the vertices will be represented by elements of the set 
\[\left\{\{1,2,3\}, \{1,2,4\}, \{1,3,5\}, \{1,4,5\}, \{2,3,6\}, \{2,4,6\}, \{3,5,6\}, \{4,5,6\}\right\}\] 
such that face $F_n$ is incident to vertex $\{n_1,n_2,n_3\}$ if and only if $n \in \{n_1,n_2,n_3\}$. Consider the labeling of the net in Figure \ref{cubenet} fixed for the remainder of our discussion on the cube.

As with the tetrahedron, points on the surface of the cube will have multiple representations.  First, if a point is a vertex of the cube or lies on an edge of the cube it may have multiple home-faces.  In particular, the vertex $\{n_1,n_2,n_3\}$ can be represented as $(F_{n_1}, F_{n_2}, 0, 0)$, $(F_{n_2}, F_{n_3}, 0, 0)$, and $(F_{n_3}, F_{n_1}, 0, 0)$, with the remaining vertices of the cube similarly having such representations. Next note, if a point lies on an edge of $\mathcal P_6$ say, $\overline {\{n_1, n_2, n_4\} \{n_1, n_2, n_3\}}$, and is represented as $(F_{n_1}, F_{n_2}, x, 0)$ then it can also be represented as $(F_{n_2}, F_{n_1}, 1-x, 0)$.  While this addresses one's ability to switch between representations using different home-faces for a point which is incident to more than one face, we still need to address how one switches between the representations of a point using different shared-faces, and so we introduce the following lemma.

\begin{lemma}\label{Cube_Shared_Face}
Suppose $\{n_1,n_2,n_3,n_4,n_5,n_6\}$ $=$ $\{1,2,3,4,5,6\}$ such that in nets of our copy of $\mathcal P_6$ $F_{n_1}$ and $F_{n_6}$ are opposite faces and $F_{n_2}$, $F_{n_3}$, $F_{n_4}$, and $F_{n_5}$ occur about $F_{n_1}$ in the same counter-clockwise order as $F_2$, $F_3$, $F_4$, and $F_5$ occur about $F_1$. Then if a point $p\in\mathcal P_6$ can be represented as $(F_{n_1}, F_{n_2}, x, y)$, then it can also be represented as $(F_{n_1},F_{n_4},y,1-x)$.
\end{lemma}

\begin{figure}[h]
  \centering
\begin{tikzpicture}[scale=1]
\draw [line width=0.05mm] (0,0) rectangle (3.5,3.5);
\draw [line width=0.05mm] (-1.75,0) -- (5.25,0);
\draw [line width=0.05mm] (-1.75,3.5) -- (5.25,3.5);
\draw [line width=0.05mm] (0,-1.75) -- (0,5.25);
\draw [line width=0.05mm] (3.5,-1.75) -- (3.5,5.25);
\filldraw[fill=black,draw=black] (0,0) circle (2pt);
\filldraw[fill=black,draw=black] (0,3.5) circle (2pt);
\filldraw[fill=black,draw=black] (3.5,3.5) circle (2pt);
\filldraw[fill=black,draw=black] (3.5,0) circle (2pt);
\filldraw[fill=black,draw=black] (2.5,1.6) circle (2pt);
\node[above right, inner sep=12] at (0,3.5) {$F_{n_5}$};
\node[above left][font = {\footnotesize}] at (0,3.5) {$\{n_{1},n_{3},n_{5}\}$};
\node[above right][font = {\small}] at (3.5,3.5) {$\{n_{1},n_{4},n_{5}\}$};
\node[below right, inner sep=12][font = {\small}] at (0,3.5) {$F_{n_1}$};
\node[below right, inner sep=12][font = {\small}] at (3.5,3.5) {$F_{n_4}$};
\node[below left][font = {\small}] at (0,0) {$\{n_{1},n_{2},n_{3}\}$};
\node[above left, inner sep=12] at (0,0) {$F_{n_3}$};
\node[below right, inner sep=12] at (0,0) {$F_{n_2}$};
\node[below right][font = {\small}] at (3.5,0) {$\{n_{1},n_{2},n_{4}\}$};
\draw[line width=0.05mm] (3.5,0) rectangle (2.5,1.6);
\node[above][font = {\small}] at (3,1.6) {$1-x$};
\node[left][font = {\small}] at (2.5,.8) {$y$};
\node[below][font = {\small}] at (1,0) {$x$};
\node[above left][font = {\small}] at (2.5,1.6) {$p$};
\end{tikzpicture}
    \caption{Representations of a point $p$ with Different Shared-Faces}
    \label{cube_coordinates}
\end{figure}
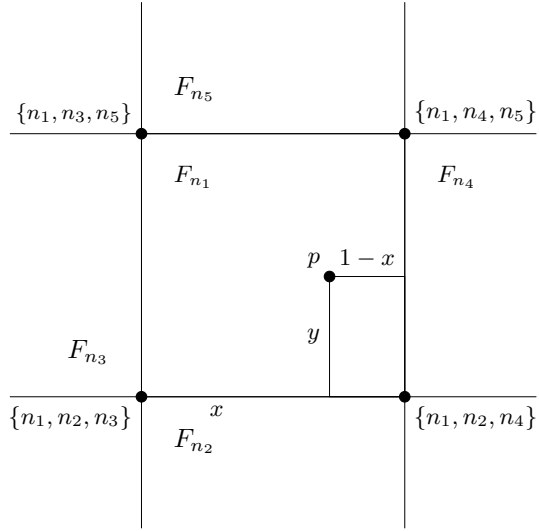

\begin{proof}
As mentioned previously, $p$ has a representation for which $F_{n_4}$ is used as the shared face.  Suppose $(u,v) \in \mathbb R^2$ is such that $p=(F_{n_1}, F_{n_4}, u, v)$.  Clearly, the distance from $\overline{\{n_1,n_2,n_4\}\{n_1,n_2,n_3\}}$ to $p$ will be $u = y$.  Also, as seen in Figure \ref{cube_coordinates}, the distance from $ \overline{\{n_1,n_2,n_4\}\{n_1,n_4,n_5\}}$ to $p$ is $v = 1 - x$, since $\{n_1,n_2,n_3\}\{n_1,n_2,n_4\} = 1$.   
\end{proof}

\section{Landscapes of a Cube}

Applying our notion of a landscape defined in section 2, we develop the necessary landscapes of a cube that will eventually lead us to the desired set of valid landscapes. Once we have identified the set of valid landscapes, we will also, as in our discussion of the tetrahedron, have identified the set of trails of minimum distance for each pair of points on the surface of the cube. 

In the following two theorems, we will only be concerned with certain landscapes of $\mathcal P_6$, in particular none of these landscapes will have dual-graphs with more than four vertices. In Theorem \ref{Cube_Adjacent_Face} we will construct landscapes which arise when we consider the trail between two points which lie on adjacent faces. Likewise, in Theorem \ref{Cube_Opposite_Face} we will construct landscapes which arise when we consider the trail between two points which lie on opposite faces. In each case, to do so, we will fix a first and final face and consider the faces which can lie on the trail between the two. For each landscape we will also determine the length of these trails and these results will be presented as Corollary \ref{cube_adjacent_surface} and Corollary \ref{cube_opposite_surface}. Additionally, to allow for easy reference in the future, we will name the fifteen landscapes $L_1$, $L_2$,..., $L_{15}$ and will refer to them as such for the remainder of our discussion. Later on, Theorem \ref{finaltheorem} shows that given a pair of adjacent faces, $L_1$, $L_2$, and $L_3$ are the only valid landscapes of $\mathcal P_6$ between these two faces; and likewise given a pair of opposite faces, $L_4$, $L_5$, ... $L_{15}$ are the only valid landscapes of $\mathcal P_6$ between these two faces.  The most important tool in reaching this final conclusion is Lemma \ref{boundary}, which provides a way to simplify the calculations necessary for showing that a given landscape is not valid.  It is worth noting that Lemma \ref{boundary} applies to the landscapes of any convex unit polyhedron.

We first address the case of landscapes between adjacent faces.  It is worth noting that Figure \ref{cubeorientation1}, Figure \ref{cubeorientation2}, and Figure \ref{cubeorientation3} will each be referenced for both the proof of Theorem \ref{Cube_Adjacent_Face} and the proof of Corollary \ref{cube_adjacent_surface}, with the structure of each landscape being of particular interest in Theorem \ref{Cube_Adjacent_Face} and the length of the trail contained therein being of interest in Corollary \ref{cube_adjacent_surface}.

\begin{theorem}\label{Cube_Adjacent_Face}
Let $F_{n_1}$ and $F_{n_2}$ be adjacent faces of $\mathcal P_6$.  Then there are at least 3 landscapes of the form $L_i(F_{n_1} \rightarrow F_{n_2})$.
\end{theorem}

\setcounter{case}{0}

\begin{proof}
Suppose $\{n_1,n_2,n_3,n_4,n_5,n_6\}$ = $\{1,2,3,4,5,6\}$ such that in nets of our copy of $\mathcal P_6$, $F_{n_1}$ and $F_{n_6}$ are opposite faces, and $F_{n_2}$, $F_{n_3}$, $F_{n_4}$, and $F_{n_5}$ occur in the same counter-clockwise order about $F_{n_1}$ as the faces $F_2$, $F_3$, $F_4$, and $F_5$ occur about the face $F_1$. 

\begin{figure}[h]
  \centering
\begin{tikzpicture}[scale=1]
\draw [line width=0.05mm] (0,0) rectangle (3,6);
\filldraw[fill=black,draw=black] (0,0) circle (1.5pt);
\filldraw[fill=black,draw=black] (0,3) circle (1.5pt);
\filldraw[fill=black,draw=black] (3,3) circle (1.5pt);
\filldraw[fill=black,draw=black] (3,6) circle (1.5pt);
\filldraw[fill=black,draw=black] (3,0) circle (1.5pt);
\filldraw[fill=black,draw=black] (0,6) circle (1.5pt);
\node[above][font = {\small}] at (0,6) {$\{n_{1},n_{3},n_{5}\}$};
\node[above][font = {\small}] at (3,6) {$\{n_{1},n_{4},n_{5}\}$};
\node[left][font = {\small}] at (0,3) {$\{n_{1},n_{2},n_{3}\}$};
\node[right][font = {\small}] at (3,3) {$\{n_{1},n_{2},n_{4}\}$};
\node[below][font = {\small}] at (0,0) {$\{n_{2},n_{3},n_{6}\}$};
\node[below][font = {\small}] at (3,0) {$\{n_{2},n_{4},n_{6}\}$};
\node[below left, inner sep=8] at (3,3) {$F_{n_2}$};
\node[above left, inner sep=8] at (3,3) {$F_{n_1}$};
\draw[line width=0.01mm] (0,3) -- (3,3);
\filldraw[fill=black,draw=black] (2.4,5) circle (1.5pt);
\filldraw[fill=black,draw=black] (.6,1.2) circle (1.5pt);
\draw[line width=0.05mm] (2.4,5) -- (.6,1.2);
\node[right][font = {\small}] at (2.4,5) {$p_{1}$};
\node[right][font = {\small}] at (.6,1.2) {$p_{2}$};
\draw [line width=0.05mm] (2.4,5) rectangle (.6,5.2);
\draw [line width=0.05mm] (.6,1.2) rectangle (.4,5);
\draw [line width=0.05mm] (.6,5) rectangle (.8,4.8);
\node[left,inner sep=1][font = {\small}] at (.4,2) {$y_{2}$};
\node[left,inner sep=1][font = {\small}] at (.4,4) {$y_{1}$};
\node[above, inner sep=1][font = {\small}] at (1.5,5.2) {$|1-x_{1}-x_{2}|$};
\end{tikzpicture}
    \caption{$L_1(F_{n_1} \rightarrow F_{n_2})$ and $T_O(p_1, p_2)$, where $O=(L_1,F_{n_1},F_{n_2})$}
    \label{cubeorientation1}
\end{figure}
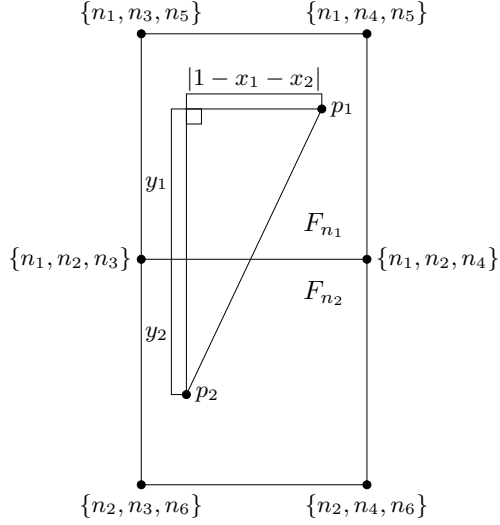

\begin{case}
The dual-graph of the landscape is a path of two vertices.
\end{case}

Since we are constructing landscapes from $F_{n_1}$ to $F_{n_2}$, and $F_{n_1}$ and $F_{n_2}$ have only one edge in common, the only such landscape is that given in Figure \ref{cubeorientation1}, which we will refer to as $L_1$.

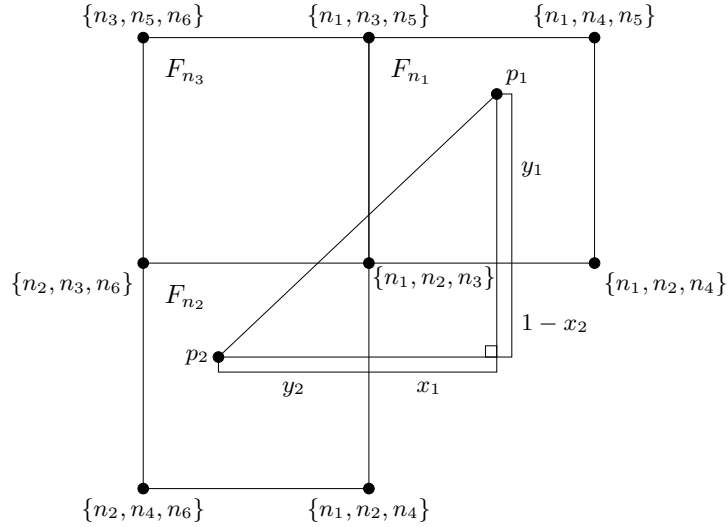
\begin{figure}[h]
  \centering
\begin{tikzpicture}[scale=1]
\draw [line width=0.05mm] (0,0) rectangle (3,6);
\draw [line width=0.05mm] (3,3) rectangle (6,6);
\filldraw[fill=black,draw=black] (0,0) circle (2pt);
\filldraw[fill=black,draw=black] (0,3) circle (2pt);
\filldraw[fill=black,draw=black] (3,3) circle (2pt);
\filldraw[fill=black,draw=black] (3,6) circle (2pt);
\filldraw[fill=black,draw=black] (0,6) circle (2pt);
\filldraw[fill=black,draw=black] (6,3) circle (2pt);
\filldraw[fill=black,draw=black] (6,6) circle (2pt);
\filldraw[fill=black,draw=black] (3,0) circle (2pt);
\draw[line width=0.05mm](0,3) -- (3,3);
\node[below right, inner sep=8] at (0,3) {$F_{n_2}$};
\node[below right, inner sep=8] at (0,6) {$F_{n_3}$};
\node[below right, inner sep=8] at (3,6) {$F_{n_1}$};
\node[above][font = {\small}] at (0,6) {$\{n_{3},n_{5},n_{6}\}$};
\node[above][font = {\small}] at (3,6) {$\{n_{1},n_{3},n_{5}\}$};
\node[above][font = {\small}] at (6,6) {$\{n_{1},n_{4},n_{5}\}$};
\node[below right][font = {\small}] at (6,3) {$\{n_{1},n_{2},n_{4}\}$};
\node[below right, inner sep=1][font = {\small}] at (3,3) {$\{n_{1},n_{2},n_{3}\}$};
\node[below][font = {\small}] at (3,0) {$\{n_{1},n_{2},n_{4}\}$};
\node[below][font = {\small}] at (0,0) {$\{n_{2},n_{4},n_{6}\}$};
\node[below left][font = {\small}] at (0,3) {$\{n_{2},n_{3},n_{6}\}$};
\filldraw[fill=black,draw=black] (1,1.75) circle (2pt);
\filldraw[fill=black,draw=black] (4.7,5.25) circle (2pt);
\draw[line width = 0.05mm] (1,1.75) -- (4.7,5.25);
\node[above right][font = {\small}] at (4.7,5.25) {$p_{1}$};
\node[left][font = {\small}] at (1,1.75) {$p_{2}$};
\draw [line width=0.05mm] (1,1.75) rectangle (4.7,1.55);
\draw [line width=0.05mm] (4.7,5.25) rectangle (4.9,1.75);
\draw [line width=0.05mm] (4.7,1.75) rectangle (4.55,1.9);
\node[right][font = {\small}] at (4.9,4.25) {$y_{1}$};
\node[right][font = {\small}] at (4.9,2.2) {$1-x_{2}$};
\node[below][font = {\small}] at (2,1.55) {$y_{2}$};
\node[below][font = {\small}] at (3.8,1.55) {$x_{1}$};
\end{tikzpicture}
    \caption{$L_2(F_{n_1} \rightarrow F_{n_2})$ and $T_O(p_1, p_2)$, where $O = (L_2, F_{n_1}, F_{n_2})$}
    \label{cubeorientation2}
\end{figure}

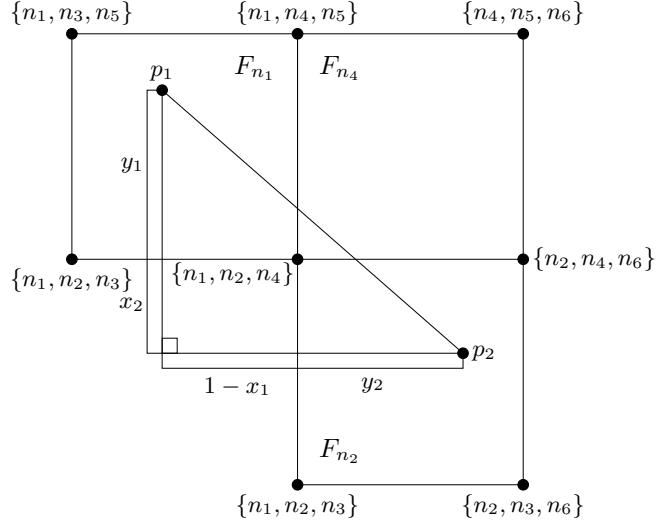
\begin{figure}[h]
  \centering
\begin{tikzpicture}[scale=1]
\draw [line width=0.05mm] (3,0) rectangle (6,6);
\draw [line width=0.05mm] (3,3) rectangle (0,6);
\filldraw[fill=black,draw=black] (6,0) circle (2pt);
\filldraw[fill=black,draw=black] (0,3) circle (2pt);
\filldraw[fill=black,draw=black] (3,3) circle (2pt);
\filldraw[fill=black,draw=black] (3,6) circle (2pt);
\filldraw[fill=black,draw=black] (0,6) circle (2pt);
\filldraw[fill=black,draw=black] (6,3) circle (2pt);
\filldraw[fill=black,draw=black] (6,6) circle (2pt);
\filldraw[fill=black,draw=black] (3,0) circle (2pt);
\draw[line width=0.05mm](6,3) -- (3,3);
\node[above right, inner sep=8] at (3,0) {$F_{n_2}$};
\node[below right, inner sep=8] at (3,6) {$F_{n_4}$};
\node[below left, inner sep=8] at (3,6) {$F_{n_1}$};
\node[above][font = {\small}] at (6,6) {$\{n_{4},n_{5},n_{6}\}$};
\node[above][font = {\small}] at (0,6) {$\{n_{1},n_{3},n_{5}\}$};
\node[above][font = {\small}] at (3,6) {$\{n_{1},n_{4},n_{5}\}$};
\node[right][font = {\small}] at (6,3) {$\{n_{2},n_{4},n_{6}\}$};
\node[below left, inner sep=1][font = {\small}] at (3,3) {$\{n_{1},n_{2},n_{4}\}$};
\node[below][font = {\small}] at (3,0) {$\{n_{1},n_{2},n_{3}\}$};
\node[below][font = {\small}] at (6,0) {$\{n_{2},n_{3},n_{6}\}$};
\node[below][font = {\small}] at (0,3) {$\{n_{1},n_{2},n_{3}\}$};
\filldraw[fill=black,draw=black] (5.2,1.75) circle (2pt);
\filldraw[fill=black,draw=black] (1.2,5.25) circle (2pt);
\draw[line width = 0.05mm] (5.2,1.75) -- (1.2,5.25);
\node[above][font = {\small}] at (1.2,5.25) {$p_{1}$};
\node[right][font = {\small}] at (5.2,1.75) {$p_{2}$};
\draw [line width=0.05mm] (5.2,1.75) rectangle (1.2,1.55);
\draw [line width=0.05mm] (1.2,5.25) rectangle (1,1.75);
\draw [line width=0.05mm] (1.2,1.75) rectangle (1.4,1.95);
\node[left, inner sep=1][font = {\small}] at (1,4.25) {$y_{1}$};
\node[left, inner sep=1][font = {\small}] at (1,2.4) {$x_{2}$};
\node[below][font = {\small}] at (4,1.55) {$y_{2}$};
\node[below][font = {\small}] at (2.2,1.55) {$1-x_{1}$};
\end{tikzpicture}
    \caption{$L_3(F_{n_1} \rightarrow F_{n_2})$ and $T_O(p_1, p_2)$, where $O = (L_3,F_{n_1}, F_{n_2})$}
    \label{cubeorientation3}
\end{figure}

\begin{case}
    The dual-graph of the landscape is a path of three vertices.
\end{case}

Since we are constructing landscapes of the form $L_i(F_{n_1} \rightarrow F_{n_2})$ and any two adjacent faces of $\mathcal P_6$ share two neighbors, the first face must be $F_{n_1}$, the third face $F_{n_2}$, and the second face is either $F_{n_3}$ or $F_{n_4}$ (with $F_{n_6}$ and $F_{n_5}$ being the opposite faces of $F_{n_1}$ and $F_{n_2}$ respectively).  Since two faces of $\mathcal P_6$ share a unique edge, the two landscapes in Figure \ref{cubeorientation2} and Figure \ref{cubeorientation3}  are the only such landscapes. 
\end{proof}

Having identified a collection of landscapes of $\mathcal P_6$, we have the following corollary.  Note that due to the symmetries of $\mathcal P_6$ and the presence of a formula for switching from one shared-face to another we can assume without loss of generality that the representations of the points $p_1$ and $p_2$ are such that $p_1=(F_{n_1}, F_{n_2}, x_1, y_1)$ and $p_2=(F_{n_2}, F_{n_1}, x_2, y_2)$.  Later when we prove that given a pair of adjacent faces these are the only three valid landscapes between said adjacent faces, we will see that the bound in this corollary actually provides equality.    

\begin{corollary}\label{cube_adjacent_surface}
Given two points $p_1 \in F_{n_1} \setminus F_{n_2}$ and $p_2 \in F_{n_2} \setminus F_{n_1}$ on two adjacent faces of the cube, with $p_1=(F_{n_1},F_{n_2},x_1,y_1)$ and $p_2=(F_{n_2},F_{n_1},x_2,y_2)$,
\[d_{\mathcal P_6}(p_1,p_2) \leq d_{\mathcal P_6}^A(p_1,p_2),\text{ where}\] 
\[d_{\mathcal P_6}^A(p_1,p_2)=\min\left\{\abs{T_{L_i}(p_1,p_2)}: i \in \mathbb N \text{ with } 1 \leq i \leq 3\right\},\]
and for each $i \in \mathbb N$ with $1 \leq i \leq 3$, the trail length $\abs{T_{L_i}(p_1,p_2)}$ is given in the table below:

\begingroup
\begin{center}
\begin{tabular}{|c|c|}
    \hline
    Landscape & Trail Length Formula\\
    \hline
    \rule{0pt}{12pt}$L_1$ &\rule{-4pt}{12pt} $\sqrt{(x_1+x_2-1)^2+(y_1+y_2)^2}$\\[5pt]
    \hline
    \rule{0pt}{12pt}$L_2$ &\rule{-4pt}{12pt} $\sqrt{(x_1+y_2)^2+(y_1-x_2+1)^2}$\\[5pt]
    \hline
    \rule{0pt}{12pt}$L_3$ &\rule{-4pt}{12pt} $\sqrt{(x_1-y_2-1)^2+(y_1+x_2)^2}$\\[5pt]
    \hline
\end{tabular}
\end{center}
\endgroup
\end{corollary}

\begin{proof}
Suppose $\{n_1,n_2,n_3,n_4,n_5,n_6\}$ = $\{1,2,3,4,5,6\}$ such that in nets of our copy of $\mathcal P_6$, $F_{n_1}$ and $F_{n_6}$ are opposite faces, and $F_{n_2}$, $F_{n_3}$, $F_{n_4}$, and $F_{n_5}$ occur in the same counter-clockwise order about $F_{n_1}$ as the faces $F_2$, $F_3$, $F_4$, and $F_5$ occur about the face $F_1$.  Since $p_2$'s home-face is $F_{n_2}$ and $p_2$'s shared-face is $F_{n_1}$, in the orientation $(L_1,F_{n_1},F_{n_2})$ the point $p_2$ has been rotated $180$ degrees about the origin and shifted $1$ unit to the right from $p_2$'s standard position with respect to the representation $(F_{n_2},F_{n_1},x_2,y_2)$. Thus, $p_2(L_1,F_{n_1},F_{n_2})$ $=$ ($1-x_2$, $-y_2$) and we get that 
\[\abs{T_{L_1}(p_1, p_2)}=\sqrt{(x_1+x_2-1)^2+(y_1+y_2)^2}.\]
Since $p_2$'s home-face is $F_{n_2}$ and shared-face is $F_{n_1}$, in the orientation $(L_2,F_{n_1},F_{n_2})$ the point $p_2$ has been rotated $90$ degrees counter-clockwise about the origin and shifted $1$ unit down from $p_2$'s standard position with respect to the representation $(F_{n_2},F_{n_1},x_2,y_2)$. Thus, $p_2(L_2,F_{n_1},F_{n_2}) = (-y_2,x_2-1)$ and we get that 
\[\abs{T_{L_2}(p_1, p_2)} = \sqrt{(x_1+y_2)^2+(y_1-x_2+1)^2}.\]
Similarly, in the orientation $(L_3, F_{n_1}, F_{n_2})$ the point $p_2$ has been rotated $90$ degrees clock-wise about the origin and shifted $1$ unit to the right from $p_2$'s standard position with respect to the representation $(F_{n_2},F_{n_1},x_2,y_2)$. Thus, $p_2(L_3,F_{n_1},F_{n_2})$ = $(1+y_2,-x_2)$ and we get that 
\[\abs{T_{L_3}(p_1, p_2)} = \sqrt{(x_1-y_2-1)^2+(y_1+x_2)^2}.\]
\end{proof}

We now address the case of landscapes between opposite faces.  As in the case of adjacent faces, Figure \ref{cubeorientation4-7}, Figure \ref{cubeorientation8-11}, and Figure \ref{cubeorientation12-15} will each be referenced for both the proof of Theorem \ref{Cube_Opposite_Face} and the proof of Corollary \ref{cube_opposite_surface}, with the structure of each landscape being of particular interest in Theorem \ref{Cube_Opposite_Face} and the length of the trail contained therein being of interest in Corollary \ref{cube_opposite_surface}.

\begin{theorem}\label{Cube_Opposite_Face}
Let $F_{n_1}$ and $F_{n_6}$ be opposite faces of $\mathcal P_6$.  Then there are at least 12 landscapes of the form $L_m(F_{n_1} \rightarrow F_{n_6})$. 
\end{theorem}

\setcounter{case}{0}

\begin{proof}

Suppose $\{n_1,n_2,n_3,n_4,n_5,n_6\}$ = $\{1,2,3,4,5,6\}$ such that in nets of our copy of $\mathcal P_6$, $F_{n_1}$ and $F_{n_6}$ are opposite faces, and $F_{n_2}$, $F_{n_3}$, $F_{n_4}$, and $F_{n_5}$ occur in the same counter-clockwise order about $F_{n_1}$ as the faces $F_2$, $F_3$, $F_4$, and $F_5$ occur about the face $F_1$.  We will be constructing landscapes of the form $L_m(F_{n_1} \rightarrow F_{n_6})$, where $m \in \mathbb{N}$ such that $4\leq m\leq 15$. We will let the set $\{n_i,n_j,n_k,n_{\ell}\} = \{n_2,n_3,n_4,n_5\}$ such that $i+j=7$ and $k+\ell=7$.

\begin{figure}[h]
  \centering
\begin{tikzpicture}[scale=1]
\draw [line width=0.05mm] (0,0) rectangle (3,9);
\filldraw[fill=black,draw=black] (0,0) circle (2pt);
\filldraw[fill=black,draw=black] (0,3) circle (2pt);
\filldraw[fill=black,draw=black] (0,6) circle (2pt);
\filldraw[fill=black,draw=black] (0,9) circle (2pt);
\filldraw[fill=black,draw=black] (3,0) circle (2pt);
\filldraw[fill=black,draw=black] (3,6) circle (2pt);
\filldraw[fill=black,draw=black] (3,3) circle (2pt);
\filldraw[fill=black,draw=black] (3,9) circle (2pt);
\draw[line width = 0.05mm] (0,3) -- (3,3);
\draw[line width = 0.05mm] (0,6) -- (3,6);
\draw[line width = 0.05mm] (0,9) -- (3,9);
\node[left][font = {\small}] at (0,0) {$\{n_{j},n_{k},n_{6}\}$};
\node[left][font = {\small}] at (0,3) {$\{n_{i},n_{k},n_{6}\}$};
\node[left][font = {\small}] at (0,6) {$\{n_{1},n_{i},n_{k}\}$};
\node[left][font = {\small}] at (0,9) {$\{n_{1},n_{j},n_{k}\}$};
\node[right][font = {\small}] at (3,0) {$\{n_{j},n_{\ell},n_{6}\}$};
\node[right][font = {\small}] at (3,3) {$\{n_{i},n_{\ell},n_{6}\}$};
\node[right][font = {\small}] at (3,6) {$\{n_{1},n_{i},n_{\ell}\}$};
\node[right][font = {\small}] at (3,9) {$\{n_{1},n_{j},n_{\ell}\}$};
\node[below left, inner sep=15] at (3,9) {$F_{n_1}$};
\node[below left, inner sep=15] at (3,6) {$F_{n_i}$};
\node[above right, inner sep=15] at (0,0) {$F_{n_6}$};
\filldraw[fill=black,draw=black] (.75,8) circle (2pt);
\filldraw[fill=black,draw=black] (2.4,2.2) circle (2pt);
\node[left][font = {\small}] at (.75,8) {$p_{1}$};
\node[right][font = {\small}] at (2.4,2.2) {$p_{2}$};
\draw[line width = 0.05mm] (.75,8) -- (2.4,2.2);
\draw[line width = 0.65mm] (.75,8) -- (.75,6) -- (0,6);
\draw[line width = 0.65mm] (2.4,2.2) -- (2.4,3) -- (3,3);
\draw[line width = 0.05mm] (.75,8) -- (.75,2.2) -- (2.4,2.2);
\draw[line width = 0.05mm] (.75,2.2) rectangle (.95,2.4);
\node[left][font = {\small}] at (.75,7) {$v$};
\node[above][font = {\small}] at (.375,6) {$u$};
\node[above][font = {\small}] at (2.7,3) {$s$};
\node[right][font = {\small}] at (2.4,2.6) {$t$};
\end{tikzpicture}
    \caption{$L_m(F_{n_1} \rightarrow F_{n_6})$ and $T_O(p_1, p_2)$, where $O= (L_m, F_{n_1}, F_{n_i})$ with $m \in \{4,5,6,7\}$}
    \label{cubeorientation4-7}
\end{figure}
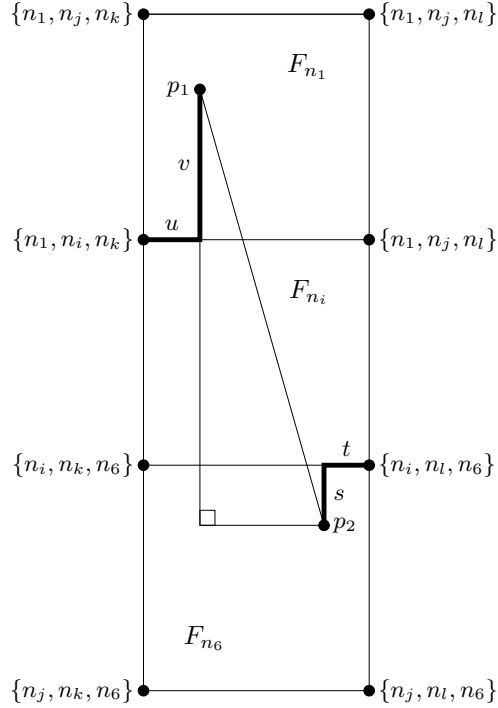

Since the dual graph is a path of three vertices, $F_{n_1}$ and $F_{n_6}$ are opposite faces, and there are four faces adjacent to $F_{n_1}$, $F_{n_1}$ must be the first face, $F_{n_6}$ must be the third face, and the second face must come from the following set: $\{F_{n_2}, F_{n_3}, F_{n_4}, F_{n_5}\}$. This provides us with what we define to be landscapes of the form $L_m(F_{n_1} \rightarrow F_{n_6})$, where $m \in \{4,5,6,7\}$.

\begin{itemize}
    \item For $L_4$, we will let $n_i=n_2$. In turn we get that $n_j=n_5$, and without loss of generality $n_k=n_3$, and $n_{\ell}=n_4$.
    \item For $L_5$, we let $n_i=n_3$ in the orientation $(L_5,F_{n_1},F_{n_3})$, and we get that $n_j=n_4$, $n_k=n_5$, and $n_{\ell}=n_2$. 
    \item For $L_6$, we let $n_i=n_4$, in the orientation $(L_6,F_{n_1},F_{n_4})$, and we get that $n_j=n_3$, $n_k=n_2$, and $n_{\ell}=n_5$. 
    \item For $L_7$, we let $n_i=n_5$, in the orientation $(L_7,F_{n_1},F_{n_5})$, and we get that $n_j=n_2$, $n_k=n_4$, and $n_{\ell}=n_3$. 
\end{itemize}

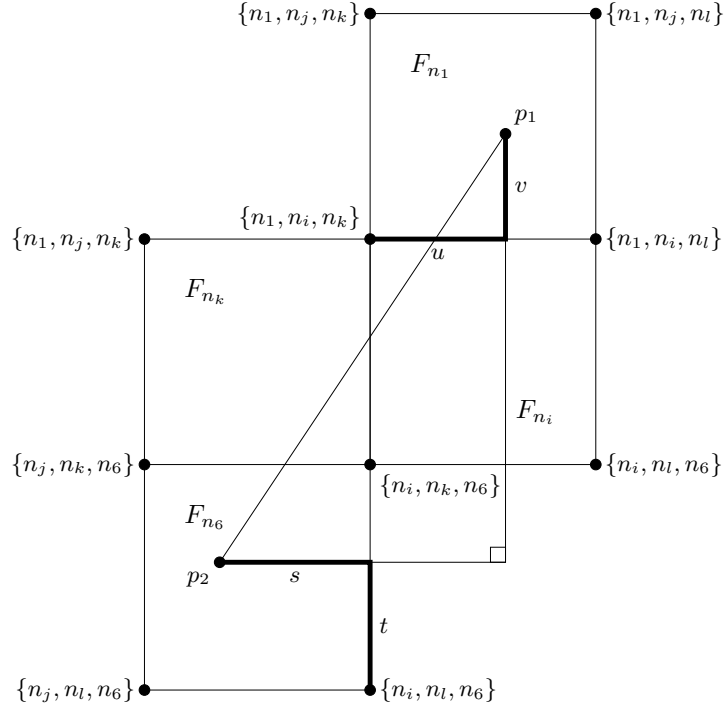
\begin{figure}[h]
  \centering
\begin{tikzpicture}[scale=1]
\draw [line width=0.05mm] (0,0) rectangle (3,6);
\draw [line width=0.05mm] (3,3) rectangle (6,9);
\draw [line width=0.05mm] (3,6) rectangle (6,6);
\draw [line width=0.05mm] (0,3) rectangle (3,3);
\filldraw[fill=black,draw=black] (0,3) circle (2pt);
\filldraw[fill=black,draw=black] (0,0) circle (2pt);
\filldraw[fill=black,draw=black] (0,6) circle (2pt);
\filldraw[fill=black,draw=black] (3,3) circle (2pt);
\filldraw[fill=black,draw=black] (3,6) circle (2pt);
\filldraw[fill=black,draw=black] (6,3) circle (2pt);
\filldraw[fill=black,draw=black] (6,6) circle (2pt);
\filldraw[fill=black,draw=black] (3,9) circle (2pt);
\filldraw[fill=black,draw=black] (6,9) circle (2pt);
\filldraw[fill=black,draw=black] (3,0) circle (2pt);

\node[left][font = {\small}] at (0,0) {$\{n_{j},n_{\ell},n_{6}\}$};
\node[left][font = {\small}] at (0,3) {$\{n_{j},n_{k},n_{6}\}$};
\node[left][font = {\small}] at (0,6) {$\{n_{1},n_{j},n_{k}\}$};
\node[above left][font = {\small}] at (3,6) {$\{n_{1},n_{i},n_{k}\}$};
\node[left][font = {\small}] at (3,9) {$\{n_{1},n_{j},n_{k}\}$};
\node[right][font = {\small}] at (6,9) {$\{n_{1},n_{j},n_{\ell}\}$};
\node[right][font = {\small}] at (6,6) {$\{n_{1},n_{i},n_{\ell}\}$};
\node[right][font = {\small}] at (6,3) {$\{n_{i},n_{\ell},n_{6}\}$};
\node[below right][font = {\small}] at (3,3) {$\{n_{i},n_{k},n_{6}\}$};
\node[right][font = {\small}] at (3,0) {$\{n_{i},n_{\ell},n_{6}\}$};
\node[below right,inner sep=15] at (0,6) {$F_{n_k}$};
\node[below right,inner sep=15] at (0,3) {$F_{n_6}$};
\node[above left,inner sep=15] at (6,3) {$F_{n_i}$};
\node[below right,inner sep=15] at (3,9) {$F_{n_1}$};
\filldraw[fill=black,draw=black] (1,1.7) circle (2pt);
\filldraw[fill=black,draw=black] (4.8,7.4) circle (2pt);
\node[above right][font = {\small}] at (4.8,7.4) {$p_{1}$};
\node[below left][font = {\small}] at (1,1.7) {$p_{2}$};
\draw [line width=0.05mm] (1,1.7) -- (4.8,7.4);
\draw [line width=0.05mm] (4.8,7.4) -- (4.8,1.7) -- (1,1.7);
\draw [line width=0.05mm] (4.8,1.7) rectangle (4.6,1.9);
\draw[line width = 0.65mm] (4.8,7.4) -- (4.8,6) -- (3,6);
\draw[line width = 0.65mm] (1,1.7) -- (3,1.7) -- (3,0);
\node[right][font = {\small}] at (3,.85) {$s$};
\node[below][font = {\small}] at (2,1.7) {$t$};
\node[below][font = {\small}] at (3.9,6) {$u$};
\node[right][font = {\small}] at (4.8,6.7) {$v$};

\end{tikzpicture}
    \caption{$L_m(F_{n_1} \rightarrow F_{n_6})$ and $T_O(p_1, p_2)$, where $O= (L_m, F_{n_1}, F_{n_i})$ with $m \in \{8,9,10,11\}$}
    \label{cubeorientation8-11}
\end{figure}

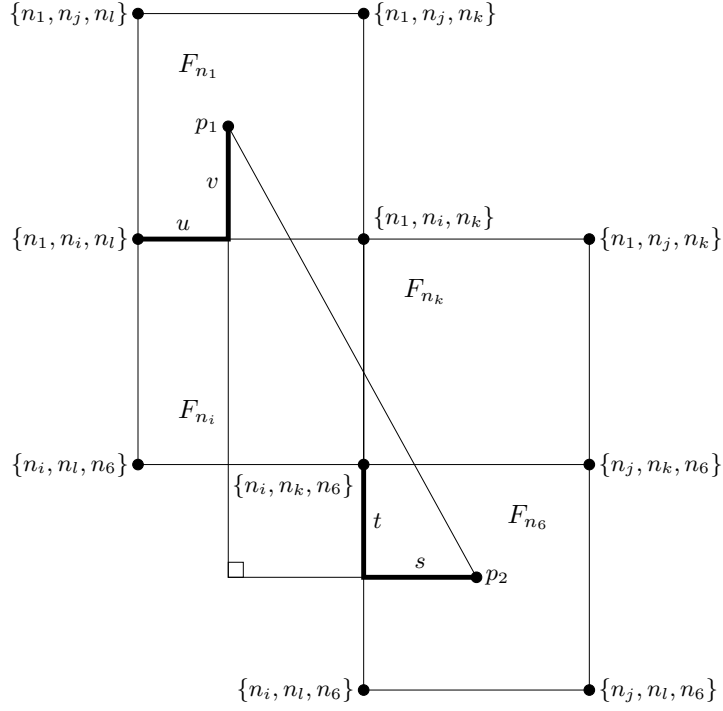
\begin{figure}[h]
  \centering
\begin{tikzpicture}[scale=1]
\draw [line width=0.05mm] (6,0) rectangle (3,6);
\draw [line width=0.05mm] (3,3) rectangle (0,9);
\draw [line width=0.05mm] (3,3) -- (6,3);
\draw [line width=0.05mm] (0,6) -- (3,6);
\draw [line width=0.05mm] (3,3) -- (3,6);
\filldraw[fill=black,draw=black] (0,3) circle (2pt);
\filldraw[fill=black,draw=black] (6,0) circle (2pt);
\filldraw[fill=black,draw=black] (0,6) circle (2pt);
\filldraw[fill=black,draw=black] (3,3) circle (2pt);
\filldraw[fill=black,draw=black] (3,6) circle (2pt);
\filldraw[fill=black,draw=black] (6,3) circle (2pt);
\filldraw[fill=black,draw=black] (6,6) circle (2pt);
\filldraw[fill=black,draw=black] (3,9) circle (2pt);
\filldraw[fill=black,draw=black] (0,9) circle (2pt);
\filldraw[fill=black,draw=black] (3,0) circle (2pt);
\node[right][font = {\small}] at (6,0) {$\{n_{j},n_{k},n_{6}\}$};
\node[right][font = {\small}] at (6,3) {$\{n_{j},n_{\ell},n_{6}\}$};
\node[right][font = {\small}] at (6,6) {$\{n_{1},n_{j},n_{\ell}\}$};
\node[above right][font = {\small}] at (3,6) {$\{n_{1},n_{i},n_{\ell}\}$};
\node[right][font = {\small}] at (3,9) {$\{n_{1},n_{j},n_{\ell}\}$};
\node[left][font = {\small}] at (0,9) {$\{n_{1},n_{j},n_{k}\}$};
\node[left][font = {\small}] at (0,6) {$\{n_{1},n_{i},n_{k}\}$};
\node[left][font = {\small}] at (0,3) {$\{n_{i},n_{k},n_{6}\}$};
\node[below left][font = {\small}] at (3,3) {$\{n_{i},n_{\ell},n_{6}\}$};
\node[left][font = {\small}] at (3,0) {$\{n_{i},n_{k},n_{6}\}$};
\node[below right,inner sep=15] at (3,6) {$F_{n_{\ell}}$};
\node[below left,inner sep=15] at (6,3) {$F_{n_6}$};
\node[above right,inner sep=15] at (0,3) {$F_{n_i}$};
\node[below right,inner sep=15] at (0,9) {$F_{n_1}$};
\filldraw[fill=black,draw=black] (1.2,7.5) circle (2pt);
\filldraw[fill=black,draw=black] (4.5,1.5) circle (2pt);
\node[left][font = {\small}] at (1.2,7.5) {$p_{1}$};
\node[right][font = {\small}] at (4.5,1.5) {$p_{2}$};
\draw [line width=0.05mm] (1.2,7.5) -- (1.2,1.5) -- (4.5,1.5);
\draw [line width=0.05mm] (1.2,7.5) -- (4.5,1.5);
\draw [line width=0.05mm] (1.2,1.5) rectangle (1.4,1.7);
\draw[line width = 0.65mm] (1.2,7.5) -- (1.2,6) -- (0,6);
\draw[line width = 0.65mm] (4.5,1.5) -- (3,1.5) -- (3,3);
\node[right][font = {\small}] at (3,2.25) {$s$};
\node[above][font = {\small}] at (3.75,1.5) {$t$};
\node[above][font = {\small}] at (.6,6) {$u$};
\node[left][font = {\small}] at (1.2,6.75) {$v$};
\end{tikzpicture}
    \caption{$L_m(F_{n_1} \rightarrow F_{n_6})$ and $T_O(p_1, p_2)$, where $O= (L_m, F_{n_1}, F_{n_i})$ with $m \in \{12,13,14,15\}$}
    \label{cubeorientation12-15}
\end{figure}

\begin{case}
The dual graph is a path of four vertices
\end{case}

Since we are constructing landscapes of the form $L_m(F_{n_1}\rightarrow F_{n_6})$ and $\mathcal{P}_6$ has six faces, the first face must be $F_{n_1}$, the fourth face must be $F_{n_6}$, and the second and third faces must be distinct elements of the set $\{F_{n_2},F_{n_3},F_{n_4},F_{n_5}\}$. Furthermore, we  cannot have pairs of opposite faces from the aforementioned set in the same landscape. Thus, there are eight such landscapes since we cannot have $F_{n_2}$ and $F_{n_5}$ or $F_{n_3}$ and $F_{n_4}$ in the same landscape.
These landscapes will produce two general structures, depicted in Figures \ref{cubeorientation8-11} and \ref{cubeorientation12-15}.

We will first construct landscapes having the general structure depicted in Figure \ref{cubeorientation8-11}, namely $L_m(F_{n_1} \rightarrow F_{n_6})$, where $m \in \{8,9,10,11\}$.  
\begin{itemize}
    \item For $L_8$, we will respectively let the second and third faces, $F_{n_i}$ and $F_{n_k}$, be $F_{n_2}$ and $F_{n_3}$. In turn we get that $n_j=n_5$ and $n_{\ell}=n_4$.  
    \item For $L_9$, we will respectively let the second and third faces, $F_{n_i}$ and $F_{n_k}$, be $F_{n_3}$ and $F_{n_5}$. In turn we get that $n_j=n_4$ and $n_{\ell}=n_2$
    \item For $L_{10}$, we will respectively let the second and third faces, $F_{n_i}$ and $F_{n_k}$, be $F_{n_5}$ and $F_{n_4}$. In turn we get that $n_j=n_2$ and $n_{\ell}=n_3$.
    \item For $L_{11}$, we will respectively let the second and third faces, $F_{n_i}$ and $F_{n_k}$, be $F_{n_4}$ and $F_{n_2}$. In turn we get that $n_j=n_3$ and $n_{\ell}=n_5$.
\end{itemize}

Finally, we will construct landscapes having the general structure depicted in Figure \ref{cubeorientation12-15}, namely $L_m(F_{n_1} \rightarrow F_{n_6})$, where $m \in \{12,13,14,15\}$.

\begin{itemize}
    \item For $L_{12}$, we will respectively let the second and third faces, $F_{n_i}$ and $F_{n_{\ell}}$, be $F_{n_2}$ and $F_{n_4}$. In turn we get that $n_j=n_5$ and $n_k=n_3$.
    \item For $L_{13}$, we will respectively let the second and third faces, $F_{n_i}$ and $F_{n_{\ell}}$, be $F_{n_3}$ and $F_{n_2}$. In turn we get that $n_j=n_4$ and $n_k=n_5$.
    \item For $L_{14}$, we will respectively let the second and third faces, $F_{n_i}$ and $F_{n_{\ell}}$, be $F_{n_5}$ and $F_{n_3}$. In turn we get that $n_j=n_2$ and $n_k=n_4$.
    \item For $L_{15}$, we will respectively let the second and third faces, $F_{n_i}$ and $F_{n_{\ell}}$, be $F_{n_4}$ and $F_{n_5}$. In turn we get that $n_j=n_3$ and $n_k=n_2$.
\end{itemize}
\end{proof}

Having identified a collection of landscapes of $\mathcal P_6$, we have the following corollary.  Note that due to the symmetries of $\mathcal P_6$ and the presence of a formula for switching from one shared-face to another we can assume without loss of generality that the representations of the points $p_1$ and $p_2$ are such that $p_1=(F_{n_1}, F_{n_2}, x_1, y_1)$ and $p_2=(F_{n_6}, F_{n_2}, x_2, y_2)$.  Later when we prove that given a pair of opposite faces these are the only twelve valid landscapes between said opposite faces, we will see that the bound in this corollary actually provides equality.  As a preliminary notion for the statement of Corollary \ref{cube_opposite_surface} and the proof of Lemma \ref{boundary}, we use the topological concepts \emph{interior of a set} and \emph{boundary of a set} in the standard way.  That is, we define the interior of a landscape $L_i$, denoted $\inte (L_i)$,  to be the topological interior of said set of points given by the standard topology on $\mathbb R^2$.  Furthermore, we define the
boundary of a face $F_n$ of some convex unit polyhedron $\mathcal{P}_n $, denoted $\partial{F}_n$, to be the set of points contained in the edges of $F_n$ or equivalently the set of points contained in $F_n \setminus \inte(F_n)$.  

\begin{corollary}\label{cube_opposite_surface}
Given two points $p_1 \in \inte(F_{n_1})$ and $p_2 \in \inte(F_{n_6})$ on two opposite faces of the cube, with $p_1=(F_{n_1},F_{n_2},x_1,y_1)$ and $p_2=(F_{n_6},F_{n_2},x_2,y_2)$,
\[d_{\mathcal P_6}(p_1,p_2) \leq d_{\mathcal P_6}^O(p_1,p_2),\text{ where}\] 
\[d_{\mathcal P_6}^O(p_1,p_2)=\min\left\{\abs{T_{L_i}(p_1,p_2)}: i \in \mathbb N \text{ with } 4 \leq i \leq 15\right\},\]
and for each $i \in \mathbb N$ with $4 \leq i \leq 15$, the trail length $\abs{T_{L_i}(p_1,p_2)}$ is given in the table below:

\begin{center}
\begin{tabular}{ |c|c||c|c| }
    \hline
    Landscape & Trail Length Formula & Landscape & Trail Length Formula \\
    \hline
    \rule{0pt}{12pt}$L_4$ &\rule{-4pt}{12pt} $\sqrt{(x_1+x_2-1)^2+(y_1+y_2+1)^2}$ & \rule{0pt}{12pt}$L_{10}$ &\rule{-4pt}{12pt} $\sqrt{(x_1+y_2-2)^2+(y_1-x_2-2)^2}$\\[1pt]
    \hline
    \rule{0pt}{12pt}$L_5$ &\rule{-4pt}{12pt} $\sqrt{(y_1-y_2)^2+(x_1-x_2+2)^2}$ & \rule{0pt}{12pt}$L_{11}$ &\rule{-4pt}{12pt} $\sqrt{(y_1+x_2)^2+(x_1-y_2-2)^2}$\\[1pt]
    \hline
    \rule{0pt}{12pt}$L_6$ &\rule{-4pt}{12pt} $\sqrt{(y_1-y_2)^2+(x_1 -x_2-2)^2}$ & \rule{0pt}{12pt}$L_{12}$ &\rule{-4pt}{12pt} $\sqrt{(x_1-y_2-1)^2+(y_1+x_2+1)^2}$\\[1pt]
    \hline
    \rule{0pt}{12pt}$L_7$ &\rule{-4pt}{12pt} $\sqrt{( x_1+x_2-1)^2+(y_1+y_2-3)^2}$ & \rule{0pt}{12pt}$L_{13}$ &\rule{-4pt}{12pt} $\sqrt{(y_1-x_2+1)^2+(x_1+y_2+1)^2}$\\[1pt]
    \hline
    \rule{0pt}{12pt}$L_8$ &\rule{-4pt}{12pt} $\sqrt{(x_1+y_2)^2+(y_1-x_2+2)^2}$ & \rule{0pt}{12pt}$L_{14}$ &\rule{-4pt}{12pt} $\sqrt{(x_1-y_2+1)^2+(y_1+x_2-3)^2}$\\[1pt]
    \hline
    \rule{0pt}{12pt}$L_9$ &\rule{-4pt}{12pt} $\sqrt{(y_1+x_2-2)^2+(x_1-y_2+2)^2}$ & \rule{0pt}{12pt}$L_{15}$ &\rule{-4pt}{12pt} $\sqrt{(y_1-x_2-1)^2+(x_1+y_2-3)^2}$\\[1pt]
    \hline
\end{tabular}
\end{center}
\end{corollary}

\begin{proof}
Suppose $\{n_1,n_2,n_3,n_4,n_5,n_6\}$ = $\{1,2,3,4,5,6\}$ such that in nets of our copy of $\mathcal P_6$, $F_{n_1}$ and $F_{n_6}$ are opposite faces, and $F_{n_2}$, $F_{n_3}$, $F_{n_4}$, and $F_{n_5}$ occur in the same counter-clockwise order about $F_{n_1}$ as the faces $F_2$, $F_3$, $F_4$, and $F_5$ occur about the face $F_1$.  For the ease of notation in diagrams, for each $m$ with $4 \leq m \leq 15$, we will let $u_m = p_1(L_m,F_{n_1},F_{n_i})_x$, $v_m =p_1(L_m,F_{n_1},F_{n_i})_y, s_m = p_2(L_m,F_{n_6},F_{n_i})_x$, and $t_m =p_2(L_m,F_{n_6},F_{n_i})_y$.

Since each of the four landscapes $L_4$, $L_5$, $L_6$, and $L_7$ have the same general structure, if we apply Lemma \ref{Cube_Shared_Face} a sufficient number of times so that the face $F_{n_i}$ is the shared face for the representations of $p_1$ and $p_2$ being considered, then the same transformation takes us from $p_2$'s standard position with respect to this representation to $p_2$'s new position in the orientation $(L_m, F_{n_1}, F_{n_i})$. Specifically, in each case, since $p_2$'s home-face is $F_{n_6}$ and shared-face in this representation is $F_{n_i}$, in the orientation $(L_m, F_{n_1}, F_{n_i})$, the point $p_2$ has been rotated $180$ degrees about the origin, shifted $1$ unit to the right, and shifted 1 unit down from $p_2$'s standard position with respect to the representation $(F_{n_6},F_{n_i},s_m,t_m)$ yielding a result of $(1-s_m,-t_m-1)$. 

In $L_4$, since $n_i=n_2$, we do not need to apply Lemma \ref{Cube_Shared_Face}. Thus, $p_1(L_4,F_{n_1},F_{n_2}) = (x_1,y_1)$, after applying the transformation  $p_2(L_4, F_{n_1}, F_{n_2}) = (1-x_2, -y_2 - 1)$, and we get that
\[\abs{T_{L_4}(p_1,p_2)} = \sqrt{(x_1+x_2-1)^2+(y_1+y_2+1)^2}.\]

In $L_5$, by application of Lemma \ref{Cube_Shared_Face}, we see that $p_1= (F_{n_1}, F_{n_3}, 1-y_1,x_1)$ and $p_2= (F_{n_6}, F_{n_3}, y_2,1-x_2)$. Thus, after application of the transformation  $p_2(L_5,F_{n_1},F_{n_3}) = (1-y_2,x_2-2)$ and we get that 
\[\abs{T_{L_5}(p_1, p_2)} = \sqrt{(y_1-y_2)^2+(x_1-x_2+2)^2}.\]

In $L_6$, also by application of Lemma \ref{Cube_Shared_Face}, we see that $p_1 = (F_{n_1},F_{n_4},y_1,1-x_1)$ and $p_2 = (F_{n_6},F_{n_4},1-y_2,x_2)$. Thus, after application of the transformation $p_2(L_6,F_{n_1},F_{n_4}) = (y_2,-x_2-1)$ and we get that
\[\abs{T_{L_6}(p_1, p_2)} = \sqrt{(y_1-y_2)^2+(x_1 -x_2-2)^2}.\]

In $L_7$, also by application of Lemma \ref{Cube_Shared_Face}, we see that $p_1 = (F_{n_1},F_{n_5},1 - x_1,1-y_1)$ and $p_2 = (F_{n_6},F_{n_5},1 - x_2,1-y_2)$. Thus, after application of the transformation $p_2(L_7,F_{n_1},F_{n_5}) = (x_2,y_2-2)$ and we get that
\[\abs{T_{L_7}(p_1, p_2)} = \sqrt{( x_1+x_2-1)^2+(y_1+y_2-3)^2}.\]

As before, landscapes $L_8$, $L_9$, $L_{10}$, and $L_{11}$ have the same general structure, so if we apply Lemma \ref{Cube_Shared_Face} a sufficient number of times so that the face $F_{n_i}$ is the shared face for the representations of $p_1$ and $p_2$ being considered, then the same transformation takes us from $p_2$'s standard position with respect to this representation to $p_2$'s new position in the orientation $(L_m, F_{n_1}, F_{n_i})$. Specifically, in each case, since $p_2$'s home-face is $F_{n_6}$ and shared-face in this representation is $F_{n_i}$, in the orientation $(L_m, F_{n_1}, F_{n_i})$, the point $p_2$ has been rotated $90$ degrees counterclockwise about the origin and shifted $2$ units down from $p_2$'s standard position with respect to the representation $(F_{n_6},F_{n_i},s_m,t_m)$ yielding a result of $(-t_m,s_m-2)$.

In $L_8$, since $n_i=n_2$, we do not need to apply Lemma \ref{Cube_Shared_Face}.  Thus, $p_1(L_8,F_{n_1},F_{n_2}) = (x_1,y_1)$, after applying the transformation $p_2(L_8, F_{n_1}, F_{n_2})=(-y_2, x_2-2)$ and we get that
\[\abs{T_{L_8}(p_1, p_2)} = \sqrt{(x_1+y_2)^2+(y_1-x_2+2)^2}.\]

In $L_9$, by application of Lemma \ref{Cube_Shared_Face}, we see that $p_1 = (F_{n_1},F_{n_3},1 - y_1,x_1)$ and $p_2 = (F_{n_6},F_{n_3},y_2,1-x_2)$. Thus, after application of the transformation $p_2(L_9,F_{n_1},F_{n_3}) = (x_2-1,y_2-2)$ and we get that
\[\abs{T_{L_9}(p_1, p_2)} = \sqrt{(y_1+x_2-2)^2+(x_1-y_2+2)^2}.\]

In $L_{10}$, also by application of Lemma \ref{Cube_Shared_Face}, we see that $p_1 = (F_{n_1},F_{n_5},1 - x_1,1-y_1)$ and $p_2 = (F_{n_6},F_{n_5},1-x_2,1-y_2)$. Thus after application of the transformation $p_2(L_{10},F_{n_1},F_{n_5}) = (y_2-1,-x_2-1)$ and we get that
\[\abs{T_{L_{10}}(p_1, p_2)} = \sqrt{(x_1+y_2-2)^2+(y_1-x_2-2)^2}.\]

In $L_{11}$, also by application of Lemma \ref{Cube_Shared_Face}, we see that $p_1 = (F_{n_1},F_{n_4},y_1,1-x_1)$ and $p_2 = (F_{n_6},F_{n_4},1-y_2,x_2)$. Thus after application of the transformation $p_2(L_{11},F_{n_1},F_{n_4}) = (-x_2,-y_2-1)$ and we get that
\[\abs{T_{L_{11}}(p_1, p_2)}= \sqrt{(y_1+x_2)^2+(x_1-y_2-2)^2}.\]

Once more, the landscapes $L_{12}$, $L_{13}$, $L_{14}$, and $L_{15}$ have the same general structure, and so if we apply Lemma \ref{Cube_Shared_Face} a sufficient number of times so that the face $F_{n_i}$ is the shared face for the representations of $p_1$ and $p_2$ being considered, then the same transformation takes us from $p_2$'s standard position with respect to this representation to $p_2$'s new position in the orientation $(L_m, F_{n_1}, F_{n_i})$. Specifically, in each case, since $p_2$'s home-face is $F_{n_6}$ and shared-face in this representation is $F_{n_i}$, in the orientation $(L_m, F_{n_1}, F_{n_i})$, the point $p_2$ has been rotated $90$ degrees clockwise about the origin, shifted $1$ unit down, and shifted 1 unit to the right from $p_2$'s standard position with respect to the representation $(F_{n_6},F_{n_i},s_m,t_m)$ yielding a result of $(t_m+1,-s_m-1)$.

In $L_{12}$, since $n_i=n_2$, we do not need to apply Lemma \ref{Cube_Shared_Face}.  Thus, $p_1(L_{12},F_{n_1},F_{n_2}) = (x_1,y_1)$, after applying the transformation $p_2(L_{12}, F_{n_1}, F_{n_2})=(y_2+1,-x_2-1)$ and we get that 
\[\abs{T_{L_{12}}(p_1, p_2)} = \sqrt{(x_1-y_2-1)^2+(y_1+x_2+1)^2}.\]

In $L_{13}$, by application of Lemma \ref{Cube_Shared_Face}, we see that $p_1 = (F_{n_1},F_{n_3},1 - y_1,x_1)$ and $p_2 = (F_{n_6},F_{n_3},y_2,1-x_2)$. Thus after application of the transformation  $p_2(L_{13},F_{n_1},F_{n_3}) = (2-x_2,-y_2-1)$ and we get that
\[\abs{T_{L_{13}}(p_1, p_2)} = \sqrt{(y_1-x_2+1)^2+(x_1+y_2+1)^2}.\]

In $L_{14}$, also by application of Lemma \ref{Cube_Shared_Face}, we see that $p_1 = (F_{n_1},F_{n_5},1 - x_1,1-y_1)$ and $p_2 = (F_{n_6},F_{n_5},1-x_2,1-y_2)$.  Thus after application of the transformation
$p_2(L_{14},F_{n_1},F_{n_5}) = (2-y_2,x_2-2)$ and we get that
\[\abs{T_{L_{14}}(p_1, p_2)} = \sqrt{(x_1-y_2+1)^2+(y_1+x_2-3)^2}.\]

Finally, in $L_{15}$ by application of Lemma \ref{Cube_Shared_Face}, we see that $p_1 = (F_{n_1},F_{n_4},y_1,1-x_1)$ and $p_2 = (F_{n_6},F_{n_4},1-y_2,x_2)$. Thus after application of the transformation
$p_2(L_{15},F_{n_1},F_{n_4}) = (x_2+1,y_2-2)$ and we get that
\[\abs{T_{L_{15}}(p_1, p_2)} = \sqrt{(y_1-x_2-1)^2+(x_1+y_2-3)^2}.\] 
\end{proof}

The development of the following lemma is motivated by the convenience and simplicity of considering points on interior edges of a given origin face and destination face of an arbitrary landscape in any convex unit polyhedron. By doing so, we eliminate the need for four variables when constructing arguments relating to trails incident to some landscape. This becomes extremely useful in the proof of Theorem \ref{finaltheorem}.

\begin{lemma}\label{boundary}
Let $F_n$ and $F_m$ be two distinct faces of some landscape $L_i(F_n \rightarrow F_m)$ of some convex unit polyhedron $\mathcal P$.
Then for all points $x$ and $y$ such that $x \in  \overline{\partial F_n \cap \inte (L_i)}$
and $y \in  \overline{\partial F_m \cap \inte (L_i)}$,
there exists some landscape $L$ containing $F_n$ and $F_m$ such that $\abs{T_{L}(x,y)} < \abs{T_{L_i}(x,y)}$ if and only if for every $p_1$ and $p_2$ such that $p_1 \in F_n$ and $p_2 \in F_m$, there is some landscape $L$ such that $\abs{T_{L}(p_1,p_2)} < \abs{T_{L_i}(p_1,p_2)}$.
\end{lemma}

\begin{proof}
We will begin with the forward direction. Suppose for all points $x \in  \overline{\partial F_n \cap \inte (L_i)}$ and $y \in  \overline{\partial F_m \cap \inte (L_i)}$ there is some landscape $L$ so that $\abs{T_{L}(x,y)} < \abs{T_{L_i}(x,y)}$. Let $p_1 \in F_n$, $p_2 \in F_m$ and define $x = \overline{\partial F_n \cap \inte (L_i)} \cap T_{O_i}(p_1,p_2)$, where $O_i$ is an orientation of $L_i$. Similarly, define $y = \overline{\partial F_m \cap \inte (L_i)} \cap T_{O_i}(p_1,p_2)$. Then by assumption we can fix a landscape $\hat{L}$ containing $F_n$ and $F_m$ with $\abs{T_{\hat{L}}(x,y)} < \abs{T_{L_i}(x,y)}$.  Likewise fix an orientation $\hat{O}$ of $\hat{L}$. Now, $\abs{T_{L_i}(p_1,x)} = \abs{T_{\hat{L}}(p_1,x)}$ and $\abs{T_{L_i}(p_2,y)} = \abs{T_{\hat{L}}(p_2,y)}$ since $F_n$ and $F_m$ are both contained in the landscapes $L_i$ and $\hat{L}$. 
Thus, 
\begin{equation*}
\begin{split}
\abs{T_{\hat{L}}(p_1,x)} + \abs{T_{\hat{L}}(x,y)} + \abs{T_{\hat{L}}(p_2,y)} & = \abs{T_{L_i}(p_1,x)} + \abs{T_{\hat{L}}(x,y)} + \abs{T_{L_i}(p_2,y)} \\
& < \abs{T_{L_i}(p_1,x)}+\abs{T_{L_i}(x,y)} + \abs{T_{L_i}(p_2,y)} \\
& = \abs{T_{L_i}(p_1,p_2)}.
\end{split}
\end{equation*}
There are two cases: 

\setcounter{case}{0}

\begin{case}
$T_{\hat{O}}(p_1,p_2) \subseteq \hat{L}$
\end{case}

Let $L=\hat{L}$. We then have that $\abs{T_L(p_1,p_2)} \leq \abs{T_L(p_1,x)} + \abs{T_L(x,y)} + \abs{T_L(p_2,y)}$.
Which implies that $\abs{T_L(p_1,p_2)} < \abs{T_{L_i}(p_1,p_2)}$ by the above strict inequality.

\begin{case}
$T_{\hat{O}}(p_1,p_2) \not\subseteq \hat{L}$
\end{case}

Then $\abs{T_{\hat{L}}(p_1,p_2)}=\infty$, but since the path given by the concatenation $T_{\hat{O}}(p_1,x) ^\frown T_{\hat{O}}(x,y) ^\frown T_{\hat{O}}(y,p_2)$ is still shorter than the trail $T_{O_i}(p_1,p_2)$, it follows that $T_{O_i}(p_1,p_2)$ can not be the shortest path between $p_1$ and $p_2$ on the surface of $\mathcal P$.  Now, by Theorem \ref{fullgenerality}, there must exist an orientation $O$ of $\mathcal P$ so that the shortest path between $p_1$ and $p_2$ is $T_O(p_1,p_2)$. It then follows that there must be a landscape $L$ (of which $O$ is an orientation)
such that $\abs{T_{L}(p_1,p_2)} < \abs{T_{L_i}(p_1,p_2)}.$

Conversely, Suppose for every $p_1\in F_n$ and $p_2\in F_m$, there is some $L$ such that 
$\abs{T_{L}(p_1,p_2)} < \abs{T_{L_i}(p_1,p_2)}$. Let $x \in  \overline{\partial F_n \cap \inte (L_i)}$ and $y \in  \overline{\partial F_m \cap \inte (L_i)}$. Since $F_n$ is closed by definition,  $\overline{\partial F_n \cap \inte(L_i)} \subseteq F_n$ and likewise $\overline{\partial F_m \cap \inte(L_i)} \subseteq F_m$. Then by assumption we can fix a particular $L$ so that $\abs{T_{L}(x,y)} < \abs{T_{L_i}(x,y)}$ which completes the proof. 

\end{proof}

We have now reached a pivotal point in the development of our framework for the cube. The following theorem, while intuitively simple, will enable us to eliminate a potentially vast amount of landscapes that would have otherwise been difficult to eliminate. This will allow us to show that the only valid landscapes are the fifteen developed in Theorems \ref{Cube_Adjacent_Face} and \ref{Cube_Opposite_Face}.

\begin{theorem}\label{finaltheorem}
Suppose $\{n_1,n_2,n_3,n_4,n_5,n_6\}$ $=$ $\{1,2,3,4,5,6\}$ such that in nets of our copy of $\mathcal P_6$ we have that $F_{n_1}$ and $F_{n_6}$ are opposite faces and $F_{n_2}$, $F_{n_3}$, $F_{n_4}$, and $F_{n_5}$ occur about $F_{n_1}$ in the same counter-clockwise order as $F_2$, $F_3$, $F_4$, and $F_5$ occur about $F_1$.  Let $L(F_{n_1} \rightarrow F_{n_i})$ with $i \in \{2,3,4,5,6\}$, be a landscape of $\mathcal P_6$.  If $L \not \in \{L_i\}_{i=1}^{15}$, then $L$ is not valid.
\end{theorem}

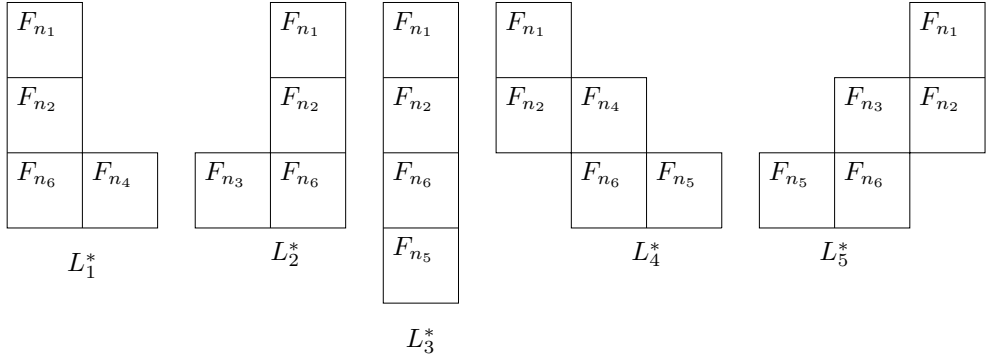
\begin{figure}[h]
  \centering
\begin{tikzpicture}[scale=1]
\draw [line width=0.05mm] (0,0) rectangle (1,4);
\draw [line width=0.05mm] (0,1) -- (1,1);
\draw [line width=0.05mm] (0,2) -- (1,2);
\draw [line width=0.05mm] (0,3) -- (1,3);
\node[below right] at (0,1) {$F_{n_5}$};
\node[below right] at (0,2) {$F_{n_6}$};
\node[below right] at (0,3) {$F_{n_2}$};
\node[below right] at (0,4) {$F_{n_1}$};

\draw [line width=0.05mm] (-2.5,1) -- (-.5,1) -- (-.5,4) -- (-1.5,4) -- (-1.5,2) -- (-2.5,2) -- cycle;
\draw [line width=0.05mm] (-1.5,1) -- (-1.5,2) -- (-.5,2);
\draw [line width=0.05mm] (-1.5,3) -- (-.5,3);
\node[below right] at (-2.5,2) {$F_{n_3}$};
\node[below right] at (-1.5,2) {$F_{n_6}$};
\node[below right] at (-1.5,3) {$F_{n_2}$};
\node[below right] at (-1.5,4) {$F_{n_1}$};

\draw [line width=0.05mm] (-5,1) -- (-3,1) -- (-3,2) -- (-4,2) -- (-4,4) -- (-5,4) -- cycle;
\draw [line width=0.05mm] (-4,1) -- (-4,2) -- (-5,2);
\draw [line width=0.05mm] (-5,3) -- (-4,3);
\node[below right] at (-4,2) {$F_{n_4}$};
\node[below right] at (-5,2) {$F_{n_6}$};
\node[below right] at (-5,3) {$F_{n_2}$};
\node[below right] at (-5,4) {$F_{n_1}$};

\draw [line width=0.05mm] (4.5,1) -- (2.5,1) -- (2.5,2) -- (1.5,2) -- (1.5,4) -- (2.5,4) -- (2.5,2) -- (4.5,2) -- (4.5,1);
\draw [line width=0.05mm] (1.5,3) -- (2.5,3);
\draw [line width=0.05mm] (3.5,1) -- (3.5,3) -- (2.5,3);
\node[below right] at (1.5,3) {$F_{n_2}$};
\node[below right] at (1.5,4) {$F_{n_1}$};
\node[below right] at (2.5,3) {$F_{n_4}$};
\node[below right] at (2.5,2) {$F_{n_6}$};
\node[below right] at (3.5,2) {$F_{n_5}$};

\draw [line width=0.05mm] (5,1) -- (7,1) -- (7,4) -- (8,4) -- (8,2) -- (5,2) -- cycle;
\draw [line width=0.05mm] (6,1) -- (6,3) -- (8,3);
\node[below right] at (5,2) {$F_{n_5}$};
\node[below right] at (6,2) {$F_{n_6}$};
\node[below right] at (6,3) {$F_{n_3}$};
\node[below right] at (7,3) {$F_{n_2}$};
\node[below right] at (7,4) {$F_{n_1}$};

\node at (-4,.5) {$L^{*}_{1}$};
\node at (-1.3,.65) {$L^{*}_{2}$};
\node at (.5,-.5) {$L^{*}_{3}$};
\node at (3.5,.65) {$L^{*}_{4}$};
\node at (6,.65) {$L^{*}_{5}$};

\end{tikzpicture}
\caption{Invalid Landscapes}
  \label{Invalid Landscapes}
\end{figure}

\begin{proof}
First, we state the following claim.

\vspace{0.1in}

\begin{addmargin}[0.87cm]{0cm}
\noindent {\bf Claim.} No landscape in the set $\{L_i^*\}_{i=1}^5$ depicted in Figure \ref{Invalid Landscapes} is valid.

\vspace{0.1in}

Before we prove the claim, it is worth noting that one can show through exhaustion that every landscape of the cube with initial face $F_{n_1}$ and second face $F_{n_2}$ not in the set $\{L_i\}_{i=1}^{15}$ contains an element of $\{L_i^*\}_{i=1}^5$ as a sublandscape; and thus, once the claim is proven, it is clear that no such landscape is valid.  Furthermore, due to the symmetries of the cube any landscape with initial face $F_{n_1}$ can be mapped under rotation, and thus a reidentification of $\{n_2,n_3,n_4,n_5\} \rightarrow \{2,3,4,5\}$, to a landscape with initial face $F_{n_1}$ and second face $F_{n_2}$.  Finally, since path lengths on a cube are invariant under such reidentifications, it follows that any landscape with initial face $F_{n_1}$ not in the set $\{L_i\}_{i=1}^{15}$ must be invalid.  With this in mind, the following case by case analysis of the family of landscapes $\{L_i^*\}_{i=1}^5$ completes the proof of the theorem.

\vspace{0.1in}

\noindent {\em Proof of Claim.}

\vspace{-0.1in}

\setcounter{case}{0}
\begin{case}
Landscape $L_1^*$
\end{case}

Let $p_1 \in {F_{n_1}}, p_2 \in {F_{n_4}}$ and let $O_1^*$ be some orientation of $L_1^*$. Now let $w = T_{O_1^*}(p_1,p_2)\cap \overline{\{n_1,n_2,n_4\}\{n_1,n_2,n_3\}}$ and $z = T_{O_1^*}(p_1,p_2)\cap  \overline{\{n_2,n_4,n_6\}\{n_4,n_5,n_6\}}$, and let $s$ be the length of the line segment $\overline{w\{n_1,n_2,n_4\}}$ and $t$ be the length of the line segment $\overline{z\{n_2,n_4,n_6\}}$.  This gives us three cases.

\begin{figure}[h]
  \centering
\begin{tikzpicture}[scale=1]
\draw [line width=0.05mm] (-.5,0) rectangle (3.5,2);
\draw [line width=0.05mm] (-.5,0) rectangle (1.5,6);
\draw [line width=0.05mm] (-.5,4) -- (1.5,4);
\filldraw[fill=black,draw=black] (-.5,0) circle (2pt);
\filldraw[fill=black,draw=black] (-.5,2) circle (2pt);
\filldraw[fill=black,draw=black] (-.5,4) circle (2pt);
\filldraw[fill=black,draw=black] (-.5,6) circle (2pt);
\filldraw[fill=black,draw=black] (1.5,0) circle (2pt);
\filldraw[fill=black,draw=black] (1.5,2) circle (2pt);
\filldraw[fill=black,draw=black] (1.5,4) circle (2pt);
\filldraw[fill=black,draw=black] (1.5,6) circle (2pt);
\filldraw[fill=black,draw=black] (3.5,0) circle (2pt);
\filldraw[fill=black,draw=black] (3.5,2) circle (2pt);
\node[below right, inner sep=.5][font = {\small}] at (3.5,2) {$\{n_{1},n_{2},n_{4}\}$};
\node[left][font = {\small}] at (-.5,0) {$\{n_{3},n_{5},n_{6}\}$};
\node[left][font = {\small}] at (-.5,2) {$\{n_{2},n_{3},n_{6}\}$};
\node[left][font = {\small}] at (-.5,4) {$\{n_{1},n_{2},n_{3}\}$};
\node[above][font = {\small}] at (-.5,6) {$\{n_{1},n_{3},n_{5}\}$};
\node[below][font = {\small}] at (1.5,0) {$\{n_{4},n_{5},n_{6}\}$};
\node[above right][font = {\small}] at (1.5,2) {$\{n_{2},n_{4},n_{6}\}$};
\node[right][font = {\small}] at (1.5,4) {$\{n_{1},n_{2},n_{4}\}$};
\node[above][font = {\small}] at (1.5,6) {$\{n_{1},n_{4},n_{5}\}$};
\node[below][font = {\small}] at (3.5,0) {$\{n_{1},n_{4},n_{5}\}$};
\node[above right, inner sep=8] at (-.5,4) {$F_{n_1}$};
\node[below right, inner sep=10] at (-.5,4) {$F_{n_2}$};
\node[below right, inner sep=10] at (-.5,2) {$F_{n_6}$};
\node[below left, inner sep=10] at (3.5,2) {$F_{n_4}$};
\filldraw[fill=black,draw=black] (.5,5) circle (2pt);
\filldraw[fill=black,draw=black] (1.9,.3) circle (2pt);
\node[left][font = {\small}] at (.5,5) {$p_{1}$};
\node[right][font = {\small}] at (1.9,.3) {$p_{2}$};
\draw [line width=0.05mm] (.5,5) -- (1.9,.3);
\filldraw[fill=black,draw=black] (.8,4) circle (2pt);
\filldraw[fill=black,draw=black] (1.5,1.6) circle (2pt);
\node[below left][font = {\small}] at (.915,4) {$w$};
\node[right][font = {\small}] at (1.5,1.6) {$z$};
\draw [line width=0.05mm] (.8,4) rectangle (1.5,4.3);
\draw [line width=0.05mm] (1.5,1.6) rectangle (1.25,2);
\node[above][font = {\small}] at (1.15,4.3) {$s$};
\node[left][font = {\small}] at (1.25,1.775) {$t$};

\draw [line width=0.05mm] (7.5,0) rectangle (9.5,4);
\draw [line width=0.05mm] (7.5,2) rectangle (11.5,4);
\filldraw[fill=black,draw=black] (7.5,0) circle (2pt);
\filldraw[fill=black,draw=black] (7.5,2) circle (2pt);
\filldraw[fill=black,draw=black] (7.5,4) circle (2pt);
\filldraw[fill=black,draw=black] (9.5,0) circle (2pt);
\filldraw[fill=black,draw=black] (9.5,2) circle (2pt);
\filldraw[fill=black,draw=black] (9.5,4) circle (2pt);
\filldraw[fill=black,draw=black] (11.5,2) circle (2pt);
\filldraw[fill=black,draw=black] (11.5,4) circle (2pt);
\node[below][font = {\small}] at (7.5,0) {$\{n_{4},n_{5},n_{6}\}$};
\node[above left, inner sep=1][font = {\small}] at (7.5,2) {$\{n_{2},n_{4},n_{6}\}$};
\node[above][font = {\small}] at (7.5,4) {$\{n_{2},n_{3},n_{6}\}$};
\node[below][font = {\small}] at (9.5,0) {$\{n_{1},n_{4},n_{5}\}$};
\node[below right][font = {\small}] at (9.4,2) {$\{n_{1},n_{2},n_{4}\}$};
\node[above][font = {\small}] at (9.5,4) {$\{n_{1},n_{2},n_{3}\}$};
\node[font = {\small}] at (12.2,1.715) {$\{n_{1},n_{4},n_{5}\}$};
\node[above][font = {\small}] at (11.5,4) {$\{n_{1},n_{3},n_{5}\}$};
\filldraw[fill=black,draw=black] (10.5,3) circle (2pt);
\filldraw[fill=black,draw=black] (7.9,.3) circle (2pt);
\node[right][font = {\small}] at (10.5,3) {$p_{1}$};
\node[right][font = {\small}] at (7.9,.3) {$p_{2}$};
\filldraw[fill=black,draw=black] (9.5,2.53) circle (2pt);
\filldraw[fill=black,draw=black] (7.5,1.6) circle (2pt);
\draw [line width=0.05mm] (10.5,3) -- (7.5,1.6) -- (7.9,.3);
\node[font = {\small}] at (7.7,1.5) {$z$};
\node[font = {\small}] at (9.65,2.75) {$w$};
\draw [line width=0.05mm] (7.5,1.6) rectangle (7.3,2);
\draw [line width=0.05mm] (9.5,2.53) rectangle (9.7,2);
\node[right, inner sep=1][font = {\small}] at (9.7,2.25) {$s$};
\node[left, inner sep=.5][font = {\small}] at (7.25,1.75) {$t$};
\node[below right, inner sep=10] at (7.5,4) {$F_{n_2}$};
\node[below right, inner sep=10] at (9.5,4) {$F_{n_1}$};
\node[above left, inner sep=10] at (9.5,0) {$F_{n_4}$};
\end{tikzpicture}
    \caption{Trail on $L_1^*$ when $t>0$}
\end{figure}

\begin{addmargin}[0.87cm]{0cm}
    {\bf Subcase 1a.} {\em $s<1$ and $t>0$}

\vspace{0.1 in}
 
   Note: if we can find a landscape $\hat{L}$ for which $\abs{T_{\hat{L}}(w,z)} < \abs{T_{L_1^*}(w,z)}$ then, by Lemma \ref{boundary}, there exists a landscape $L$ for which $\abs{T_{L}(p_1,p_2)} < \abs{T_{L_1^*}(p_1,p_2)}$.  In particular, if we let $\hat{L}$ be the landscape $L_2(F_{n_1} \rightarrow F_{n_4})$, then we have that $\abs{T_{L_1^*}(w,z)}=\sqrt{s^2+1+2t+t^2}$ and $\abs{T_{\hat{L}}(w,z)} = \sqrt{s^2+1+2st+t^2}$.  Since $s<1$ and $t>0$, it follows that $2st < 2t$, and so $\abs{T_{\hat{L}}(w,z)} < \abs{T_{L_1^*}(w,z)}$.  Thus $T_{O_1^*}(p_1,p_2)$ does not witness the validity of $L_1^*$.
   \end{addmargin}
   
  \vspace{0.1in} 
   
\begin{addmargin}[0.87cm]{0cm}
    {\bf Subcase 1b.} { $s=1$}
    \vspace{0.1 in}
    
Since $s=1$, we have that $w=\{n_1,n_2,n_3\}$. Since $T_{O_1^*}(p_1,p_2)$ is a line segment with $p_1 \in {F_{n_1}}$ and $p_2 \in {F_{n_4}}$, it follows that $p_1=w$. Thus, $T_{O_1^*}(p_1,p_2)$ is completely contained in $L_2(F_{n_2}\rightarrow F_{n_4})$, a proper sublandscape of $L_1^*$, and so $T_{O_1^*}(p_1,p_2)$ does not witness the validity of $L_1^*$.
\end{addmargin}

\vspace{0.1in}

\begin{figure}[h]
  \centering
\begin{tikzpicture}[scale=1]
\draw [line width=0.05mm] (3,-7) rectangle (5,-3);
\draw [line width=0.05mm] (3,-5) rectangle (7,-3);
\filldraw[fill=black,draw=black] (3,-7) circle (2pt);
\filldraw[fill=black,draw=black] (3,-5) circle (2pt);
\filldraw[fill=black,draw=black] (3,-3) circle (2pt);
\filldraw[fill=black,draw=black] (5,-7) circle (2pt);
\filldraw[fill=black,draw=black] (5,-5) circle (2pt);
\filldraw[fill=black,draw=black] (5,-3) circle (2pt);
\filldraw[fill=black,draw=black] (7,-5) circle (2pt);
\filldraw[fill=black,draw=black] (7,-3) circle (2pt);
\node[left][font = {\small}] at (3,-5) {$\{n_{4},n_{5},n_{6}\} = z$};
\node[below right, inner sep=10] at (3,-3) {$F_{n_2}$};
\node[below right, inner sep=10] at (5,-3) {$F_{n_1}$};
\node[below left, inner sep=10] at (5,-5) {$F_{n_4}$};
\filldraw[fill=black,draw=black] (6.3,-3.6) circle (2pt);
\filldraw[fill=black,draw=black] (3.7,-6.6) circle (2pt);
\filldraw[fill=black,draw=black] (3.7,-5) circle (2pt);
\node[below][font = {\small}] at (6.3,-3.6) {$p_{1}$};
\node[right][font = {\small}] at (3.7,-6.6) {$p_{2}$};
\node[below right][font = {\small}] at (3.7,-5) {$v$};
\draw [line width=0.05mm] (6.3,-3.6) -- (3,-5) -- (3.7,-6.6);
\draw[line width=0.05mm][dash pattern={on 3pt}] (6.3,-3.6) -- (3.7,-5) -- (3.7,-6.6);
\node[below][font = {\small}] at (3,-7) {$\{n_{4},n_{5},n_{6}\}$};
\node[above][font = {\small}] at (3,-3) {$\{n_{2},n_{3},n_{6}\}$};
\node[below][font = {\small}] at (5,-7) {$\{n_{1},n_{4},n_{5}\}$};
\node[below right][font = {\small}] at (4.9,-5) {$\{n_{1},n_{2},n_{4}\}$};
\node[above][font = {\small}] at (5,-3) {$\{n_{1},n_{2},n_{3}\}$};
\node[font = {\small}] at (7.7,-5.285) {$\{n_{1},n_{4},n_{5}\}$};
\node[above][font = {\small}] at (7,-3) {$\{n_{1},n_{3},n_{5}\}$};
\end{tikzpicture}
    \caption{Trail on $L_1^*$ for the special case when $t=0$}
\end{figure}

\begin{addmargin}[0.87cm]{0cm}
    {\bf Subcase 1c.} { $t=0$}
    \vspace{0.1 in}

Again let $\hat{L}$ be the landscape $L_2(F_{n_1} \rightarrow F_{n_4})$ and let $\hat{O}$ be an orientation of $\hat{L}$.  Since $t=0$, we have that $z=\{n_2,n_4,n_6\}$.  In turn, both $T_{\hat{O}}(p_1,z)$ and $T_{O_1^*}(p_1,z)$ lie completely in $F_{n_1} \cup F_{n_2}$.  Similarly, both $T_{\hat{O}}(p_2,z)$ and $T_{O_1^*}(p_2,z)$ lie completely in $F_{n_4}$, and so 
\[\abs{T_{\hat{L}}(p_1,z)}=\abs{T_{L_1^*}(p_1,z)} \text { and } \abs{T_{\hat{L}}(p_2,z)}=\abs{T_{L_1^*}(p_2,z)} .\]
As a result, we have that
\[\abs{T_{L_1^*}(p_1,p_2)}=\abs{T_{\hat{L}}(p_1,z)}+\abs{T_{\hat{L}}(p_2,z)}.\]
Now let $v$ be the point on the line segment $\overline {z\{n_1,n_2,n_4\}}$ such that $\overline{vp_2}\perp\overline{z\{n_1,n_2,n_4\}}$.  Since $\Delta zvp_1$ is an obtuse triangle with $\angle v$ obtuse and $\Delta zvp_2$ is a right triangle with $\angle v$ right, it follows that
\[\abs{T_{\hat{L}}(p_1,v)} < \abs{T_{\hat{L}}(p_1,z)} \text{ and } \abs{T_{\hat{L}}(p_2,v)} < \abs{T_{\hat{L}}(p_2,z)}.\]
Thus, we have that
\[\abs{T_{\hat{L}}(p_1,v)}+\abs{T_{\hat{L}}(p_2,v)}<\abs{T_{L_1^*}(p_1,p_2)}.\]
It follows that the concatenation $T_{\hat{O}}(p_1,v) ^\frown T_{\hat{O}}(p_2,v)$ is a path on the surface of $\mathcal P_6$ between $p_1$ and $p_2$ whose length is shorter than $T_{O_1^*}(p_1,p_2)$, and thus by Theorem \ref{fullgenerality} there must exist an orientation $O$ such that $T_O(p_1,p_2)$ is shorter than $T_{O_1^*}(p_1,p_2)$.  Due to this, $T_{O_1^*}(p_1,p_2)$ does not witness the validity of $L_1^*$.
\end{addmargin}

\vspace{0.1in}

\noindent Since in each of the three cases $T_{O_1^*}(p_1,p_2)$ does not witness the validity of $L_1^*$, and $p_1$ and $p_2$ were arbitrarily chosen it follows that $L_1^*$ can not be a valid landscape.

\begin{case}
Landscape $L_2^*$
\end{case}

Due to the symmetries of $\mathcal P_6$, Case 2 is, up to rigid transformation and relabeling, identical to Case 1 and thus an immediate result of the proof of Case 1.

\begin{case}
Landscape $L_3^*$
\end{case}

Let $p_1 \in {F_{n_1}}, p_2 \in {F_{n_5}}$ and let $O_3^*$ be some orientation of $L_3^*$.  Now let $w = T_{O_3^*}(p_1,p_2)\cap \overline{\{n_1,n_2,n_4\}\{n_1,n_2,n_3\}}$ and $z = T_{O_3^*}(p_1,p_2)\cap \overline{\{n_3,n_5,n_6\}\{n_4,n_5,n_6\}}$. Then we have two cases.

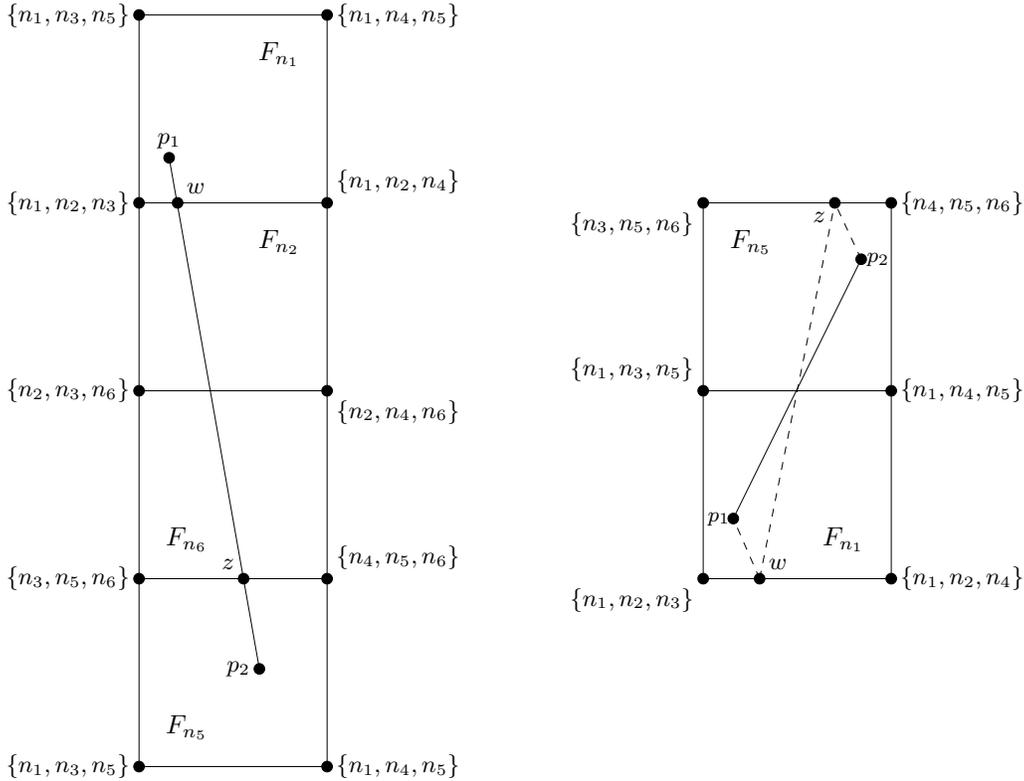
\begin{figure}[h]
  \centering
\begin{tikzpicture}[scale=1]
\draw [line width=0.05mm] (-.625,0) rectangle (1.875,10);
\filldraw[fill=black,draw=black] (-.625,0) circle (2pt);
\filldraw[fill=black,draw=black] (-.625,2.5) circle (2pt);
\filldraw[fill=black,draw=black] (-.625,5) circle (2pt);
\filldraw[fill=black,draw=black] (-.625,7.5) circle (2pt);
\filldraw[fill=black,draw=black] (-.625,10) circle (2pt);
\filldraw[fill=black,draw=black] (1.875,0) circle (2pt);
\filldraw[fill=black,draw=black] (1.875,2.5) circle (2pt);
\filldraw[fill=black,draw=black] (1.875,5) circle (2pt);
\filldraw[fill=black,draw=black] (1.875,7.5) circle (2pt);
\filldraw[fill=black,draw=black] (1.875,10) circle (2pt);
\draw [line width=0.05mm] (-.625,2.5) -- (1.875,2.5);
\draw [line width=0.05mm] (-.625,5) -- (1.875,5);
\draw [line width=0.05mm] (-.625,7.5) -- (1.875,7.5);
\node[left][font = {\small}] at (-.625,0) {$\{n_{1},n_{3},n_{5}\}$};
\node[left][font = {\small}] at (-.625,2.5) {$\{n_{3},n_{5},n_{6}\}$};
\node[left][font = {\small}] at (-.625,5) {$\{n_{2},n_{3},n_{6}\}$};
\node[left][font = {\small}] at (-.625,7.5) {$\{n_{1},n_{2},n_{3}\}$};
\node[left][font = {\small}] at (-.625,10) {$\{n_{1},n_{3},n_{5}\}$};
\node[right][font = {\small}] at (1.875,0) {$\{n_{1},n_{4},n_{5}\}$};
\node[above right][font = {\small}] at (1.875,2.5) {$\{n_{4},n_{5},n_{6}\}$};
\node[below right][font = {\small}] at (1.875,5) {$\{n_{2},n_{4},n_{6}\}$};
\node[above right][font = {\small}] at (1.875,7.5) {$\{n_{1},n_{2},n_{4}\}$};
\node[right][font = {\small}] at (1.875,10) {$\{n_{1},n_{4},n_{5}\}$};
\node[above right,inner sep=10] at (-.625,0) {$F_{n_5}$};
\node[above right, inner sep=10] at (-.625,2.5) {$F_{n_6}$};
\node[below left, inner sep=10] at (1.875,7.5) {$F_{n_2}$};
\node[below left, inner sep=10] at (1.875,10) {$F_{n_1}$};
\filldraw[fill=black,draw=black] (-.225,8.3) circle (2pt);
\filldraw[fill=black,draw=black] (1.475,1.75) circle (2pt);
\node[above][font = {\small}] at (-.225,8.3) {$p_{1}$};
\node[left][font = {\small}] at (1.475,1.75) {$p_{2}$};
\draw [line width=0.05mm] (-.225,8.3) -- (1.475,1.75);
\filldraw[fill=black,draw=black] (-.02,7.5) circle (2pt);
\filldraw[fill=black,draw=black] (1.28,2.5) circle (2pt);
\node[above left][font = {\small}] at (1.28,2.5) {$z$};
\node[above right][font = {\small}] at (-.02,7.5) {$w$};

\draw [line width=0.05mm] (6.875,2.5) rectangle (9.375,7.5);
\draw [line width=0.05mm] (6.875,5) -- (9.375,5);
\filldraw[fill=black,draw=black] (6.875,2.5) circle (2pt);
\filldraw[fill=black,draw=black] (6.875,5) circle (2pt);
\filldraw[fill=black,draw=black] (6.875,7.5) circle (2pt);
\filldraw[fill=black,draw=black] (9.375,2.5) circle (2pt);
\filldraw[fill=black,draw=black] (9.375,5) circle (2pt);
\filldraw[fill=black,draw=black] (9.375,7.5) circle (2pt);
\filldraw[fill=black,draw=black] (8.77,7.5) circle (2pt);
\filldraw[fill=black,draw=black] (7.48,2.5) circle (2pt);
\filldraw[fill=black,draw=black] (7.275,3.3) circle (2pt);
\filldraw[fill=black,draw=black] (8.975,6.75) circle (2pt);
\node[left, inner sep=1][font = {\footnotesize}] at (7.275,3.3) {$p_{1}$};
\node[right, inner sep=2][font = {\footnotesize}] at (8.975,6.75) {$p_{2}$};
\draw [line width=0.05mm] (7.275,3.3) -- (8.975,6.75);
\node[below left][font = {\small}] at (8.77,7.5) {$z$};
\node[above right][font = {\small}] at (7.48,2.5) {$w$};
\draw[line width=0.05mm][dash pattern={on 3pt}] (8.975,6.75) -- (8.77,7.5);
\draw[line width=0.05mm][dash pattern={on 3pt}] (7.48,2.5) -- (7.275,3.3);
\draw[line width=0.05mm][dash pattern={on 3pt}] (7.48,2.5) -- (8.77,7.5);
\node[below right,inner sep=10] at (6.875,7.5) {$F_{n_5}$};
\node[above left, inner sep=10] at (9.375,2.5) {$F_{n_1}$};
\node[above left][font = {\small}] at (6.875,5) {$\{n_{1},n_{3},n_{5}\}$};
\node[below left][font = {\small}] at (6.875,2.5) {$\{n_{1},n_{2},n_{3}\}$};
\node[below left][font = {\small}] at (6.875,7.5) {$\{n_{3},n_{5},n_{6}\}$};
\node[right][font = {\small}] at (9.375,5) {$\{n_{1},n_{4},n_{5}\}$};
\node[right][font = {\small}] at (9.375,2.5) {$\{n_{1},n_{2},n_{4}\}$};
\node[right][font = {\small}] at (9.375,7.5) {$\{n_{4},n_{5},n_{6}\}$};
\end{tikzpicture}
    \caption{Trail on $L_3^*$}
  \label{L_3^*}
\end{figure}

\begin{addmargin}[0.87cm]{0cm}
    {\bf Subcase 3a.} {\em $w=p_1$ and $z=p_2$}
    
   \vspace{0.1 in}
    
  Since $w=p_1$ and $z=p_2$, it follows that $T_{O_3^*}(p_1,p_2)$ is completely contained in $L_1(F_{n_2}\rightarrow F_{n_6})$ and since $L_1(F_{n_2}\rightarrow F_{n_6})$ is a proper sublandscape of $L_3^*$, $T_{O_3^*}(p_1,p_2)$ does not witness the validity of $L_3^*$.
   \end{addmargin}
   
   \vspace{0.1in}
   
\begin{addmargin}[0.87cm]{0cm}
    {\bf Subcase 3b.} { Either $w \ne p_1$ or $z \ne p_2$}
    \vspace{0.1 in}
    
Let $p_1(L_3^*,F_{n_1}, F_{n_2})=(x_1,y_1)$, $p_2(L_3^*,F_{n_1},F_{n_2})=(x_2,y_2)$, $w(L_3^*,F_{n_1},F_{n_2})=(x_w,y_w)$, and $z(L_3^*,F_{n_1},F_{n_2})=(x_z,y_z)$. Also let $L=L_1(F_{n_5} \rightarrow F_{n_1})$, and let $x_1^*$, $x_2^*$, $y_1^*$, and $y_2^*$ be such that $p_1(L,F_{n_5},F_{n_1})=(x_1^*,y_1^*)$ and $p_2(L,F_{n_5},F_{n_1})=(x_2^*,y_2^*)$. Note, that $\abs{x_1-x_2}=\abs{x_1^*-x_2^*}$, and also note that 
\begin{equation}\label{x^*,y^*}
    \begin{split}
        \abs{y_1-y_2}=&2+\abs{y_2-y_z}+\abs{y_1-y_w} \hspace{0.1in} \text{ and} \\
        \abs{y_1^*-y_2^*}=&2-\abs{y_2-y_z}-\abs{y_1-y_w}.
    \end{split}
\end{equation}
Calculating the length of the two relevant trails, we see that 
\begin{equation}\label{T^*}
    \begin{split}
        \abs{T_{L_3^*}(p_1,p_2)}=& \sqrt{\abs{x_1-x_2}^2+\abs{y_1-y_2}^2}
        \hspace{0.1in} \text{ and} \\
        \abs{T_L(p_1,p_2)}=&\sqrt{\abs{x_1^*-x_2^*}^2+\abs{y_1^*-y_2^*}^2}
    \end{split}
\end{equation}
Since either $w \ne p_1$ or $z \ne p_2$, it follows that either $\abs{y_2-y_z}>0$ or $\abs{y_1-y_w}>0$. Due to this, from equation \ref{x^*,y^*} we have that $\abs{y_1-y_2}>\abs{y_1^*-y_2^*}$. Finally, from equation \ref{T^*} we have that $\abs{T_{L_3^*}(p_1,p_2)}> \abs{T_L(p_1,p_2)}$, and thus $T_{O_3^*}(p_1,p_2)$ does not witness the validity of $L_3^*$.

\end{addmargin}

\vspace{0.1in}

\noindent Again, in either case, since $p_1$ and $p_2$ were chosen arbitrarily it follows that $L_3^*$ can not be a valid landscape.

\begin{case} Landscape $L_4^*$
\end{case}

Let $p_1 \in {F_{n_1}}, p_2 \in {F_{n_5}}$ and let $O_4^*$ be some orientation of $L_4^*$.  Now let $w = T_{O_4^*}(p_1,p_2)\cap  \overline{\{n_1,n_2,n_4\}\{n_1,n_2,n_3\}}$ and $z = T_{O_4^*}(p_1,p_2)\cap \overline{\{n_4,n_5,n_6\}\{n_3,n_5,n_6\}}$. We will have two cases.

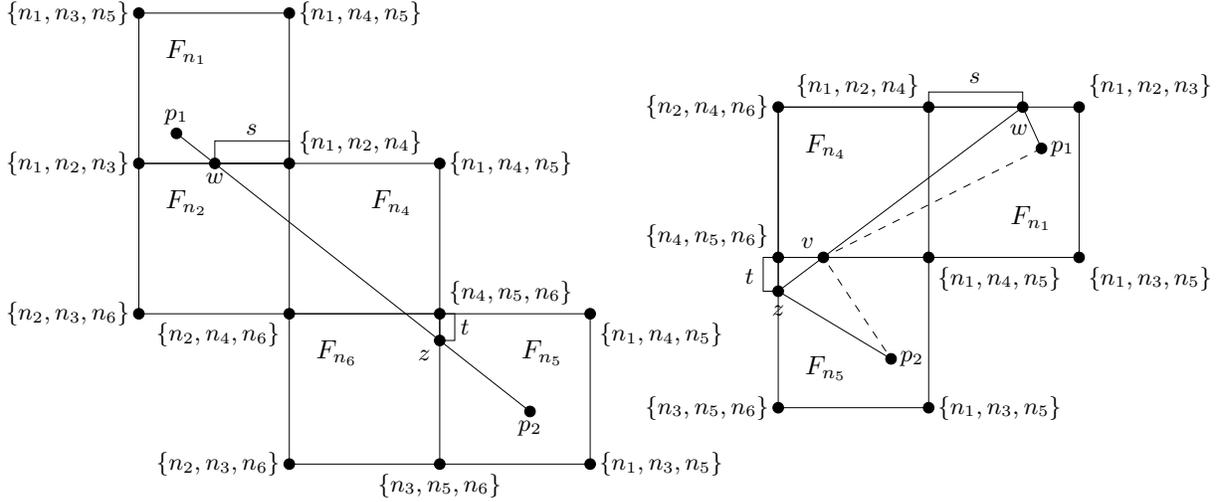
\begin{figure}[h]
  \centering
\begin{tikzpicture}[scale=1]
\draw [line width=0.05mm] (0,0) rectangle (4,2);
\draw [line width=0.05mm] (-2,2) rectangle (2,4);
\draw [line width=0.05mm] (-2,4) rectangle (0,6);
\draw [line width=0.05mm] (0,4) -- (0,2);
\draw [line width=0.05mm] (2,2) -- (2,0);
\filldraw[fill=black,draw=black] (-2,2) circle (2pt);
\filldraw[fill=black,draw=black] (-2,4) circle (2pt);
\filldraw[fill=black,draw=black] (-2,6) circle (2pt);
\filldraw[fill=black,draw=black] (0,0) circle (2pt);
\filldraw[fill=black,draw=black] (0,2) circle (2pt);
\filldraw[fill=black,draw=black] (0,4) circle (2pt);
\filldraw[fill=black,draw=black] (0,6) circle (2pt);
\filldraw[fill=black,draw=black] (2,0) circle (2pt);
\filldraw[fill=black,draw=black] (2,2) circle (2pt);
\filldraw[fill=black,draw=black] (2,4) circle (2pt);
\filldraw[fill=black,draw=black] (4,0) circle (2pt);
\filldraw[fill=black,draw=black] (4,2) circle (2pt);
\node[left][font = {\small}] at (-2,6) {$\{n_{1},n_{3},n_{5}\}$};
\node[left][font = {\small}] at (-2,4) {$\{n_{1},n_{2},n_{3}\}$};
\node[left][font = {\small}] at (-2,2) {$\{n_{2},n_{3},n_{6}\}$};
\node[right][font = {\small}] at (0,6) {$\{n_{1},n_{4},n_{5}\}$};
\node[above right][font = {\small}] at (0,4) {$\{n_{1},n_{2},n_{4}\}$};
\node[below left][font = {\small}] at (0,2) {$\{n_{2},n_{4},n_{6}\}$};
\node[left][font = {\small}] at (0,0) {$\{n_{2},n_{3},n_{6}\}$};
\node[below][font = {\small}] at (2,0) {$\{n_{3},n_{5},n_{6}\}$};
\node[above right][font = {\small}] at (2,2) {$\{n_{4},n_{5},n_{6}\}$};
\node[right][font = {\small}] at (2,4) {$\{n_{1},n_{4},n_{5}\}$};
\node[right][font = {\small}] at (4,0) {$\{n_{1},n_{3},n_{5}\}$};
\node[below right][font = {\small}] at (4,2) {$\{n_{1},n_{4},n_{5}\}$};
\node[below right, inner sep=10] at (-2,6) {$F_{n_1}$};
\node[below right, inner sep=10] at (-2,4) {$F_{n_2}$};
\node[below left, inner sep=10] at (2,4) {$F_{n_4}$};
\node[below right, inner sep=10] at (0,2) {$F_{n_6}$};
\node[below left, inner sep=10] at (4,2) {$F_{n_5}$};
\filldraw[fill=black,draw=black] (3.2,.7) circle (2pt);
\filldraw[fill=black,draw=black] (-1.5,4.4) circle (2pt);
\node[above][font = {\small}] at (-1.5,4.4) {$p_{1}$};
\node[below][font = {\small}] at (3.2,.7) {$p_{2}$};
\draw [line width=0.05mm] (3.2,.7) -- (-1.5,4.4);
\filldraw[fill=black,draw=black] (-.99,4) circle (2pt);
\filldraw[fill=black,draw=black] (2,1.644531915) circle (2pt);
\node[below][font = {\small}] at (-.99,4) {$w$};
\node[below left][font = {\small}] at (2,1.644531915) {$z$};
\draw [line width=0.05mm] (-.99,4) rectangle (0,4.3);
\draw [line width=0.05mm] (2,1.644531915) rectangle (2.2,2);
\node[right,inner sep=2][font = {\small}] at (2.2,1.82) {$t$};
\node[above,inner sep=2][font = {\small}] at (-.5,4.3) {$s$};

\draw [line width=0.05mm] (6.5,.75) rectangle (8.5,4.75);
\draw [line width=0.05mm] (6.5,2.75) rectangle (10.5,4.75);
\filldraw[fill=black,draw=black] (6.5,.75) circle (2pt);
\filldraw[fill=black,draw=black] (6.5,2.75) circle (2pt);
\filldraw[fill=black,draw=black] (6.5,4.75) circle (2pt);
\filldraw[fill=black,draw=black] (8.5,.75) circle (2pt);
\filldraw[fill=black,draw=black] (8.5,2.75) circle (2pt);
\filldraw[fill=black,draw=black] (8.5,4.75) circle (2pt);
\filldraw[fill=black,draw=black] (10.5,2.75) circle (2pt);
\filldraw[fill=black,draw=black] (10.5,4.75) circle (2pt);
\filldraw[fill=black,draw=black] (10,4.35) circle (2pt);
\filldraw[fill=black,draw=black] (7.7,1.45) circle (2pt);
\node[right][font = {\small}] at (10,4.35) {$p_{1}$};
\node[right][font = {\small}] at (7.7,1.45) {$p_{2}$};
\draw[line width=0.05mm][dash pattern={on 3pt}] (10,4.35) -- (6.95,2.75) -- (7.7,1.45);
\filldraw[fill=black,draw=black] (6.95,2.75) circle (2pt);
\filldraw[fill=black,draw=black] (6.5,2.4) circle (2pt);
\filldraw[fill=black,draw=black] (9.5,4.75) circle (2pt);
\draw [line width=0.05mm] (6.5,2.4) -- (9.5,4.75);
\node[above left][font = {\small}] at (7.4,2.4) {$v$};
\node[font = {\small}] at (6.6,2.15) {$z$};
\node[below][font = {\small}] at (9.55,4.65) {$w$};
\draw [line width=0.05mm] (9.5,4.75) -- (10,4.35);
\draw [line width=0.05mm] (6.5,2.4) -- (7.7,1.45);
\draw [line width=0.05mm] (8.5,4.95) rectangle (9.5,4.75);
\draw [line width=0.05mm] (6.3,2.4) rectangle (6.5,2.75);
\node[left][font = {\small}] at (6.3,2.5) {$t$};
\node[above][font = {\small}] at (9,4.95) {$s$};
\node[above left, inner sep=10] at (10.5,2.75) {$F_{n_1}$};
\node[below right, inner sep=10] at (6.5,4.75) {$F_{n_4}$};
\node[above right, inner sep=10] at (6.5,.75) {$F_{n_5}$};
\node[left][font = {\small}] at (6.5,4.75) {$\{n_{2},n_{4},n_{6}\}$};
\node[above left][font = {\small}] at (6.5,2.75) {$\{n_{4},n_{5},n_{6}\}$};
\node[left][font = {\small}] at (6.5,.75) {$\{n_{3},n_{5},n_{6}\}$};
\node[below right][font = {\small}] at (8.5,2.75) {$\{n_{1},n_{4},n_{5}\}$};
\node[right][font = {\small}] at (8.5,.75) {$\{n_{1},n_{3},n_{5}\}$};
\node[above left][font = {\small}] at (8.5,4.75) {$\{n_{1},n_{2},n_{4}\}$};
\node[above right][font = {\small}] at (10.5,4.75) {$\{n_{1},n_{2},n_{3}\}$};
\node[below right][font = {\small}] at (10.5,2.75) {$\{n_{1},n_{3},n_{5}\}$};
\end{tikzpicture}
    \caption{Trail on $L_4^*$}
\end{figure}

\vspace{0.1in}

\begin{addmargin}[0.87cm]{0cm}
    {\bf Subcase 4a.} {\em $w=p_1$ and $z=p_2$}
    
   \vspace{0.1 in}
    
  Since $w=p_1$ and $z=p_2$, the trail is completely contained in $L_3(F_{n_2}\rightarrow F_{n_6})$ and since $L_3(F_{n_2}\rightarrow F_{n_6})$ is a proper sublandscape of $L_4^*$, $T_{O_4^*}(p_1,p_2)$ does not witness the validity of $L_4^*$.
   \end{addmargin}
   
 \vspace{0.1in}  
   
\begin{addmargin}[0.87cm]{0cm}
    {\bf Subcase 4b.} { Either $w \ne p_1$ or $z \ne p_2$}
    \vspace{0.1 in}
    
Let $\hat{L}=L_2(F_{n_1}\rightarrow F_{n_5})$ and $\hat{O}=(\hat{L},F_{n_1},F_{n_5})$. Now let $s$ be the length of the line segment $\overline{w\{n_1,n_2,n_4\}}$ and $t$ be the length of the line segment $\overline{z\{n_4,n_5,n_6\}}$. 

Since $\Delta \{n_2,n_4,n_6\}\{n_1,n_2,n_3\}\{n_3,n_5,n_6\}$ is convex and contained in $\hat{L}$, it follows that $T_{\hat{O}}(w,z)$ must be contained in $\hat{L}$ and thus intersect with $\overline{\{n_4,n_5,n_6\}\{n_1,n_4,n_5\}}$.  With this in mind, we can let $v=T_{\hat{O}}(w,z)\cap  \overline{\{n_4,n_5,n_6\}\{n_1,n_4,n_5\}}$. Then, we have that 
\[\abs{T_{L_4^*}(w,z)}=\sqrt{s^2+2+2s+2t+t^2}=\abs{T_{\hat{L}}(w,z)}.\] 
We also have that $\abs{T_{L_4^*}(p_1,w)}=\abs{T_{\hat{L}}(p_1,w)}$ and $\abs{T_{L_4^*}(p_2,z)}=\abs{T_{\hat{L}}(p_2,z)}$. Since either $w \ne p_1$ or $z \ne p_2$, it follows from the triangle inequality that either 
\[\abs{T_{\hat{L}}(p_1,v)}<\abs{T_{\hat{L}}(p_1,w)}+\abs{T_{\hat{L}}(w,v)} \hspace{0.25in} \text{or} \hspace{0.25in} \abs{T_{\hat{L}}(p_2,v)}<\abs{T_{\hat{L}}(p_2,z)}+\abs{T_{\hat{L}}(z,v)}.\]
Furthermore, since $\abs{T_{\hat{L}}(w,z)}=\abs{T_{\hat{L}}(w,v)}+\abs{T_{\hat{L}}(v,z)}$, it follows that 
\begin{equation*}
\begin{split}
    \abs{T_{\hat{L}}(p_1,v)}+\abs{T_{\hat{L}}(p_2,v)}&<\abs{T_{\hat{L}}(p_1,w)}+\abs{T_{\hat{L}}(w,v)}+\abs{T_{\hat{L}}(p_2,z)}+\abs{T_{\hat{L}}(z,v)}\\
    &=\abs{T_{L_4^*}(w,z)}+\abs{T_{L_4^*}(p_1,w)}+\abs{T_{L_4^*}(p_2,z)}.
\end{split}
\end{equation*}
It follows that the concatenation $T_{\hat{O}}(p_1,v) ^\frown T_{\hat{O}}(p_2,v)$ is a path on the surface of $\mathcal P_6$ between $p_1$ and $p_2$ whose length is shorter than $T_{O_4^*}(p_1,p_2)$, and thus by Theorem \ref{fullgenerality} there must exist an orientation $O$ such that $T_O(p_1,p_2)$ is shorter than $T_{O_4^*}(p_1,p_2)$.  Due to this, $T_{O_4^*}(p_1,p_2)$ does not witness the validity of $L_4^*$.
\end{addmargin}
 
\vspace{0.1in}

\noindent Once more, in either case, since $p_1$ and $p_2$ were chosen arbitrarily, $L_4^*$ can not be a valid landscape.

\begin{case}
Landscape $L_5^*$

\end{case}
Due to the symmetries of $\mathcal P_6$, Case 5 is, up to rigid transformation and relabeling, identical to Case 4 and thus an immediate result of the proof of Case 4. $\Box$

\vspace{0.1in}

\end{addmargin}

\noindent Having shown that no landscape in the family $\{L_i^*\}_{i=1}^5$ is valid, we have thus completed our proof. 
\end{proof}

With the previous theorem in mind, and the collection of landscapes which may be valid narrowed down substantially, we can now confirm that the landscapes $\{L_i\}_{i=1}^{15}$ are each valid simply by providing points which witness their validity. 

\begin{corollary}
The cube has $15$ valid landscapes $\{L_i\}_{i=1}^{15}$.
\end{corollary}

\begin{proof}
One can confirm via a short computation that the following table provides pairs of points which witness the validity of the landscapes $L_1, L_2, L_3, L_4, L_5, L_6, L_7, L_8, L_9, L_{10}, L_{11}, L_{12}, L_{13}, L_{14}$, and $L_{15}$.  Thus all fifteen of these landscapes are valid.  However, as shown by Theorem \ref{finaltheorem}, the cube has no other valid landscapes.

\begin{center}
\begin{tabular}{ |c|c|}
    \hline
    Landscape & Pair of Points Witnessing Validity\\
    \hline
    $L_1$ & $(F_{n_1}, F_{n_2}, 0.5, 0.1),(F_{n_2}, F_{n_1}, 0.5, 0.1)$\\
    \hline
    $L_2$ & $(F_{n_1}, F_{n_2}, 0.1, 0.9), (F_{n_2}, F_{n_1}, 0.9, 0.9)$ \\
    \hline
    $L_3$ & $(F_{n_1}, F_{n_2}, 0.9, 0.9), (F_{n_2}, F_{n_1}, 0.1, 0.9)$ \\
    \hline
    $L_4$ & $(F_{n_1}, F_{n_2}, 0.5, 0.1), (F_{n_6}, F_{n_2}, 0.5, 0.1)$ \\
    \hline
    $L_5$ & $(F_{n_1}, F_{n_2}, 0.1, 0.5), (F_{n_6}, F_{n_2}, 0.9, 0.5)$ \\
    \hline
    $L_6$ & $(F_{n_1}, F_{n_2}, 0.9, 0.5), (F_{n_6}, F_{n_2}, 0.1, 0.5)$ \\
    \hline
    $L_7$ & $(F_{n_1}, F_{n_2}, 0.5, 0.9), (F_{n_6}, F_{n_2}, 0.5, 0.9)$ \\
    \hline
    $L_8$ & $(F_{n_1}, F_{n_2}, 0.2, 0.1), (F_{n_6}, F_{n_2}, 0.9, 0.2)$ \\
    \hline
    $L_9$ & $(F_{n_1}, F_{n_2}, 0.1, 0.8), (F_{n_6}, F_{n_2}, 0.8, 0.9)$ \\
    \hline
    $L_{10}$ & $(F_{n_1}, F_{n_2}, 0.8, 0.9), (F_{n_6}, F_{n_2}, 0.1, 0.8)$ \\
    \hline
    $L_{11}$ & $(F_{n_1}, F_{n_2}, 0.9, 0.2), (F_{n_6}, F_{n_2}, 0.2, 0.1)$ \\
    \hline
    $L_{12}$ & $(F_{n_1}, F_{n_2}, 0.8, 0.1), (F_{n_6}, F_{n_2}, 0.1, 0.2)$ \\
    \hline
    $L_{13}$ & $(F_{n_1}, F_{n_2}, 0.1, 0.2), (F_{n_6}, F_{n_2}, 0.8, 0.1)$ \\
    \hline
    $L_{14}$ & $(F_{n_1}, F_{n_2}, 0.2, 0.9), (F_{n_6}, F_{n_2}, 0.9, 0.8)$ \\
    \hline
    $L_{15}$ & $(F_{n_1}, F_{n_2}, 0.9, 0.8), (F_{n_6}, F_{n_2}, 0.2, 0.9)$ \\
    \hline
\end{tabular}
\end{center}
\end{proof}

Having identified the collection of valid landscapes of $\mathcal P_6$, we provide the following corollary which identifies a way to determine the surface distance between any two points on a cube.

\begin{corollary}
Let $p_1,p_2$ be two distinct points on the cube.
\begin{itemize}
    \item If $p_1 \in F_n \setminus F_m$ and $p_2 \in F_m \setminus F_n$, with $n,m$ distinct and $n+m \neq 7$, then
    \[d_{\mathcal P_6}(p_1,p_2) = d_{\mathcal P_6}^A(p_1,p_2).\]
    \item If $p_1 \in int(F_n)$ and $p_2 \in int(F_m)$ with $n+m=7$, then \[d_{\mathcal P_6}(p_1,p_2) = d_{\mathcal P_6}^O(p_1,p_2).\]
\end{itemize}
\end{corollary}

Having identified the surface distance of points on the cube, and thus having developed formulae for the length of the shortest paths between two points on the surfaces of both tetrahedra and cubes we have at this point reached the end of the current discussion.  However, the concepts developed here can be applied to any convex unit polyhedron, and thus there are a great many problems in the area left to be explored. In particular, this includes the remaining platonic solids.  At the time of this writing, two of the authors are exploring the problem of determining the surface distance between points on octahedra.  It of course remains to be seen, and could be of some interest, if this concept can be applied to any convex polyhedron or perhaps certain classes of nonconvex polyhedra.

\section*{Acknowledgements}

The research of Emiko Saso and Xin Shi was supported by the Trinity College Summer Research Program.  We would also like to thank the Undergrad Research Incubator Program at the University of North Texas for its support.  

\bibliographystyle{abbrv}
\bibliography{landscapes}

\end{document}